\documentclass[reqno, 11pt]{article}
\usepackage[
a4paper,
margin = 2.54cm
]
{geometry}

\usepackage{amsmath}
\usepackage{amsfonts}
\usepackage{amsthm}
\usepackage{amssymb}
\usepackage{latexsym}
\usepackage{enumitem} 
\usepackage{mathrsfs}
\usepackage{mathtools}
\usepackage[hidelinks]{hyperref}
\usepackage{xcolor}

\def\A{\mathcal{A}}

\def\R{\mathbb{R}}

\def\A{\mathscr{A}}

\DeclareMathOperator{\Var}{Var}

\DeclareMathOperator{\Bernoulli}{Bernoulli}
\DeclareMathOperator{\Exp}{Exponential}

\DeclareMathOperator{\TV}{TV}
\DeclareMathOperator{\HC}{HC}

\DeclareMathOperator{\ess}{ess}

\newtheorem{proposition}{Proposition}
\newtheorem{lemma}{Lemma}
\newtheorem{theorem}{Theorem}
\newtheorem{corollary}{Corollary}

\theoremstyle{definition}
\newtheorem{definition}{Definition}
\newtheorem{assumption}{Assumption}

\newtheorem{remark}{Remark}

\title{Statistical Limits of Sparse Mixture Detection}
\author{Subhodh Kotekal \\ \\ \textit{University of Chicago}}
\date{}
\begin{document}
    \maketitle
    \begin{abstract}
        We consider the problem of detecting a general sparse mixture and obtain an explicit characterization of the phase transition under some conditions, generalizing the univariate results of Cai and Wu. Additionally, we provide a sufficient condition for the adaptive optimality of a Higher Criticism type testing statistic formulated by Gao and Ma. In the course of establishing these results, we offer a unified perspective through the large deviations theory. The phase transition and adaptive optimality we establish are direct consequences of the large deviation principle of the normalized log-likelihood ratios between the null and the signal distributions. 
    \end{abstract}

    \section{Introduction}
    Modern technological advancements have ushered in a new scientific regime in which researchers simultaneously take measurements of a very large number of units with only a small fraction of units potentially exhibiting a signal. Typical examples include microarrays in genomics \cite{efronMicroarraysEmpiricalBayes2008} and microwave probes in cosmology \cite{vielvaComprehensiveOverviewCold2010}; this new regime is ubiquitous in modern science \cite{efronLargeScaleInferenceEmpirical2010}. Moreover, in many applications the signal is believed to be not only sparse but also sufficiently weak such that consistent identification of signal exhibiting units is impossible. In such a regime, two statistical problems immediately come to mind. First is the detection problem: for which sparsity levels can the presence of a signal be consistently detected? Second is the adaptation problem: does there exist a test which can detect a detectable signal without knowledge of the signal sparsity?

    The detection problem is formally stated as a sparse mixture testing problem 
    \begin{align}\label{problem:sparse_mixture_detection_1}
        H_0^{(n)} &: X_1,...,X_n \overset{iid}{\sim} P_n, \\
        H_1^{(n)} &: X_1,...,X_n \overset{iid}{\sim} (1-\varepsilon) P_n + \varepsilon Q_n
        \label{problem:sparse_mixture_detection_2}
    \end{align}
    where \(\varepsilon \in (0, 1)\) and \(\{P_n\}\) and \(\{Q_n\}\) are collections of probability distributions on, say for convenience, a separable metric space \(\mathcal{X}\). A consistent sequence of tests for testing (\ref{problem:sparse_mixture_detection_1})-(\ref{problem:sparse_mixture_detection_2}) is a sequence of measurable functions \(\varphi_n : \mathcal{X}^n \to \{0, 1\}\) such that
    \begin{equation*}
        \lim_{n \to \infty} P_{H_0^{(n)}}(\varphi_n(X_1,...,X_n) = 1) + P_{H_1^{(n)}}(\varphi_n(X_1,...,X_n) = 0) = 0.
    \end{equation*}
    To model the sparsity of the signal, the sparse mixture detection literature has adopted the calibration
    \begin{equation}\label{eqn:beta_sparsity}
        \varepsilon = n^{-\beta}
    \end{equation}
    with \(0 < \beta < 1\). The detection problem is to characterize, for fixed \(\{P_n\}\) and \(\{Q_n\}\), the values \(\beta\) such that there exist a sequence of consistent tests for the testing problem (\ref{problem:sparse_mixture_detection_1})-(\ref{problem:sparse_mixture_detection_2}). For such \(\beta\), the collection of mixtures \(\{(1-n^{-\beta}) P_n + n^{-\beta}Q_n\}\) is said to be detectable. Note that by the Neyman-Pearson lemma, the likelihood ratio test is consistent whenever the collection of mixtures \(\{(1-n^{-\beta})P_n + n^{-\beta} Q_n\}\) is detectable. However, the likelihood ratio test is not a solution to the adaptation problem as it requires knowledge of \(\beta\). The adaptation problem remains of practical interest.
    
    Arguably the prototypical sparse mixture detection problem is the sparse normal mixture detection problem considered by Ingster \cite{ingsterProblemsHypothesisTesting1996} and Jin \cite{jinDetectingEstimatingSparse2003,jinDetectingTargetVery2004}. Specifically, this is the testing problem (\ref{problem:sparse_mixture_detection_1})-(\ref{problem:sparse_mixture_detection_2}) under calibration (\ref{eqn:beta_sparsity}) with \(P_n = N(0, 1)\) and \(Q_n = N(\mu_n, 1)\) where \(\mu_n = \sqrt{2r\log n}\) and \(r \in (0, 1)\). Ingster \cite{ingsterProblemsHypothesisTesting1996} and Jin \cite{jinDetectingEstimatingSparse2003,jinDetectingTargetVery2004} independently derived a subtle phase transition in this seemingly simple detection problem. A delicate asymptotic analysis showed that if \(\beta < \beta^*(r)\), then there exists a sequence of consistent tests to test (\ref{problem:sparse_mixture_detection_1})-(\ref{problem:sparse_mixture_detection_2}) where 
    \begin{equation*}
        \beta^*(r) = 
        \begin{cases}
            \frac{1}{2} + r &\text{if } r \leq \frac{1}{4}, \\
            1 - (1-\sqrt{r})_{+}^2 &\text{if } r > \frac{1}{4}.
        \end{cases}
    \end{equation*}
    Here, the notation \((x)_+ := \max\{x, 0\}\) for \(x \in \R\) is used. Additionally, Ingster and Jin independently showed that if \(\beta > \beta^*(r)\), then no sequence of tests is consistent for testing (\ref{problem:sparse_mixture_detection_1})-(\ref{problem:sparse_mixture_detection_2}). The existence of a subtle phase transition in an apparently simple detection problem sparked subsequent research interest in sparse mixture detection. Phase transitions have been discovered in a variety of other sparse mixture detection problems beyond the sparse normal mixture setting \cite{gaoFiveShadesGrey2019,arias-castroSparsePoissonMeans2015,donohoHigherCriticismDetecting2004,caiOptimalDetectionHeterogeneous2011,donohoTwosampleTestingLarge2020,gaoTestingEquivalenceClustering2019}. In investigating the asymptotic consequences of signal rarity and strength on various statistical tasks, a theoretical framework called the \emph{Asymptotic Rare/Weak} (ARW) model has been introduced \cite{jinRareWeakEffects2016,donohoHigherCriticismLargeScale2015}. Detection of sparse mixtures is only one instance of a statistical task in the ARW model; sparse signal recovery (i.e. identifying which \(X_i\) follow the signal distribution \(Q_n\)) is another. The primary focus in the ARW model is on determining the phase transition of signal rarity and strength which separate success and failure of the statistical task under consideration. The framework's introduction has been followed by an active research program arguably spearheaded by Jin and collaborators \cite{donohoHigherCriticismLargeScale2015,jinRareWeakEffects2016,hallInnovatedHigherCriticism2010,caiOptimalDetectionHeterogeneous2011,jinPhaseTransitionsHigh2017a,jinImpossibilitySuccessfulClassification2009,jinOptimalityGraphletScreening2014}. We refer the reader to the review articles \cite{donohoHigherCriticismLargeScale2015,jinRareWeakEffects2016} for a detailed treatment.
    
    In the context of sparse mixture detection, the phase transitions were initially obtained through delicate asymptotic analyses of the likelihood ratio test. Later, a unified approach to deriving phase transitions for general sparse mixtures on \(\R\) was put forth by Cai and Wu \cite{caiOptimalDetectionHeterogeneous2011}. Cai and Wu characterized the exact asymptotic order of the Hellinger distance between \(P_n\) and \((1-n^{-\beta}) P_n + n^{-\beta} Q_n\) in terms of \(\beta\), and it turns out the exact asymptotic order fundamentally determines the phase transition (under some regularity conditions). Many of the phase transitions in the literature follow directly from the results of Cai and Wu (see Section V of \cite{caiOptimalDetectionSparse2014}). Ditzhaus \cite{ditzhausSignalDetectionPhidivergences2019} extended the results of Cai and Wu \cite{caiOptimalDetectionSparse2014} to a larger class of univariate sparse mixtures beyond those satisfying the regularity conditions of \cite{caiOptimalDetectionSparse2014}. 
    
    In the adaptation problem, Donoho and Jin \cite{donohoHigherCriticismDetecting2004} delivered a key construction when investigating the sparse normal mixture detection problem. Donoho and Jin formulated a sequence of tests based on Tukey's Higher Criticism statistic that is consistent whenever \(\beta < \beta^*(r)\) and adapts to not only the sparsity level \(\beta\) but also to the signal strength \(r\). Specifically, Donoho and Jin considered the sequence of tests 
    \begin{equation}\label{eqn:donoho_jin_HC_test}
        \psi_{\HC_n} := \mathbf{1}_{\left\{ \HC_n > \sqrt{2(1+\delta) \log \log n} \right\}}
    \end{equation}
    where \(\delta > 0\) is an arbitrary constant and the Higher Criticism statistic is defined as
    \begin{equation}\label{eqn:donoho_jin_HC}
        \HC_n := \sup_{t \in \R} \frac{\left| \sum_{i=1}^{n} \mathbf{1}_{\{X_i \leq t\}} - n\Phi(t)\right|}{\sqrt{n \Phi(t)(1-\Phi(t))}}.
    \end{equation}
    Here, \(\Phi\) is the cumulative distribution function of the standard normal distribution. Calculating the \(p\)-value \(p_i = 1-\Phi(X_i)\), a change of variable yields a more evocative form
    \begin{equation*}
        \HC_n = \sup_{u \in (0, 1)} \frac{\left|\sum_{i=1}^{n} \mathbf{1}_{\{p_i \leq u\}} - n u\right|}{\sqrt{n u(1-u)}}.
    \end{equation*}
    With the formulation in terms of \(p\)-values in mind, Higher Criticism is attractive in that it can be widely applied to sparse mixture detection beyond the initial sparse normal mixture setting. One need only craft \(p\)-values \(\{p_i\}_{1 \leq i \leq n}\) from the observations \(\{X_i\}_{1 \leq i \leq n}\) to use in \(\HC_n\). In some detection problems there is a natural construction of \(p\)-values, whereas it is a challenge in other problems. We refer the reader to the review articles \cite{donohoHigherCriticismLargeScale2015,jinRareWeakEffects2016} for a detailed discussion of Higher Criticism and its applications beyond the sparse normal mixture detection problem originally in mind. 
    
    Remarkably, Higher Criticism achieves the detection boundary in many other sparse mixture detection problems beyond the sparse normal mixture setting. For example, the optimality of Higher Criticism for signal detection in the heteroscedastic sparse normal mixture with \(P_n = N(0, 1)\) and \(Q_n = N(\sqrt{2r\log n}, \sigma^2)\) was established by Cai, Jeng, and Jin \cite{caiOptimalDetectionHeterogeneous2011}. In the case of Gaussian null \(P_n = N(0, 1)\) and general \(Q_n\) (under some conditions), Cai and Wu proved that Higher Criticism is optimal. Later, Ditzhaus \cite{ditzhausSignalDetectionPhidivergences2019} proved that Higher Criticism is optimal for general distributions \(P_n\) and \(Q_n\) on \(\R\) (again, under some conditions). 
    
    As Donoho and Jin \cite{donohoHigherCriticismDetecting2004} note, Higher Criticism has a goodness-of-fit interpretation in which the empirical distribution of the \(p\)-values is compared to the uniform distribution i.e. the null distribution. Leveraging this goodness-of-fit interpretation, Jager and Wellner \cite{jagerGoodnessoffitTestsPhidivergences2007} formulate a class of testing statistics based on \(\varphi\)-divergences which are adaptive to \(\beta\); Higher Criticism is included as a special case. Jager and Wellner showed, in the sparse normal mixture detection problem, their class of testing statistics furnish a sequence of tests which are consistent whenever the collection of mixtures is detectable. Ditzhaus \cite{ditzhausSignalDetectionPhidivergences2019} proved the optimality of Jager and Wellner's tests in the setting of general sparse mixtures on \(\R\) (under some technical conditions).
    
    The existing literature has largely focused on the setting where the probability distributions \(P_n\) and \(Q_n\) are on \(\R\). However, the data measured from each unit in many modern applications is usually multivariate or structured in some manner (e.g. graph, partition). To apply the results of Cai and Wu \cite{caiOptimalDetectionSparse2014}, one would need to reduce the multivariate or structured measurement from each unit into a univariate summary statistic and deduce a phase transition from the resulting univariate sparse mixture. However, it is not clear that a safe reduction to a univariate summary statistic can be done without loss of inferentially-relevant information. The phase transition deduced from the univariate statistic may not be the true phase transition for the original sparse mixture testing problem. Of course, an alternative approach is to directly examine the asymptotics of the likelihood ratio test to derive a phase transition, but the analysis must be tailored on a problem-by-problem basis. Cai and Wu's results are attractive precisely because they offer a unified perspective on deriving phase transitions in the univariate setting. It is of theoretical interest to obtain analogous perspective in the multivariate case. 
    
    In the multivariate sparse mixture setting, it's unclear that Higher Criticism can be adapted to furnish adaptively optimal tests. Higher Criticism requires a reduction to a univariate summary statistic; the typical reduction to keep in mind is the \(p\)-value. Indeed, in their review article, Donoho and Jin \cite{donohoHigherCriticismLargeScale2015} carefully craft relevant \(p\)-values to apply Higher Criticism to multidimensional sparse mixtures testing problems. As a practical matter, it need not be clear how the univariate summary statistic should be constructed; often it is done in an ad hoc manner. Furthermore, it may not even be the case \emph{a priori} that there exists a univariate reduction such that the corresponding Higher Criticism test is optimal. It is of practical interest to find an optimal test which adapts to the signal sparsity without requiring ad hoc constructions. 

    The goal of the present paper is twofold. Firstly, we seek to generalize the results of Cai and Wu \cite{caiOptimalDetectionSparse2014} by establishing phase transitions for the detection of general sparse mixtures beyond the univariate case. The proofs of \cite{caiOptimalDetectionSparse2014} rely on examining the asymptotic behavior of the likelihood ratio \(\frac{dQ_n}{dP_n}\) evaluated at suitable quantiles of \(P_n\); the fact that \(P_n\) and \(Q_n\) are distributions on \(\R\) is put to good use by virtue of examining quantiles. We provide a broader perspective beyond the univariate case through the theory of large deviations. The core idea of Cai and Wu \cite{caiOptimalDetectionSparse2014} in characterizing the sharp Hellinger asymptotics is crucial to our analysis; the large deviations theory gives suitable tools to treat the general case with Cai and Wu's idea in hand. Secondly, we give a sufficient condition ensuring the optimality of a Higher Criticism type testing statistic proposed by Gao and Ma (Section 3.2 of \cite{gaoTestingEquivalenceClustering2019}). Specifically, Gao and Ma's statistic is obtained by applying Higher Criticism to the empirical likelihood ratios \(\left\{ \frac{dQ_n}{dP_n}(X_i)\right\}_{i=1}^{n}\). While generically requiring knowledge of both \(P_n\) and \(Q_n\), the statistic is adaptive to the signal sparsity and can be broadly applied without the need for ad hoc constructions.

    \subsection*{Organization}
    The remainder of the paper is organized as follows. Section \ref{section:detection} reviews the connection between Hellinger asymptotics and phase transitions established by Cai and Wu \cite{caiOptimalDetectionSparse2014}, states some requisite definitions from large deviations theory, and presents our main result characterizing the phase transition in general sparse mixture testing problems (under some technical conditions). Section \ref{section:HC} reviews a Higher Criticism type testing statistic proposed by Gao and Ma (Section 3.2 of \cite{gaoTestingEquivalenceClustering2019}) and presents a sufficient condition under which this statistic furnishes optimal tests that are adaptive to the signal sparsity. Section \ref{section:examples} illustrates our results and typical methods of calculation through some examples. Section \ref{section:discussion} discusses our results and a few directions of further work. Proofs not presented in the main text are found in Section \ref{section:proofs}. 

    \subsection*{Notation}
    We use the following notation throughout the paper. For \(a, b \in [-\infty, \infty]\), denote \(a \vee b = \max\{a, b\}\) and \(a \wedge b = \min\{a, b\}\). For sequences \(\{a_n\}, \{b_n\} \subset [-\infty, \infty]\), denote \(a_n = o(b_n)\) if \(\frac{a_n}{b_n} \to 0\) as \(n \to \infty\). Further, denote \(a_n = \omega(b_n)\) if \(b_n = o(a_n)\). For positive sequences \(\{a_n\}, \{b_n\}\), denote \(a_n \lesssim b_n\) if there exists a constant \(C > 0\) not depending on \(n\) such that \(a_n \leq C b_n\). Denote \(a_n \gtrsim b_n\) if \(b_n \lesssim a_n\), and denote \(a_n \asymp b_n\) if \(a_n \lesssim b_n\) and \(a_n \gtrsim b_n\). For \(a \in [-\infty, \infty]\), denote \((a)_{+} = \max\{a, 0\}\). Denote \(\R_{+} = [0, \infty)\). For a probability measure \(P\), let \(P^n\) denote the \(n\)-fold product measure of \(P\). A probability measure \(Q\) on a measurable space is said to be absolutely continuous with respect to a probability measure \(P\) on the same measurable space if \(P(A) = 0\) implies \(Q(A) = 0\) for every measurable set \(A\). We denote this as \(Q \ll P\). The total variation distance between \(P\) and \(Q\) is given by \(\TV(P, Q) = \sup_{A} |P(A) - Q(A)|\). The Hellinger distance between \(P\) and \(Q\) is given by \(H(P, Q) = \left( \int \left(\sqrt{dP/d\nu} - \sqrt{dQ/d\nu}\right)^2 \right)^{1/2}\) where \(\nu\) is a measure such that \(P, Q \ll \nu\). 

    \section{A large deviations perspective on detection limits}\label{section:detection}
    To determine the fundamental detection limits, we follow the approach laid out by Cai and Wu \cite{caiOptimalDetectionSparse2014} in characterizing the Hellinger asymptotics. Cai and Wu \cite{caiOptimalDetectionSparse2014} only obtain results in the univariate case, namely where \(\{P_n\}\) and \(\{Q_n\}\) are probability distributions on \(\R\). We generalize their results and offer a unified perspective through the theory of large deviations. 
    
    \subsection{Preliminaries}

    Without loss of generality, we will take \(Q_n \ll P_n\) for all \(n \geq 1\). No generality is lost as argued in Section III.C in \cite{caiOptimalDetectionSparse2014}. Assume further that \(\{P_n\}\) and \(\{Q_n\}\) are dominated by a common measure on \(\mathcal{X}\) and so admit densities \(\{p_n\}\) and \(\{q_n\}\). We bundle these assumptions together as Assumption \ref{assumption:probs}, which will be in force for the remainder of the paper.

    \begin{assumption}\label{assumption:probs}
        The probability distributions \(\{P_n\}\) and \(\{Q_n\}\) are dominated by a common measure on a separable metric space \(\mathcal{X}\) and admit densities \(\{p_n\}\) and \(\{q_n\}\). Furthermore, \(Q_n \ll P_n\) for all \(n \geq 1\).
    \end{assumption}

    The following definition introduces precise quantities \(\underline{\beta}^*\) and \(\overline{\beta^*}\) which enable a precise statement of the detection problem. Specifically, the detection problem, for fixed \(\{P_n\}\) and \(\{Q_n\}\), is the problem of explicitly characterizing \(\underline{\beta}^*\) and \(\overline{\beta}^*\).
    
    \begin{definition}
        Consider the testing problem (\ref{problem:sparse_mixture_detection_1})-(\ref{problem:sparse_mixture_detection_2}) with calibration (\ref{eqn:beta_sparsity}). Define
        \begin{align*}
            \overline{\beta}^* &:= \inf\left\{\beta \geq 0 : \lim_{n \to \infty} \TV(P_n^n, ((1-n^{-\beta})P_n + n^{-\beta}Q_n)^n) = 0\right\}, \\
            \underline{\beta}^* &:= \sup\left\{\beta \geq 0 : \lim_{n \to \infty} \TV(P_n^n, ((1-n^{-\beta})P_n + n^{-\beta}Q_n)^n) = 1\right\}
        \end{align*}
        where \(\TV\) denotes the total variation distance.
    \end{definition}
    
    By the Neyman-Pearson lemma, \(\overline{\beta}^*\) is the smallest number such that if \(\beta > \overline{\beta}^*\), then every sequence of tests for testing (\ref{problem:sparse_mixture_detection_1})-(\ref{problem:sparse_mixture_detection_2}) has a sum of Type I and Type II errors converging to one. Likewise, \(\underline{\beta}^*\) is the largest number such that if \(\beta < \underline{\beta}^*\), then there exists a sequence of tests for testing (\ref{problem:sparse_mixture_detection_1})-(\ref{problem:sparse_mixture_detection_2}) with vanishing sum of Type I and Type II errors. The quantities \(\overline{\beta}^*\) and \(\underline{\beta}^*\) can be equivalently characterized by the asymptotics of the Hellinger distance. The analysis is much more amenable due to the tensorization of the Hellinger distance, and the following lemma gives a clean characterization \cite{caiOptimalDetectionSparse2014}.
    
    \begin{lemma}[Equations (25) and (26) - \cite{caiOptimalDetectionSparse2014}]\label{lemma:Hellinger}
        Consider the testing problem (\ref{problem:sparse_mixture_detection_1})-(\ref{problem:sparse_mixture_detection_2}) with calibration (\ref{eqn:beta_sparsity}). Let \(H_n^2(\beta) := H^2(P_n, (1-n^{-\beta})P_n + n^{-\beta}Q_n)\) where \(H\) is the Hellinger distance. Then
        \begin{align*}
            \overline{\beta}^* &= \inf\{\beta \geq 0 : H_n^2(\beta) = o(n^{-1})\}, \\
            \underline{\beta}^* &= \sup\{\beta \geq 0 : H_n^2(\beta) = \omega(n^{-1})\}.
        \end{align*}
    \end{lemma}
    Lemma 1 of \cite{caiOptimalDetectionSparse2014} establishes the result
    \begin{equation*}
        0 \leq \underline{\beta}^* \leq \overline{\beta}^* \leq 1,
    \end{equation*}
    confirming the intuition that \(\underline{\beta}^* \leq \overline{\beta}^*\) and additionally establishing that any phase transition must occur in \([0, 1]\).

    \subsection{Main results}
    We present our main results in this section. First, we briefly state a few definitions which are special cases of definitions formulated in the general large deviations theory \cite{demboLargeDeviationsTechniques2010}. 

    \begin{definition}
        Let \(\mathcal{Y}\) be a separable metric space. A \textit{rate function} \(I : \mathcal{Y} \to [0, \infty]\) is a lower semicontinuous function. A rate function \(I\) is \textit{good} if the sublevel sets \(\{y \in \mathcal{Y} : I(y) \leq \alpha\}\) are compact for all \(\alpha \geq 0\).
    \end{definition}

    \begin{definition}\label{def:general_ldp}
        Let \(\{\mu_n\}\) be a family of probability measures on \((\mathcal{Y}, \mathcal{B})\) where \(\mathcal{Y}\) is a separable metric space and \(\mathcal{B}\) is the completed Borel field on \(\mathcal{Y}\). We say that \(\{\mu_n\}\) satisfies the \textit{large deviation principle} with speed \(\{a_n\}\) and rate function \(I\) if for all \(\Gamma \in \mathcal{B}\),
        \begin{equation*}
            -\inf_{y \in \Gamma^\circ} I(y) \leq \liminf_{n \to \infty} a_n \log \mu_n (\Gamma) \leq \limsup_{n \to \infty} a_n \log \mu_n (\Gamma) \leq -\inf_{y \in \overline{\Gamma}} I(y).
        \end{equation*}
        Here, \(\{a_n\}\) is a sequence of reals with \(a_n \to 0\). Additionally, \(\Gamma^\circ\) and \(\overline{\Gamma}\) denote the interior and closure of \(\Gamma\) respectively.
    \end{definition}

    We specialize the definition of the large deviation principle further to a form most frequently used in our arguments. Note that the following definition requires that \(I\) be a good rate function, whereas the general definition of the large deviation principle does not have such a requirement.

    \begin{definition}\label{def:ldp}
        Suppose \(\{P_n\}\) and \(\{Q_n\}\) satisfy Assumption \ref{assumption:probs}. We say that the sequence of (normalized) log-likelihood ratios \(\left\{\frac{\log \frac{q_n}{p_n}}{\log n}\right\}\) satisfies the \textit{large deviation principle under the null} if there exists a good rate function \(I : \R \to [0, \infty]\) and for all Borel sets \(\Gamma \subset \R\) we have 
        \begin{align*}
            - \inf_{t \in \Gamma^\circ} I(t) &\leq \liminf_{n \to \infty}\frac{1}{\log n} \cdot \log P\left(\frac{\log \frac{q_n}{p_n}(X_n)}{\log n} \in \Gamma \right) \\
            &\leq \limsup_{n \to \infty} \frac{1}{\log n} \cdot \log P\left(\frac{\log \frac{q_n}{p_n}(X_n)}{\log n} \in \Gamma \right) \leq - \inf_{t \in \overline{\Gamma}} I(t)
        \end{align*}
        where \(X_n \sim P_n\). 
    \end{definition}

    Intuitively, the rate function \(I\) quantifies the asymptotic order by which \(Q_n\) deviates from \(P_n\). In particular, values of \(t \in \R\) such that \(I(t)\) is small are relatively unlikely values for the normalized log-likelihood ratio \(\frac{\log \frac{q_n}{p_n}(X_n)}{\log n}\) when the observation \(X_n\) is truly from the null, i.e. \(X_n \sim P_n\). Roughly speaking, observing values of the normalized log-likelihood ratio for which \(I\) is small constitutes as ``evidence" against the null. In this rough manner of speaking, the rate function can be thought of as quantifying the ``magnitude of evidence" against the null across all possible values the normalized log likelihood ratio may take on. In our main result, the rate function \(I\) is the fundamental object determining the phase transition. 
    
    \begin{theorem}\label{thm:beta_upper}
        Suppose \(\{P_n\}\) and \(\{Q_n\}\) are probability distributions that satisfy Assumption \ref{assumption:probs} for the testing problem (\ref{problem:sparse_mixture_detection_1})-(\ref{problem:sparse_mixture_detection_2}) with calibration (\ref{eqn:beta_sparsity}). Suppose there exists some \(\gamma > 1\) such that
        \begin{equation}\label{eqn:varadhan_tail_condition}
            \limsup_{n \to \infty} \frac{1}{\log n} \cdot \log E\left[ \left( \frac{q_n}{p_n}(X_n) \right)^\gamma \right] < \infty
        \end{equation}
        where \(X_n \sim P_n\). Suppose further that \(\left\{\frac{\log \frac{q_n}{p_n}}{\log n}\right\}\) satisfies the large deviation principle under the null. Let \(I : \R \to [0, \infty]\) be the associated good rate function. Define 
        \begin{equation}
            \overline{\beta}^\# := \frac{1}{2} + 0 \vee \sup_{t \geq 0} \left\{ t - I(t) + \frac{1 \wedge I(t)}{2} \right\}.
        \end{equation}
        Then \(\overline{\beta}^* \leq \overline{\beta}^\#\).
    \end{theorem}

    \begin{theorem}\label{thm:beta_lower}
        Consider the setting of Theorem \ref{thm:beta_upper}. Define 
        \begin{equation}
            \underline{\beta}^\# := \frac{1}{2} + \sup_{t > 0} \left\{ t - I(t) + \frac{1 \wedge I(t)}{2} \right\}.
        \end{equation}
        If the conditions of Theorem \ref{thm:beta_upper} hold, then \(\underline{\beta}^\# \leq \underline{\beta}^*\).
    \end{theorem}
    The most interesting situation is when the upper and lower bounds meet, yielding a detection boundary.

    \begin{corollary}\label{corollary:tight_limit}
        Consider the setting of Theorem \ref{thm:beta_upper}. If the conditions of Theorem \ref{thm:beta_upper} hold and
        \begin{equation}
            0 \vee \sup_{t \geq 0}\left\{ t - I(t) + \frac{1 \wedge I(t)}{2} \right\} = \sup_{t > 0} \left\{ t - I(t) + \frac{1 \wedge I(t)}{2} \right\},
        \end{equation}
        then \(\underline{\beta}^* = \overline{\beta}^* = \beta^*\) where
        \begin{equation}\label{eqn:tight_beta}
            \beta^* := \frac{1}{2} + 0 \vee \sup_{t \geq 0}\left\{ t - I(t) + \frac{1 \wedge I(t)}{2} \right\}. 
        \end{equation}
    \end{corollary}
    Corollary \ref{corollary:tight_limit} is our main result concerning phase transitions in the general sparse mixture detection problem. As mentioned before, the rate function \(I\) is fundamental in that it fully determines the detection boundary (provided the conditions of Corollary \ref{corollary:tight_limit} hold). The following corollary is a repackaging of Corollary \ref{corollary:tight_limit} with a sufficient condition that may be easier for ``off-the-shelf'' use.

    \begin{corollary}\label{corollary:nice_tight_limit}
        Consider the setting of Theorem \ref{thm:beta_upper}. If the conditions of Theorem \ref{thm:beta_upper} hold, \(I\) is right continuous at \(0\) i.e. \(\lim_{t \to 0^+} I(t) = I(0)\), and if there exists \(t^* \geq 0\) such that 
        \begin{equation}
            t^* - I(t^*) + \frac{1 \wedge I(t^*)}{2} \geq 0,
        \end{equation}
        then \(\underline{\beta}^* = \overline{\beta}^* = \beta^*\) where \(\beta^*\) is given by equation (\ref{eqn:tight_beta}). 
    \end{corollary}

    It is worth pausing to compare the role of the rate function with an analogous function in the framework of Cai and Wu \cite{caiOptimalDetectionSparse2014}. We state the main theorem regarding phase transitions from \cite{caiOptimalDetectionSparse2014} below with notational modifications to fit our context. 
    \begin{theorem}[Theorem 3 - \cite{caiOptimalDetectionSparse2014}]
        Consider the testing problem (\ref{problem:sparse_mixture_detection_1})-(\ref{problem:sparse_mixture_detection_2}) with calibration (\ref{eqn:beta_sparsity}). Suppose Assumption \ref{assumption:probs} holds. Let \(F_n\) and \(z_n\) denote the cumulative distribution function and quantile function of \(P_n\) respectively, i.e. \(z_n(p) = \inf \{y \in \R : F_n(y) \geq p\}\) for \(p \in [0, 1]\). Assume that the log-likelihood ratio \(\ell_n := \log\frac{q_n}{p_n}\) satisfies 
        \begin{equation*}
            \lim_{n \to \infty} \sup_{s \geq (\log_2 n)^{-1}} \left| \frac{\ell_n(z_n(n^{-s})) \vee \ell_n(z_n(1-n^{-s}))}{\log n} - \gamma(s) \right| = 0
        \end{equation*}
        as \(n \to \infty\) uniformly in \(s \in \R_{+}\) for some measurable function \(\gamma : \R^{+} \to \R\). If \(\gamma > 0\) on a set of positive Lebesgue measure, then 
        \begin{equation*}
            \beta^* = \frac{1}{2} + 0 \vee \ess \sup_{s \geq 0} \left\{ \gamma(s) - s + \frac{s \wedge 1}{2} \right\}.
        \end{equation*}
        Here, \(\ess\sup\) denotes the essential supremum with respect to Lebesgue measure on \(\R\).
    \end{theorem}
    In Theorem 3 of \cite{caiOptimalDetectionSparse2014}, the fundamental object determining the phase transition is the function \(\gamma : \R_{+} \to \R\), which is determined by the asymptotics of the log likelihood ratio \(\frac{\log \frac{q_n}{p_n}}{\log n}\) evaluated at the \(n^{-s}\) and \(1-n^{-s}\) quantiles of \(P_n\) across \(s \geq (\log_2 n)^{-1}\). Intuitively, the function \(\gamma\) quantifies the order at which \(Q_n\) deviates from \(P_n\). Roughly speaking, the points \(s\) at which \(\gamma(s) > 0\) are those \(P_n\)-quantiles which are ``relatively more likely'' under \(Q_n\) compared to \(P_n\); roughly, \(\gamma\) quantifies the order of ``more likely''.

    It's clear that the univariate nature of \(P_n\) and \(Q_n\) is heavily exploited as it is the likelihood ratio's asymptotic behavior at \emph{quantiles} of \(P_n\) that determines the fundamental object \(\gamma\). In the abstract setting, it is the rate function of a large deviation principle that precisely quantifies the asymptotic order of the likelihood ratio at various regions in the sample space. The rate function gives us a way to measure how much \(Q_n\) ``deviates'' from \(P_n\) asymptotically, thus allowing us to lift the core ideas of Cai and Wu (namely the idea to sharply characterize Hellinger asymptotics) to the abstract setting. We investigate the relationship between the rate function \(I\) and the function \(\gamma\) below in Proposition \ref{prop:weaker_than_cai_wu}. In fact, we can show that our condition that the normalized log-likelihood ratios satisfy the large deviation principle under the null (Definition \ref{def:ldp}) is implied, under some constraints, by the condition formulated in Theorem 3 of \cite{caiOptimalDetectionSparse2014}. 

    \begin{proposition}\label{prop:weaker_than_cai_wu}
        Suppose \(\{P_n\}\) and \(\{Q_n\}\) are probability distributions on \(\R\) satisfying Assumption \ref{assumption:probs} for the testing problem (\ref{problem:sparse_mixture_detection_1})-(\ref{problem:sparse_mixture_detection_2}). Let \(z^{(n)}\) denote the quantile function of \(P_n\). Assume there exist measurable functions \(\alpha_0 : \R_+ \to \R\) and \(\alpha_1 : \R_{+} \to \R\) such that 
        \begin{align}
            \lim_{n \to \infty} \sup_{s \geq (\log_2 n)^{-1}} \left| \frac{\log \frac{q_n}{p_n}(z^{(n)}(n^{-s}))}{\log n} - \alpha_0(s) \right| &= 0, \label{eqn:unif_convergence_0}\\
            \lim_{n \to \infty} \sup_{s \geq (\log_2 n)^{-1}} \left| \frac{\log \frac{q_n}{p_n}(z^{(n)}(1-n^{-s}))}{\log n} - \alpha_1(s) \right| &= 0. \label{eqn:unif_convergence_1}
        \end{align}
        If \(\alpha_0\) and \(\alpha_1\) are continuous, then \(\left\{\frac{\log \frac{q_n}{p_n}}{\log n} \right\} \) satisfies the large deviation principle under the null with good rate function
        \begin{equation*}
            I(t) = I_{0}(t) \wedge I_{1}(t)
        \end{equation*}
        where \(I_{0}\) and \(I_{1}\) are good rate functions given by
        \begin{align*}
            I_{0}(t) &= \inf\{s \geq 0 : t = \alpha_0(s)\}, \\
            I_{1}(t) &= \inf\{s \geq 0 : t = \alpha_1(s)\}.
        \end{align*}
        We use the convention that \(\inf \emptyset = \infty\).
    \end{proposition}

    Note that if the conditions of Corollary \ref{corollary:tight_limit} also hold, then the detection boundary is 
    \begin{align*}
        \beta^* &= \frac{1}{2} + 0 \vee \sup_{t \geq 0} \left\{ t - I(t) + \frac{1 \wedge I(t)}{2}\right\} \\
        &= \frac{1}{2} + 0 \vee \sup_{t \geq 0} \sup_{s: s = I(t)} \left\{t - s + \frac{1 \wedge s}{2} \right\} \\
        &= \frac{1}{2} + 0 \vee \sup_{s \in I(\R)} \sup_{t : I(t) = s} \left\{t - s + \frac{1 \wedge s}{2} \right\} \\
        &= \frac{1}{2} + 0 \vee \sup_{s \in I(\R)} \left\{\alpha_0(s) \vee \alpha_1(s) - s + \frac{1 \wedge s}{2} \right\} \\
        &= \frac{1}{2} + 0 \vee \sup_{s \geq 0} \left\{\alpha_0(s) \vee \alpha_1(s) - s + \frac{1 \wedge s}{2}\right\}
    \end{align*}
    where the final equality follows from the observation that \(-s + \frac{1 \wedge s}{2} \leq 0\) for all \(s \geq 0\). Observe that this is the same formula that appears in Theorem 3 of \cite{caiOptimalDetectionSparse2014} with \(\gamma(s) = \alpha_0(s) \vee \alpha_1(s)\). Note that essential supremum and supremum coincide as we assumed \(\alpha_0\) and \(\alpha_1\) are continuous. However, note that the conditions of Theorem 3 in \cite{caiOptimalDetectionSparse2014} are not exactly the same as the conditions in Proposition \ref{prop:weaker_than_cai_wu}, and so Proposition \ref{prop:weaker_than_cai_wu} does not fully import the results of Theorem 3 in \cite{caiOptimalDetectionSparse2014}.
    
    \section{A Higher Criticism type statistic}\label{section:HC}
    Beyond the detection problem, the adaptation problem is of practical interest. Namely, it's of interest to furnish an optimal sequence of tests that adapts to the unknown signal sparsity \(\beta\). The existing literature of sparse mixture detection has focused on the setting where \(P_n\) and \(Q_n\) are distributions on \(\R\), and a useful idea in this setting is to compare the empirical distribution of the data to the null distribution, usually after some transformation. Typically, this transformation takes the form of calculating a \(p\)-value \(p_i\) from each univariate data point \(X_i\). In the multivariate setting, this idea can be mimicked by calculating a \(p\)-value \(p_i\) from some univariate statistic of the multivariate observation \(X_i\). As mentioned in the introduction, Higher Criticism can be interpreted as evaluating the goodness-of-fit between the empirical \(p\)-value distribution and the null distribution. Recall that the associated test is given by \(\psi_{\HC_n} = \mathbf{1}_{\left\{\HC_n > \sqrt{2(1+\delta) \log \log n}\right\}}\) where \(\delta > 0\) is an arbitrary constant and the Higher Criticism statistic is 
    \begin{equation*}
        \HC_n = \sup_{u \in (0, 1)} \frac{\left|\sum_{i=1}^{n} \mathbf{1}_{\{p_i \leq u\}} - n u \right|}{\sqrt{n u (1-u)}}.    
    \end{equation*}
    Ditzhaus \cite{ditzhausSignalDetectionPhidivergences2019} established that Higher Criticism is adaptively optimal (meaning it achieves the detection boundary without needing knowledge of \(\beta\)) in a wide class of univariate sparse mixture detection problems beyond the original sparse normal mixture problem considered by Donoho and Jin \cite{donohoHigherCriticismDetecting2004}.

    To apply Higher Criticism to a multivariate sparse mixture detection problem, a univariate \(p\)-value \(p_i\) from the multivariate observation \(X_i\) must be constructed; these constructions are usually very specific to the problem at hand. In fact, it is usually not clear how a \(p\)-value should be crafted without loss of information resulting in a suboptimal testing procedure. However, the broad applicability of Higher Criticism is quite attractive, and so it is of interest to reconcile these difficulties. An approach to reconciliation is given by Gao and Ma in Section 3.2 of \cite{gaoTestingEquivalenceClustering2019}, the key idea being to apply Higher Criticism to the univariate statistics \(\left\{\frac{q_n}{p_n}(X_i)\right\}_{i=1}^{n}\). In this section, we review Gao and Ma's construction and give a sufficient condition for optimality.
    
    For notational convenience, we will refer to the null and alternative hypotheses in (\ref{problem:sparse_mixture_detection_1})-(\ref{problem:sparse_mixture_detection_2}) as \(H_0\) and \(H_1\), suppressing the dependence on \(n\) as the context makes it clear. Gao and Ma formulate a Higher Criticism type testing statistic for the testing problem (\ref{problem:sparse_mixture_detection_1})-(\ref{problem:sparse_mixture_detection_2}) when the distributions \(\{P_n\}\) and \(\{Q_n\}_{n\geq 1}\) are known but \(\beta\) is unknown. We reproduce their formulation here. For a collection of events \(\A\), define 
    \begin{equation}
        \HC_n(\A) := \sup_{A \in \A} |T_n(A)|
    \end{equation}
    where 
    \begin{equation}
        T_n(A) := \frac{\sum_{i=1}^{n} \mathbf{1}_{\{X_i \in A\}} - nP_n(A)}{\sqrt{n P_n(A) (1-P_n(A))}}.
    \end{equation}
    As Gao and Ma note, the supremum in the definition of \(\HC_n(\A)\) is largely for the sake of adapting to the unknown sparsity parameter \(\beta\). Focusing attention on \(T_n\), fix \(A \in \A\) and consider the test
    \begin{equation}
        \varphi_A := \mathbf{1}_{\{|T_n(A)| > c_n\}}
    \end{equation}
    where \(c_n\) is a positive diverging sequence to be specified. Observing that \(E_{H_0}(T_n(A)) = 0\) and \(\Var_{H_0}(T_n(A)) = 1\) for all \(n \geq 1\), consider 
    \begin{align*}
        P_{H_0}\{\varphi_A = 1\} = P_{H_0}\{|T_n(A)| \geq c_n\} \leq \frac{1}{c_n^2}
    \end{align*}
    by Chebyshev's inequality, and so the Type I error goes to zero. Turning attention to the Type II error, observe that if \(c_n\) diverges at a slower order than \(|E_{H_1}(T_n(A))|\) diverges, then for all sufficiently large \(n\), it follows that
    \begin{align*}
        P_{H_1}\{\varphi_A = 0\} &= P_{H_1}\{|T_n(A)| \leq c_n\} \\
        &\leq P_{H_1}\{|E_{H_1}(T_n(A))| - |T_n(A) - E_{H_1}(T_n(A))| \leq c_n \} \\
        &= P_{H_1}\{|E_{H_1}(T_n(A))| - c_n \leq |T_n(A) - E_{H_1}(T_n(A))|\} \\
        &\leq \frac{\Var_{H_1}(T_n(A))}{(|E_{H_1}(T_n(A))| - c_n)^2} \\
        &\asymp \frac{\Var_{H_1}(T_n(A))}{(E_{H_1}(T_n(A)))^2}.
    \end{align*}
    Therefore, if \(c_n\) diverges sufficiently slowly and \(\frac{(E_{H_1}(T_n(A)))^2}{\Var_{H_1}(T_n(A))} \to \infty\), then both the Type I and Type II error go to zero and so the test \(\varphi_A\) is consistent. Thus, characterizing the consistency of the test \(\varphi_A\) boils down to characterizing when \(\frac{(E_{H_1}(T_n(A)))^2}{\Var_{H_1}(T_n(A))} \to \infty\). 
    
    Gao and Ma directly calculate (equation (33) in \cite{gaoTestingEquivalenceClustering2019})
    \begin{equation*}
        \frac{(E_{H_1} T_n(A))^2}{\Var_{H_1}(T_n(A))} \asymp \frac{(n\varepsilon (Q_n(A) - P_n(A)))^2}{nP_n(A) + n\varepsilon Q_n(A)}
    \end{equation*}
    and so if \(\frac{(n\varepsilon Q_n(A))^2}{nP_n(A) + n\varepsilon Q_n(A)} \to \infty\), then \(\frac{(E_{H_1} T_n(A))^2}{\Var_{H_1}(T_n(A))} \to \infty\). (As Gao and Ma note, the condition \(n\varepsilon^2 P_n(A) \to \infty\) is also sufficient but has the strong and uninteresting requirement \(\beta < \frac{1}{2}\).) Equivalently, if both
    \begin{align*}
        \frac{n\varepsilon^2 Q_n(A)^2}{P_n(A)} &\to \infty, \\
        n\varepsilon Q_n(A) &\to \infty
    \end{align*}
    hold, then \(\frac{(E_{H_1} T_n(A))^2}{\Var_{H_1}(T_n(A))} \to \infty\). The conditions are equivalent to 
    \begin{equation}\label{eqn:HC_beta}
        \beta < \frac{1}{2} + \frac{\log Q_n(A)}{\log n} + \frac{1}{2} \min\left(1, -\frac{\log P_n(A)}{\log n} \right).
    \end{equation}
    In other words, if \(\beta\) satisfies condition (\ref{eqn:HC_beta}) for all \(n\) sufficiently large for some sequence of events \(\{A_n\} \subset \A\) and \(c_n\) is a positive sequence diverging sufficiently slowly, then the sequence of tests \(\varphi_{A_n}\) is consistent for testing (\ref{problem:sparse_mixture_detection_1})-(\ref{problem:sparse_mixture_detection_2}). To maximize the set of \(\beta\) satisfying condition (\ref{eqn:HC_beta}) for all \(n\) sufficiently large, one should select an event \(A_n\) that maximizes the right hand side of (\ref{eqn:HC_beta}) for each \(n\). Since the right hand side of (\ref{eqn:HC_beta}) is increasing in \(Q_n(A)\) and decreasing in \(P_n(A)\), Gao and Ma argue that the Neyman-Pearson lemma implies that the maximum is achieved by the event \(A_n = \left\{ x \in \mathcal{X} : \frac{q_n}{p_n}(x) > t_n \right\}\) for some \(t_n > 0\). With this observation in mind, Gao and Ma naturally select the collection of events 
    \begin{equation}\label{eqn:event_collection}
        \A^*_n := \left\{ \left\{ x \in \mathcal{X} : \frac{q_n}{p_n}(x) > t \right\} : t > 0\right\}
    \end{equation}
    and define the general HC-type statistic
    \begin{equation}\label{eqn:HC_star}
        \HC_n^* := \HC_n(\A^*_n) = \sup_{t > 0} \frac{\left|\sum_{i=1}^{n} \mathbf{1}_{\left\{\frac{q_n}{p_n}(X_i) > t\right\}} - nP\left(\frac{q_n}{p_n}(Y_n) > t\right)\right|}{\sqrt{n P\left( \frac{q_n}{p_n}(
            Y_n) > t\right)P\left( \frac{q_n}{p_n}(Y_n) \leq t \right)}}
    \end{equation}
    where \(Y_n \sim P_n\) are independent of the data \(\{X_i\}_{i=1}^{n}\). The corresponding higher criticism test is 
    \begin{equation}\label{eqn:HC_test}
        \psi_{\HC_n^*} := \mathbf{1}_{\left\{\HC_n^* > \sqrt{2(1+\delta)\log \log(n)}\right\}}
    \end{equation}
    where \(\delta > 0\) is an arbitrary constant. Note that the cutoff \(\sqrt{2(1+\delta) \log \log n}\) is the same cutoff used in Donoho and Jin's original formulation of the Higher Criticism test (\ref{eqn:donoho_jin_HC_test}); this choice of cutoff is not at all surprising since \(\HC^*_n\) is precisely Higher Criticism applied to univariate statistics. We now formally define a quantity \(\underline{\beta}^{\HC}\) which demarcates the sparsity levels for which \(\psi_{\HC_n^*}\) consistently tests (\ref{problem:sparse_mixture_detection_1})-(\ref{problem:sparse_mixture_detection_2}).

    \begin{definition}\label{def:beta_HC}
        Consider the testing problem (\ref{problem:sparse_mixture_detection_1})-(\ref{problem:sparse_mixture_detection_2}) with calibration (\ref{eqn:beta_sparsity}). Define
        \begin{equation}
            \underline{\beta}^{\HC} := \frac{1}{2} + \sup \, \liminf_{n \to \infty} \left\{\frac{\log Q_n(A_n)}{\log n} + \frac{1}{2}\min\left(1, -\frac{\log P_n(A_n)}{\log n} \right) \right\}
        \end{equation}
        where the supremum runs over sequences of events \(\{A_n\}\) with \(A_n \in \A_n^*\). Here, \(\A_n^*\) is given by (\ref{eqn:event_collection}).
    \end{definition}

    \begin{proposition}\label{prop:HC_upper_bound}
        Consider the testing problem (\ref{problem:sparse_mixture_detection_1})-(\ref{problem:sparse_mixture_detection_2}) with calibration (\ref{eqn:beta_sparsity}). If \(\beta < \underline{\beta}^{\HC}\), then \(\psi_{\HC_n^*}\) is consistent. Here, \(\psi_{\HC_n^*}\) is given by (\ref{eqn:HC_test}).
    \end{proposition}
    \begin{proof}
        The choice of threshold \(\sqrt{2(1+\delta)\log \log (n)}\) is given by Theorem 1.1 of \cite{donohoHigherCriticismDetecting2004}. Theorem 1.1 of \cite{donohoHigherCriticismDetecting2004} implies that this choice of threshold results in a vanishing Type I error of \(\psi_{\HC_n^*}\). 
        
        Turning attention to the Type II error, if \(\beta\) in the calibration (\ref{eqn:beta_sparsity}) satisfies condition (\ref{eqn:HC_beta}) for all \(n\) sufficiently large for some sequence of events \(\widetilde{A}_n \in \A^*_n\), then it immediately follows that 
        \begin{align}
            P_{H_1}\{\psi_{\HC_n^*} = 0\} &= P_{H_1}\left\{\sup_{A_n \in \A^*_n} |T_n(A_n)| \leq \sqrt{2(1+\delta)\log\log(n)}\right\} \\
            &\leq P_{H_1}\left\{|T_n(\widetilde{A}_n)| \leq \sqrt{2(1+\delta)\log\log(n)}\right\} \label{eqn:HC_Type2}\\
            &= P_{H_1}\{\varphi_{\widetilde{A}_n} = 0\}
        \end{align}
        where \(c_n = \sqrt{2(1+\delta) \log\log(n)}\). Now, observe that \(|E_{H_1}(T_n(\widetilde{A}_n))| = \left| \frac{\varepsilon \sqrt{n}Q_n(\widetilde{A}_n)}{\sqrt{P_n(\widetilde{A}_n)}} - \frac{\varepsilon \sqrt{nP_n(\widetilde{A}_n)}}{\sqrt{1-P_n(\widetilde{A}_n)}} \right|\). Since \(\beta\) satisfies condition (\ref{eqn:HC_beta}) for \(\widetilde{A}_n\) for all sufficiently large \(n\), it immediately follows that \(|E_{H_1}(T_n(\widetilde{A}_n))|\) diverges at a polynomial rate. Moreover, it follows that \(c_n\) diverges sufficiently slowly as \(c_n\) grows at a sub-polynomial rate. Hence, \(P_{H_1}\{\varphi_{\widetilde{A}_n} = 0\}\) converges to zero and so the Type II error of \(\psi_{\HC^*_n}\) vanishes. Therefore, it has been shown that if \(\beta\) in the calibration (\ref{eqn:beta_sparsity}) satisfies condition (\ref{eqn:HC_beta}) for some sequence of events \(\widetilde{A}_n \in \A^*_n\) for all \(n\) sufficiently large, then \(\psi_{\HC^*_n}\) is consistent.
    \end{proof}

    At first glance, it seems that \(\HC^*_n\) requires full knowledge of both the null \(\{P_n\}\) and signal \(\{Q_n\}\) distributions. In contrast, Donoho and Jin's formulation of Higher Criticism (\ref{eqn:donoho_jin_HC}) for the sparse normal mixture detection problem does not require knowledge of the signal strength \(r\). The key observation is that the computation of \(\HC^*_n\) actually only requires knowledge of the collection \(\A^*_n\) so that one may calculate \(\sup_{A \in \A^*_n} |T_n(A)|\). In some cases, it's possible to calculate the supremum without knowing \(\{Q_n\}\) explicitly. For example, consider the sparse normal mixture detection problem with \(P_n = N(0, 1)\) and \(Q_n = N(\sqrt{2r \log n}, 1)\) with \(0 < r \leq 1\). A direct calculation shows 
    \begin{align*}
        \A_n^* &= \left\{ \left\{x \in \R : \frac{q_n}{p_n}(x) > t \right\} : t > 0 \right\} \\
               &= \left\{ \left\{x \in \R : \exp\left(x \sqrt{2r \log n} - r \log n\right) > t \right\} : t > 0 \right\} \\
               &= \left\{ \left\{ x \in \R : x > t \right\} : t \in \R \right\}.
    \end{align*}
    Therefore, \(\HC^*_n\) reduces to Donoho and Jin's Higher Criticism statistic \(\HC_n\) given in (\ref{eqn:donoho_jin_HC}), and so knowledge of the signal strength \(r\) is not required. In many problems, it may be the case that computation of \(\HC^*_n\) does not require full knowledge of the signal distribution \(Q_n\) even though Gao and Ma's construction gives that impression. 
    
    With the \(\HC^*_n\) testing statistic in hand, the next challenge is to investigate \(\underline{\beta}^{\HC}\). When \(\left\{\frac{\log \frac{q_n}{p_n}}{\log n}\right\}\) satisfies the large deviation principle under the null, a lower bound can be derived for \(\underline{\beta}^{\HC}\).

    \begin{proposition}\label{prop:HC_bound_formula}
        Consider the testing problem (\ref{problem:sparse_mixture_detection_1})-(\ref{problem:sparse_mixture_detection_2}) with calibration (\ref{eqn:beta_sparsity}). Suppose there exists some \(\gamma > 1\) such that
        \begin{equation*}\label{eqn:HC_bound_moment_condition}
            \limsup_{n \to \infty} \frac{1}{\log n} \cdot \log E\left[ \left( \frac{q_n}{p_n}(X_n) \right)^\gamma \right] < \infty
        \end{equation*}
        where \(X_n \sim P_n\). Suppose further that \(\left\{\frac{\log \frac{q_n}{p_n}}{\log n}\right\}\) satisfies the large deviation principle under the null. Let \(I : \R \to [0, \infty]\) be the associated good rate function. Then,
        \begin{equation}\label{eqn:beta_HC_lower_bound}
            \underline{\beta}^* \geq \underline{\beta}^{\HC} \geq \frac{1}{2} + \sup_{c \geq 0} \left\{ \sup_{t > c} \{t - I(t)\} + \frac{1 \wedge \inf_{t \geq c} I(t)}{2} \right\}.    
        \end{equation}
    \end{proposition}

    When the rate function associated to the large deviation principle is convex (and satisfies some further constraints) and the conditions of Corollary \ref{corollary:tight_limit} hold, it can be shown that the lower and upper bounds in (\ref{eqn:beta_HC_lower_bound}) match.

    \begin{theorem}\label{thm:HC_tight}
        Consider the setting of Proposition \ref{prop:HC_bound_formula} and suppose that the conditions of Corollary \ref{corollary:tight_limit} hold. Suppose \(I\) is convex. Let \(D := \{t \in \R : I(t) < \infty\}\) and note that \(D\) is an interval with some left endpoint \(\underline{d}\) and some right endpoint \(\overline{d}\). Suppose further that \(I\) is such that if \(\underline{d} \in D\), we have that \(I\) is right-continuous at \(\underline{d}\) and if \(\overline{d} \in D\), we have that \(I\) is left-continuous at \(\overline{d}\). Let \(I_{-}'(t)\) be the left derivative of \(I\) (see Definition \ref{def:left_right_deriv}) with the domain of definition extended as in the statement of Theorem \ref{thm:rockafellar24.1}. Define 
        \begin{equation}
            t_0 := \sup\{t \geq 0 : I_{-}'(t) \leq 0\}
        \end{equation}
        and set \(t_0 = 0\) if \(\{t \geq 0 : I_{-}'(t) \leq 0\} = \emptyset\). Likewise, define 
        \begin{equation}
            t_1 := \sup\{t \geq 0 : I_{-}'(t) \leq 1\}
        \end{equation}
        and set \(t_1 = 0\) if \(\{t \geq 0 : I_{-}'(t) \leq 1\} = \emptyset\).  \newline
        
        If \(t_0 \vee t_1 < \infty\), then 
        \begin{equation*}
            \underline{\beta}^{\HC} = \underline{\beta}^* = \overline{\beta}^* = \beta^* = \frac{1}{2} + 0 \vee \sup_{t \geq 0} \left\{t - I(t) + \frac{1 \wedge I(t)}{2} \right\}.
        \end{equation*}
    \end{theorem}

    The G{\"a}rtner-Ellis Theorem gives general conditions ensuring the convexity of the rate function. We state a special case of the G{\"a}rtner-Ellis Theorem (Theorem 2.3.6 in \cite{demboLargeDeviationsTechniques2010}) specialized for our use in the sparse mixture detection problem. In many problems, the G{\"a}rtner-Ellis Theorem greatly simplifies the work needed in determining whether the large deviation principle under the null holds and computing the corresponding rate function. First, a regularity condition (Definition 2.3.5 from \cite{demboLargeDeviationsTechniques2010}) is needed.

    \begin{definition}
        Let \(\Lambda : \R \to (-\infty, \infty]\) be a convex function and let \(D_\Lambda := \{\lambda \in \R : \Lambda(\lambda) < \infty\}\). We say \(\Lambda\) is \textit{essentially smooth} if \(D_\Lambda^\circ \neq \emptyset\), \(\Lambda\) is differentiable on \(D_\Lambda^\circ\), and \(\lim_{n \to \infty} |\Lambda'(\lambda_n)| = \infty\) for any sequence \(\{\lambda_n\} \subset D_\Lambda^\circ\) converging to a point on the boundary of \(D_\Lambda^\circ\). 
    \end{definition}

    The following statement of the the G{\"a}rtner-Ellis Theorem follows the presentation of Theorem 2.3.6 in \cite{demboLargeDeviationsTechniques2010} with modifications to suit our setting.

    \begin{theorem}[G{\"a}rtner-Ellis]\label{thm:gartner_ellis}
        Suppose \(\{P_n\}\) and \(\{Q_n\}\) are probability measures for the testing problem (\ref{problem:sparse_mixture_detection_1})-(\ref{problem:sparse_mixture_detection_2}). For \(\lambda \in \R\), define
        \begin{equation}
            \Lambda_n(\lambda) := \frac{1}{\log n} \cdot \log E\left[\left(\frac{q_n}{p_n}(X_n) \right)^\lambda \right]
        \end{equation}
        where \(X_n \sim P_n\). Assume that the limit
        \begin{equation}\label{eqn:Lambda_n_limit}
            \lim_{n \to \infty} \Lambda_n(\lambda) =: \Lambda(\lambda)
        \end{equation}
        exists in \([-\infty, \infty]\) for \(\lambda \in \R\). If \(\Lambda\) is essentially smooth, is a lower semicontinuous function, and \(0 \in D_{\Lambda}^\circ\), then \(\left\{\frac{\log \frac{q_n}{p_n}}{\log n} \right\}\) satisfies the large deviation principle under the null with good, convex rate function 
        \begin{equation}
            \Lambda^*(t) := \sup_{\lambda \in \R} \left\{\lambda t - \Lambda(\lambda) \right\}.
        \end{equation}
    \end{theorem}
    \begin{proof}
        Since \(\lim_{n \to \infty} \Lambda_n(\lambda)\) exists in \([-\infty, \infty]\) and \(0 \in D_\Lambda^\circ\), it follows that Assumption 2.3.2 of \cite{demboLargeDeviationsTechniques2010} is satisfied. Then, Lemma 2.3.9 of \cite{demboLargeDeviationsTechniques2010} yields the convexity of \(\Lambda\) as well as establishing that \(\Lambda(\lambda) > -\infty\) for all \(\lambda\) and that \(\Lambda^*\) is a good convex rate function. Finally, Theorem 2.3.6 (G{\"a}rtner-Ellis) in \cite{demboLargeDeviationsTechniques2010} implies that \(\left\{\frac{\log \frac{q_n}{p_n}}{\log n} \right\}\) satisfies the large deviation principle under the null with rate function \(\Lambda^*\).
    \end{proof}

    \begin{remark} \label{remark:exponential_family}
        The G{\"a}rtner-Ellis Theorem simplifies calculating the rate function of the large deviation principle in some exponential families. Suppose \(\{f_{\theta} : \theta \in \Theta\}\) is an exponential family in the natural parametrization on a separable metric space with \(\Theta \subset \R^d\). Write 
        \begin{equation*}
            f_\theta(x) = c(\theta) h(x) \exp\left(\langle \theta, T(x)\rangle \right).
        \end{equation*}
        Taking \(p_n = f_{\theta}\) for some \(\theta \in \Theta\) and \(q_n = f_{\theta_n}\) for some sequence \(\{\theta_n\}\) in \(\Theta\), the limit (\ref{eqn:Lambda_n_limit}) becomes  
        \begin{equation*}
            \Lambda(\lambda) := \lim_{n \to \infty} \frac{\lambda \log c(\theta_n) + (1-\lambda) \log c(\theta) - \log c(\lambda(\theta_n - \theta) + \theta)}{\log n}
        \end{equation*}
        where we take \(\log c(\lambda (\theta_n - \theta) + \theta) = -\infty\) if \(\lambda(\theta_n - \theta) + \theta \not \in \Theta\). Of course, one must check that the limit exists and the remaining conditions of Theorem \ref{thm:gartner_ellis} hold.
    \end{remark}

    \section{Examples} \label{section:examples}
    To illustrate our results, we work out a few examples and derive explicit detection boundaries.

    \subsection{Ingster-Donoho-Jin}\label{example:idj}

    \subsubsection{Univariate}\label{example:univariate_idj}
    Consider the testing problem (\ref{problem:sparse_mixture_detection_1})-(\ref{problem:sparse_mixture_detection_2}) with calibration (\ref{eqn:beta_sparsity}) and distributions \(P_n = P = N(0, 1)\) and \(Q_n = N(\mu_n, 1)\). The detection boundary for this testing problem with calibration \(\mu_n = \sqrt{2r \log n}\) for \(0 < r \leq 1\) was obtained by Ingster \cite{ingsterProblemsHypothesisTesting1996} and then independently by Jin \cite{jinDetectingEstimatingSparse2003,jinDetectingTargetVery2004}. Donoho and Jin \cite{donohoHigherCriticismDetecting2004} introduced the Higher Criticism testing statistic and established its optimality in this sparse mixture detection problem. Following \cite{caiOptimalDetectionSparse2014, gaoTestingEquivalenceClustering2019,collierMinimaxEstimationLinear2017a}, we refer to the detection boundary as the Ingster-Donoho-Jin detection boundary.

    We illustrate how the large deviations perspective delivers both the Ingster-Donoho-Jin detection boundary 
    \begin{equation}\label{phase:beta_idj}
        \beta_{IDJ}^*(r) :=
        \begin{cases}
            \frac{1}{2} + r & \text{if } 0 < r \leq \frac{1}{4} \\
            1 - (1-\sqrt{r})_{+}^2 &\text{if } r > \frac{1}{4}
        \end{cases}
    \end{equation}
    and the optimality of the Higher Criticism statistic. We use the G{\"a}rtner-Ellis Theorem and Remark \ref{remark:exponential_family}. Before we begin the main computation, note that it is easily verified that the tail condition (\ref{eqn:varadhan_tail_condition}) is satisfied. Adopting the notation of Remark \ref{remark:exponential_family}, consider that \(\{N(\theta, 1) : \theta \in \R\}\) is an exponential family with natural parameter \(\theta\) and \(\log c(\theta) = -\frac{\theta^2}{2}\). Taking \(\theta = 0\) and \(\theta_n = \mu_n\), observe that 
    \begin{align*}
        \Lambda(\lambda) &:= \lim_{n \to \infty} \frac{\lambda \log c(\theta_n) + (1-\lambda) \log c(\theta) - \log c(\lambda(\theta_n - \theta) + \theta)}{\log n} \\ 
        &= \lim_{n \to \infty} \frac{-\lambda r \log n + \lambda^2 r \log n}{\log n} \\
        &= r(\lambda^2-\lambda).
    \end{align*}
    Noting that \(D_\Lambda = \R\), \(\Lambda\) is essentially smooth, and \(\Lambda\) is continuous, it follows from the G{\"a}rtner-Ellis Theorem that \(\left\{\frac{\log \frac{q_n}{p}}{\log n}\right\}\) satisfies the large deviation principle under the null with good convex rate function \(\Lambda^*(t) := \sup_{\lambda \in \R} \{ \lambda t - \Lambda(\lambda)\}\). Direct calculation yields \(\Lambda^*(t) = \frac{(t+r)^2}{4r}\). It is easily checked that conditions of Corollary \ref{corollary:nice_tight_limit} are satisfied. Therefore, the detection boundary is given by 
    \begin{equation*}
        \beta^*(r) = \frac{1}{2} + 0\vee \sup_{t\geq 0} \left\{t - \frac{(t+r)^2}{4r}  + \frac{1 \wedge \frac{(t+r)^2}{4r}}{2}\right\}.
    \end{equation*}
    Solving the optimization problem yields \(\beta^*(r) = \beta^*_{IDJ}(r)\).

    Turning our attention to the Higher Criticism statistic, observe that the general HC-type statistic (\ref{eqn:HC_star}) of Gao and Ma reduces to the original Higher Criticism statistic introduced by Donoho and Jin \cite{donohoHigherCriticismDetecting2004}
    \begin{align*}
        \HC^*_n &= \sup_{t >0} \frac{|\sum_{i=1}^{n} \mathbf{1}_{\{\exp(X_i\mu_n - \mu_n^2/2) > t\}} - n P(\exp(X\mu_n - \mu_n^2/2) > t)|}{\sqrt{n P(\exp(X\mu_n - \mu_n^2/2) > t)P(\exp(X\mu_n - \mu_n^2/2) \leq t)}} \\
        &= \sup_{t \in \R} \frac{|\sum_{i=1}^{n} \mathbf{1}_{\{X_i > t\}} - n P(X > t)|}{\sqrt{n P(X > t)P( X \leq t)}}
    \end{align*}
    where \(X \sim N(0, 1)\) is independent of \(\{X_i\}_{i=1}^{n}\). The conditions of Theorem \ref{thm:HC_tight} hold, and so the sequence of tests \(\psi_{\HC^*_n}\) given by (\ref{eqn:HC_test}) achieves the detection boundary while adapting to the parameters \(r\) and \(\beta\). In other words, \(\underline{\beta}^{\HC} = \beta^*\).

    \subsubsection{Multivariate}\label{example:mvn}
    The detection boundary for a multivariate version of the sparse normal mixture testing problem can be obtained in exactly the same fashion. Consider the testing problem (\ref{problem:sparse_mixture_detection_1})-(\ref{problem:sparse_mixture_detection_2}) with \(P_n = P = N(0, \Sigma)\), \(Q_n = N(\mu_n, \Sigma)\) where \(\Sigma \in \R^{d \times d}\) is a positive definite matrix. Further consider the calibration \(\mu_n = \sqrt{2r\log n} \cdot u\) where \(u \in \R^d\) with \(||u|| = 1\). We use the G{\"a}rtner-Ellis Theorem and Remark \ref{remark:exponential_family}. Before we begin the main computation, note that it is easily verified that the tail condition (\ref{eqn:varadhan_tail_condition}) is satisfied. Adopting the notation of Remark \ref{remark:exponential_family}, consider that \(\{N(\theta, \Sigma) : \theta \in \R^d\}\) is an exponential family with natural parameter \(\theta\) and \(\log c(\theta) = -\frac{\langle \theta, \Sigma^{-1} \theta \rangle}{2}\). Taking \(\theta = 0\) and \(\theta_n = \mu_n\), observe that 
    \begin{align*}
        \Lambda(\lambda) &:= \lim_{n \to \infty} \frac{\lambda \log c(\theta_n) + (1-\lambda) \log c(\theta) - \log c(\lambda(\theta_n - \theta) + \theta)}{\log n} \\ 
        &= \lim_{n \to \infty} \frac{-\lambda r (\log n) \cdot \langle u, \Sigma^{-1} u\rangle + \lambda^2 r (\log n) \cdot \langle u, \Sigma^{-1} u\rangle}{\log n}  \\
        &= r(\lambda^2-\lambda) \langle u, \Sigma^{-1} u \rangle.
    \end{align*}
    Repeating the arguments as in the Ingster-Donoho-Jin problem (Example \ref{example:univariate_idj}), it follows that \(\left\{\frac{\log \frac{q_n}{p}}{\log n} \right\}\) satisfies the large deviation principle under the null with rate function \(\Lambda^*(t) = \frac{\left(t + r\langle u, \Sigma^{-1} u \rangle \right)^2}{4r}\). Following the same reasoning as in Example \ref{example:univariate_idj}, we immediately obtain the detection boundary
    \begin{equation*}
        \beta^*(r) = 
        \begin{cases}
            \frac{1}{2} + r \langle u, \Sigma^{-1} u\rangle & \text{if } r \langle u, \Sigma^{-1} u \rangle \leq \frac{1}{4} \\
            1 - \left(1 - \sqrt{r \langle u, \Sigma^{-1} u\rangle} \right)_{+}^2 & \text{otherwise}.
        \end{cases}
    \end{equation*}
    Moreover, Theorem \ref{thm:HC_tight} guarantees that the sequence of tests \(\psi_{\HC^*_n}\) achieves the detection boundary. Consider that the testing statistic \(\HC_n^{*}\) can be written as 
    \begin{align*}
        \HC^*_n &= \sup_{t \in \R} \frac{\left|\sum_{i=1}^{n} \mathbf{1}_{\left\{\left\langle X_i, \frac{\Sigma^{-1} u}{||\Sigma^{-1} u||} \right\rangle > t\right\}} - n P\left( \left\langle X, \frac{\Sigma^{-1} u}{||\Sigma^{-1} u||} \right\rangle > t\right)\right|}{\sqrt{n P\left(\left\langle X, \frac{\Sigma^{-1} u}{||\Sigma^{-1} u||} \right\rangle > t\right)P\left(\left\langle X, \frac{\Sigma^{-1} u}{||\Sigma^{-1} u||} \right\rangle \leq t\right)}}.
    \end{align*}
    Hence, \(\HC^*_n\) is adaptive to \(r\) and \(\beta\). Furthermore, full knowledge of \(\Sigma\) is not needed, rather just knowledge of the signal direction \(\frac{\Sigma^{-1} u}{||\Sigma^{-1} u||}\) is needed.

    \subsubsection{Brownian motion}\label{example:brownian_motion}
    With the large deviations perspective in hand, a phase transition can be derived in a stylized sparse mixture detection problem where our observations are sample paths of Brownian motion with possible drift. In particular, let \(\mathcal{X} = C([0,1])\) be the space of all real-valued continuous functions on \([0,1]\), let \(P_n\) be the probability measure on \(\mathcal{X}\) associated with standard Brownian motion \(\{B_t\}_{t \in [0, 1]}\), and let \(Q_n\) be the probability measure on \(\mathcal{X}\) associated with the Brownian motion with drift \(\{m_n(t) + B_t\}_{t \in [0, 1]}\). Here, we take \(m_n(t) = \sqrt{2r\log n} \cdot f(t)\) for some fixed continuously differentiable \(f : [0, 1] \to \R\) with \(f(0) = 0\) and \(\int_{0}^{1} |f'(t)|^2 \, dt = 1\). With observations \(X_1,...,X_n \in \mathcal{X}\), the problem is to test (\ref{problem:sparse_mixture_detection_1})-(\ref{problem:sparse_mixture_detection_2}). Note that the observations are themselves real-valued functions on \([0, 1]\). Further note that since \(f' \in L^2([0, 1])\), the measures \(P_n\) and \(Q_n\) are mutually absolutely continuous (Example 4 in \cite{StatisticalProblemsTheory}). The normalized log-likelihood ratio is given by \cite{StatisticalProblemsTheory}
    \begin{align*}
        \frac{\log \frac{dQ_n}{dP_n}(X)}{\log n} &= -\frac{1}{2 \log n} \int_{0}^{1} |m_n'(t)|^2\, dt + \frac{1}{\log n}\int_{0}^{1} m'(t) \, dX_t \\
        &= -r + \frac{\sqrt{2r}}{\sqrt{\log n}} \int_{0}^{1} f'(t) \, dX_t.
    \end{align*}
    Under the null, \(\{X_t\}_{t \in [0, 1]}\) is standard Brownian motion. Since \(f\) is deterministic, it follows by the It\^{o} isometry that 
    \begin{equation*}
        \int_{0}^{1} f'(t) \, dX_t \sim N\left(0, \int_{0}^{1} |f'(t)|^2 \, dt\right) = N(0, 1).
    \end{equation*}
    Therefore, under the null we have 
    \begin{equation*}
        \frac{\log \frac{dQ_n}{dP_n}(X)}{\log n} \sim N\left(-r, \frac{2r}{\log n}\right).
    \end{equation*}
    Thus, for any Borel set \(\Gamma \subset \R\), we have under the null
    \begin{align*}
        P\left\{ \frac{\log \frac{dQ_n}{dP_n}(X)}{\log n} \in \Gamma \right\} &= \int_{\Gamma} \frac{\sqrt{2r}}{\sqrt{2\pi \log n}} \exp\left(-\frac{(x+r)^2}{4r} \cdot \log n \right) \, dx. 
    \end{align*}
    An application of Lemma 3 in \cite{caiOptimalDetectionSparse2014} immediately yields 
    \begin{equation*}
        \lim_{n \to \infty} \frac{\log P\left\{ \frac{\log \frac{dQ_n}{dP_n}(X)}{\log n} \in \Gamma \right\}}{\log n} = - \inf_{x \in \Gamma} \frac{(x+r)^2}{4r}.
    \end{equation*}
    Therefore, \(\left\{\frac{\log \frac{dQ_n}{dP_n}}{\log n} \right\}\) satisfies the large deviation principle under the null with good rate function \(I : \R \to [0, 1]\) given by \(I(t) = \frac{(t+r)^2}{4r}\). This is precisely the same rate function as in Example \ref{example:univariate_idj}, and so we immediately obtain exactly the same detection boundary \(\beta^*(r) = \beta^*_{IDJ}(r)\). Furthermore, it is clear that the rate function \(I\) is indeed convex and that the other conditions of Theorem \ref{thm:HC_tight} hold. So the \(\HC^*_n\) testing statistic indeed yields an optimal test. Note that 
    \begin{align*}
        \HC^*_n &= \sup_{t \in \R} \frac{\left|\sum_{i=1}^{n} \mathbf{1}_{\left\{ \int_{0}^{1} f'(s) \, dX_i(s) > t\right\}} - P\left(\int_{0}^{1} f'(s) \, dB_s > t\right)\right|}{\sqrt{n P\left(\int_{0}^{1} f'(s) \, dB_s > t\right) P\left( \int_{0}^{1} f'(s) \, dB_s \leq t \right)}}
    \end{align*}
    where \(\{B_s\}_{s \in [0, 1]}\) is a standard Brownian motion independent of the observations \(\{X_i\}_{i=1}^{n}\). While no knowledge of the signal strength \(r\) is needed to compute \(\HC^*_n\), knowledge of the ``signal direction'' \(f'\) is needed, analogous to the multivariate setting in Example \ref{example:mvn}.
 
    \subsection{Heteroscedastic normal mixture}\label{example:heteroscedastic}
    Cai, Jeng, and Jin \cite{caiOptimalDetectionHeterogeneous2011} consider the testing problem (\ref{problem:sparse_mixture_detection_1})-(\ref{problem:sparse_mixture_detection_2}) in a heteroscedastic normal mixture setting. More specifically, the setting where \(P_n = P = N(0, 1)\) and \(Q_n = N(\mu_n, \sigma^2)\) is considered with calibration \(\mu_n = \sqrt{2r \log n}, r > 0\), and fixed \(\sigma^2 > 0\).  Through an analysis of the likelihood ratio, they obtain the detection boundary 
    \begin{equation*}
        \beta^*(r, \sigma^2) := 
        \begin{cases}
            \frac{1}{2} + \frac{r}{2-\sigma^2} & \text{if } 2\sqrt{r} + \sigma^2 \leq 2, \\
            1 - \frac{(1-\sqrt{r})_{+}^2}{\sigma^2} & \text{if } 2\sqrt{r} + \sigma^2 > 2.
        \end{cases}
    \end{equation*}
    Note that the detection boundary stated in \cite{caiOptimalDetectionSparse2014} is in terms of \(r\) as a function of \(\beta\) and \(\sigma\), whereas the above boundary is in terms of \(\beta\) as a function of \(r\) and \(\sigma\). The boundaries are equivalent (in \cite{caiOptimalDetectionSparse2014}, see (21) and Section V.C). While one can straightforwardly obtain the detection boundary through Theorem 1 of \cite{caiOptimalDetectionSparse2014}, we illustrate a typical calculation under the large deviations perspective. For ease of calculation, let us take \(\sigma^2 \neq 1\) without loss of generality. The case of \(\sigma^2 = 1\) is just the Ingster-Donoho-Jin problem. 

    \subsubsection{Checking the tail condition}
    To ensure that we can ultimately apply Corollary \ref{corollary:tight_limit}, we must check that the tail condition (\ref{eqn:varadhan_tail_condition}) is satisfied. Consider that 
    \begin{align*}
        \frac{\log \frac{q_n}{p}(X)}{\log n} &= -\frac{\log \sigma^2}{2\log n} - \frac{\mu_n^2}{2\sigma^2 \log n} + \frac{X^2}{2\log n} \cdot \frac{\sigma^2 - 1}{\sigma^2} + \frac{X\mu_n}{\sigma^2 \log n} \\
        &= -\frac{\log \sigma^2}{2 \log n} - \frac{r}{\sigma^2} + \frac{\sigma^2 - 1}{\sigma^2} \left[\frac{\left(X + \frac{\mu_n}{\sigma^2 - 1}\right)^2}{2 \log n} - \frac{\mu_n^2}{(\sigma^2-1)^2 \cdot 2 \log n}\right] \\
        &= -\frac{\log \sigma^2}{2 \log n} - \frac{r}{\sigma^2-1} + \frac{\sigma^2-1}{\sigma^2} \left(\frac{X + \frac{\mu_n}{\sigma^2-1}}{\sqrt{2 \log n}} \right)^2.
    \end{align*}
    Consider that for any \(\gamma > 1\), we have under the null \(X \sim P\) 
    \begin{align*}
        E\left[\left(\frac{q_n}{p}(X) \right)^\gamma \right] &= \frac{(\sigma^2)^{-\gamma/2}}{\sqrt{2\pi}} \exp\left(- \gamma\frac{\mu_n^2}{2(\sigma^2 - 1)} \right) \cdot \int_{\R} \exp\left(\gamma \frac{\sigma^2 - 1}{\sigma^2} \frac{\left(x + \frac{\mu_n}{\sigma^2-1} \right)^2}{2} - \frac{x^2}{2}\right) \, dx
    \end{align*}
    It's clear that the integral is finite provided that the coefficient of the quadratic term in the exponential is negative. More specifically, the integral is finite if \(\gamma \frac{\sigma^2-1}{\sigma^2} - 1 < 0\). If \(\sigma^2 < 1\), then any \(\gamma > 1\) satisfies the tail condition (\ref{eqn:varadhan_tail_condition}). If \(\sigma^2 > 1\), then any \(1 < \gamma < \frac{\sigma^2}{\sigma^2-1}\)  satisfies the tail condition (\ref{eqn:varadhan_tail_condition}).

    \subsubsection{Finding the rate function}
    To show that \(\frac{\log \frac{q_n}{p}}{\log n}\) satisfies a large deviation principle under the null, we will first obtain a large deviation principle for \(\frac{X + \frac{\mu_n}{\sigma^2-1}}{\sqrt{2 \log n}}\) under the null \(X \sim P\). Then we apply the contraction principle (Theorem \ref{thm:contraction_principle}) and exponential equivalence (Definition \ref{def:exp_equiv}) to deduce the large deviation principle for \(\frac{\log \frac{q_n}{p}}{\log n}\).

    Consider that under the null, 
    \begin{align*}
        \frac{X + \frac{\mu_n}{\sigma^2 - 1}}{\sqrt{2 \log n}} &\sim N\left(\frac{\mu_n}{(\sigma^2-1)\sqrt{2 \log n}}, \frac{1}{2\log n}\right) \\
        &= N\left(\frac{\sqrt{r}}{\sigma^2-1}, \frac{1}{2 \log n} \right)
    \end{align*}
    Therefore, for any Borel set \(\Gamma \subset \R\),
    \begin{align*}
        P\left(\frac{X + \frac{\mu_n}{\sigma^2 - 1}}{\sqrt{2 \log n}} \in \Gamma \right) &= \int_{\Gamma} \frac{\sqrt{2 \log n}}{\sqrt{2\pi}} \exp\left(-(\log n) \cdot \left(x - \frac{\sqrt{r}}{\sigma^2-1} \right)^2 \right) \, dx. 
    \end{align*}
    Applying Lemma 3 of \cite{caiOptimalDetectionHeterogeneous2011} yields 
    \begin{align*}
        \lim_{n \to \infty} \frac{1}{\log n} \cdot \log P\left(\frac{X + \frac{\mu}{\sigma^2 - 1}}{\sqrt{2 \log n}} \in \Gamma \right) = - \inf_{x \in \Gamma} \left(x - \frac{\sqrt{r}}{\sigma^2-1}\right)^2
    \end{align*}
    and so under the null, \(\frac{X + \frac{\mu_n}{\sigma^2 - 1}}{\sqrt{2 \log n}}\) satisfies a large deviation principle with good rate function \(J(t) = \left(t - \frac{\sqrt{r}}{\sigma^2-1}\right)^2\) with respect to speed \(\left\{\frac{1}{\log n}\right\}\). 
    
    Applying the contraction principle (Theorem \ref{thm:contraction_principle}) to the function \(f : \R \to \R\) given by \(f(t) = -\frac{r}{\sigma^2-1} + \frac{\sigma^2-1}{\sigma^2} \cdot t^2\), it follows that the \(- \frac{r}{\sigma^2-1} + \frac{\sigma^2-1}{\sigma^2} \left(\frac{X + \frac{\mu_n}{\sigma^2-1}}{\sqrt{2 \log n}} \right)^2\) satisfies the large deviation principle with respect to speed \(\left\{\frac{1}{\log n}\right\}\) and good rate function \(I : \R \to [0, \infty]\) given by
    \begin{align*}
        I(y) &= \inf\{J(t) : y = f(t)\} \\
        &= \inf\left\{J(t) : \frac{\sigma^2y}{\sigma^2-1} + \frac{\sigma^2r}{(\sigma^2-1)^2} = t^2 \right\} \\
        &= 
        \begin{cases} 
            \left(\sqrt{\frac{\sigma^2y}{\sigma^2-1} + \frac{\sigma^2r}{(\sigma^2-1)^2}} - \frac{\sqrt{r}}{\sigma^2-1} \right)^2 & \text{if } y(\sigma^2 - 1) + r \geq 0, \sigma^2 > 1, \\
            \left(-\sqrt{\frac{\sigma^2y}{\sigma^2-1} + \frac{\sigma^2r}{(\sigma^2-1)^2}} - \frac{\sqrt{r}}{\sigma^2-1} \right)^2 & \text{if } y(\sigma^2 - 1) + r \geq 0, \sigma^2 < 1, \\ 
            \infty & \text{if } y(\sigma^2 - 1) + r < 0
        \end{cases} \\
        &= 
        \begin{cases}
            \left(\frac{\sqrt{\sigma^2(\sigma^2-1)y + \sigma^2 r} - \sqrt{r}}{\sigma^2-1}\right)^2 & \text{if } y(\sigma^2-1) + r \geq 0, \sigma^2 > 1, \\
            \left(\frac{\sqrt{\sigma^2(\sigma^2-1)y + \sigma^2 r} - \sqrt{r}}{\sigma^2-1}\right)^2 & \text{if } y(\sigma^2-1) + r \geq 0, \sigma^2 < 1, \\
            \infty & \text{if } y(\sigma^2-1) + r < 0
        \end{cases} \\
        &= 
        \begin{cases}
            \left(\frac{\sqrt{\sigma^2 (\sigma^2-1)y + \sigma^2 r} - \sqrt{r}}{\sigma^2-1} \right)^2 & \text{if } y(\sigma^2-1) + r \geq 0, \\
            \infty & \text{if } y(\sigma^2-1) + r < 0
        \end{cases} \\
        &= 
        \begin{cases}
            \frac{\sigma^2}{(\sigma^2-1)^2}\left(\sqrt{(\sigma^2-1)y + r} - \sqrt{\frac{r}{\sigma^2}} \right)^2 & \text{if } y(\sigma^2-1) + r \geq 0, \\
            \infty & \text{if } y(\sigma^2-1) + r < 0.
        \end{cases}
    \end{align*}
    Now, observe that for any \(\delta > 0\) we have
    \begin{equation*}
        P\left( \left| \frac{\log \frac{q_n}{p}(X)}{\log n} + \frac{r}{\sigma^2-1} - \frac{\sigma^2-1}{\sigma^2} \left(\frac{X + \frac{\mu_n}{\sigma^2-1}}{\sqrt{2 \log n}} \right)^2 \right| > \delta \right) = P\left( \left| -\frac{\log \sigma^2}{2\log n} \right| > \delta \right) = 0
    \end{equation*}
    for all \(n > \exp\left( \frac{\delta \log \sigma^2}{2}\right)\). Therefore,
    \begin{equation*}
        \limsup_{n \to \infty} \frac{1}{\log n} \log P\left( \left| \frac{\log \frac{q_n}{p}(X)}{\log n} + \frac{r}{\sigma^2-1} - \frac{\sigma^2-1}{\sigma^2} \left(\frac{X + \frac{\mu_n}{\sigma^2-1}}{\sqrt{2 \log n}} \right)^2 \right| > \delta \right) = -\infty.
    \end{equation*}
    Hence, \(\frac{\log \frac{q_n}{p}(X)}{\log n}\) and \(- \frac{r}{\sigma^2-1} + \frac{\sigma^2-1}{\sigma^2} \left(\frac{X + \frac{\mu_n}{\sigma^2-1}}{\sqrt{2 \log n}} \right)^2\) are exponentially equivalent with respect to speed \(\left\{\frac{1}{\log n}\right\}\) (see Definition \ref{def:exp_equiv}). Thus it follows by Theorem \ref{thm:same_ldp} that \(\left\{\frac{\log \frac{q_n}{p}}{\log n}\right\}\) satisfies the large deviation principle under the null with good rate function \(I\).

    \subsubsection{Determining the detection boundary}
    We are now ready to determine the detection boundary. Through the calculation, it will be seen that Corollary \ref{corollary:tight_limit} holds. Abusing notation, we have 
    \begin{align*}
        \beta^*(r, \sigma^2) &= \frac{1}{2} + 0 \vee \sup_{t\geq 0} \left\{t - I(t) + \frac{1 \wedge I(t)}{2} \right\} \\
        &= \frac{1}{2} + 0 \vee \sup_{t\geq 0, I(t) < \infty} \left\{ t - I(t) + \frac{1 \wedge I(t)}{2} \right\}.
    \end{align*}
    For ease, consider the change of variable 
    \begin{equation*}
        s = I(t) = \frac{\sigma^2}{(\sigma^2-1)^2}\left(\sqrt{(\sigma^2-1)t + r} - \sqrt{\frac{r}{\sigma^2}} \right)^2
    \end{equation*}
    for \(t\) such that \(t(\sigma^2 - 1) + r \geq 0\). Rearranging gives
    \begin{align*}
        t &= \frac{1}{\sigma^2-1} \left[\frac{(\sqrt{s}(\sigma^2-1) + \sqrt{r})^2}{\sigma^2} - r\right] \\
        &= \frac{1}{\sigma^2-1} \left[ \frac{s(\sigma^2-1)^2 - r(\sigma^2-1) + 2\sqrt{sr}(\sigma^2-1)}{\sigma^2} \right] \\
        &= s - \frac{s+r - 2\sqrt{sr}}{\sigma^2} \\
        &= s - \frac{(\sqrt{s} - \sqrt{r})^2}{\sigma^2}.
    \end{align*}
    
    To deduce the detection boundary, we consider the two cases \(\sigma^2 > 1\) and \(\sigma^2 < 1\) separately. \newline

    \textbf{Case 1:} Suppose \(\sigma^2 > 1\). By direct calculation, it follows that \(t \geq 0\) if and only if \(s \geq \frac{r}{(\sqrt{\sigma^2} + 1)^2}\). It thus follows that
    \begin{align*}
        \beta^*(r, \sigma^2) = \frac{1}{2} + 0 \vee \sup_{s \geq \frac{r}{(\sqrt{\sigma^2} + 1)^2}} \left\{ \left( s - \frac{(\sqrt{s} - \sqrt{r})^2}{\sigma^2}\right) - s + \frac{1 \wedge s}{2}\right\}.    
    \end{align*}
    Note that the objective function is exactly that considered in Section V.C of \cite{caiOptimalDetectionSparse2014}. We now solve the optimization problem to obtain the detection boundary. First, consider that if \(\sqrt{r} \geq 1+\sqrt{\sigma^2}\), then the optimization problem becomes \(\beta^*(r, \sigma^2) = \frac{1}{2} + 0 \vee \sup_{s \geq \frac{r}{(\sqrt{\sigma^2}+1)^2}} \left\{ - \frac{(\sqrt{s} - \sqrt{r})^2}{\sigma^2} + \frac{1}{2}\right\}\).  Clearly the optimum is achieved at \(s = r\), which yields \(\beta^*(r, \sigma^2) = 1\). On the other hand, let us now consider the case \(\sqrt{r} \leq 1 + \sqrt{\sigma^2}\). Then the optimization problem is given by 
    \begin{equation*}
        \beta^*(r, \sigma^2) = \frac{1}{2} + 0 \vee \left\{ E_1 \vee E_2\right\}
    \end{equation*}
    where 
    \begin{align*}
        E_1 &:= \sup_{\frac{r}{(\sqrt{\sigma^2} + 1)^2} \leq s \leq 1} \left\{ - \frac{(\sqrt{s} - \sqrt{r})^2}{\sigma^2} + \frac{s}{2} \right\}, \\
        E_2 &:= \sup_{1 < s} \left\{ - \frac{(\sqrt{s} - \sqrt{r})^2}{\sigma^2} + \frac{1}{2} \right\}.
    \end{align*}
    We examine each term separately. Looking at \(E_1\) first, let us define \(f(s) = - \frac{(\sqrt{s} - \sqrt{r})^2}{\sigma^2} + \frac{s}{2}\) and note that \(f'(s) = -\frac{1}{\sigma^2} + \frac{\sqrt{r}}{\sigma^2 \sqrt{s}} + \frac{1}{2}\). Hence, it follows that \(f'(s) \geq -\frac{1}{\sigma^2} + \frac{\sqrt{r}}{\sigma^2} + \frac{1}{2} = \frac{2\sqrt{r} + \sigma^2 - 2}{2\sigma^2}\) for all \(\frac{r}{(\sqrt{\sigma^2} - 1)^2} \leq s \leq 1\). Hence, if \(2\sqrt{r} + \sigma^2 > 2\), then \(f'(s) > 0\) and so the maximum is achieved at the right endpoint \(s = 1\). This yields \(E_1 = \frac{1}{2} - \frac{(1-\sqrt{r})^2}{\sigma^2}\). On the other hand, if \(2\sqrt{r} + \sigma^2 \leq 2\), then the maximum is achieved at \(s = \left(\frac{2\sqrt{r}}{2-\sigma^2}\right)^2\). This yields \(E_1 = \frac{r}{2-\sigma^2}\).
    Hence, we've shown that 
    \begin{align*}
        E_1 = 
        \begin{cases}
            \frac{r}{2-\sigma^2} & \text{if } 2\sqrt{r} + \sigma^2 \leq 2,\\
            \frac{1}{2} - \frac{(1-\sqrt{r})^2}{\sigma^2} &\text{if } 2\sqrt{r} + \sigma^2 > 2.
        \end{cases}
    \end{align*}
    Turning our attention to \(E_2\), consider that if \(r \geq 1\), then we immediately have \(E_2 = \frac{1}{2}\). If \(r < 1\), then we have \(E_2 = \frac{1}{2} - \frac{(1-\sqrt{r})^2}{\sigma^2}\). In particular, for any \(r > 0\), we have \(E_2 = \frac{1}{2} - \frac{(1-\sqrt{r})_{+}^2}{\sigma^2}\). By direct comparison, 
    \begin{equation*}
        E_1 \vee E_2 = 
        \begin{cases}
            \frac{r}{2-\sigma^2} & \text{if } 2\sqrt{r} + \sigma^2 \leq 2, \\
            \frac{1}{2} - \frac{(1-\sqrt{r})_{+}^2}{\sigma^2} &\text{if } 2\sqrt{r} + \sigma^2 > 2.
        \end{cases}
    \end{equation*}
    Thus, we've proved that if \(\sqrt{r} \leq 1 + \sqrt{\sigma^2}\), then 
    \begin{equation*}
        \beta^*(r, \sigma^2) = 
        \begin{cases}
            \frac{1}{2} + \frac{r}{2-\sigma^2} & \text{if } 2\sqrt{r} + \sigma^2 \leq 2, \\
            1 - \frac{(1-\sqrt{r})_{+}^2}{\sigma^2} &\text{if } 2\sqrt{r} + \sigma^2 > 2.
        \end{cases}
    \end{equation*}
    Putting this together with our earlier result that \(\beta^*(r, \sigma^2) = 1\) when \(\sqrt{r} > 1 + \sqrt{\sigma^2}\), we have the detection boundary 
    \begin{equation*}
        \beta^*(r, \sigma^2) = 
        \begin{cases}
            \frac{1}{2} + \frac{r}{2-\sigma^2} & \text{if } 2\sqrt{r} + \sigma^2 \leq 2, \\
            1 - \frac{(1-\sqrt{r})_{+}^2}{\sigma^2} &\text{if } 2\sqrt{r} + \sigma^2 > 2
        \end{cases}
    \end{equation*}
    when \(\sigma^2 > 1\). 

    \textbf{Case 2:} Suppose \(\sigma^2 < 1\). It follows that \(t \geq 0\) if and only if \(\sqrt{r} \left(\frac{\sqrt{\sigma^2} - 1}{\sigma^2 - 1}\right) \leq \sqrt{s} \leq \sqrt{r} \left( \frac{1 + \sqrt{\sigma^2}}{1-\sigma^2}\right)\). Noting that \(\frac{1+\sqrt{\sigma^2}}{1-\sigma^2} = \frac{1}{1-\sqrt{\sigma^2}}\) and \(\frac{\sqrt{\sigma^2}-1}{\sigma^2-1} = \frac{1}{\sqrt{\sigma^2} + 1}\), it follows that \(t \geq 0\) if and only if \(\frac{r}{(1+\sqrt{\sigma^2})^2} \leq s \leq \frac{r}{(1-\sqrt{\sigma^2})^2}\). Thus
    \begin{align*}
        \beta^*(r, \sigma^2) = \frac{1}{2} + 0 \vee \sup_{\frac{r}{(1+\sqrt{\sigma^2})^2} \leq s \leq \frac{r}{(1-\sqrt{\sigma^2})^2}} \left\{ \left( s - \frac{(\sqrt{s} - \sqrt{r})^2}{\sigma^2}\right) - s + \frac{1 \wedge s}{2}\right\}.    
    \end{align*}
    As in Case 1, if \(\sqrt{r} > 1+\sqrt{\sigma^2}\), then it immediately follows that the maximum is achieved at \(s = r\) yielding \(\beta^*(r, \sigma^2) = 1\). It should be noted that \(\sqrt{r} > 1 + \sqrt{\sigma^2}\) implies \(2\sqrt{r} + \sigma^2 > 2\). With this case out of the way, let us consider the remaining case \(\sqrt{r} \leq 1 + \sqrt{\sigma^2}\). \newline
    
    Suppose we also have \(\sqrt{r} + \sqrt{\sigma^2} \leq 1\). Then we can write
    \begin{equation*}
        \beta^*(r, \sigma^2) = \frac{1}{2} + 0 \vee \sup_{\frac{r}{(1+\sqrt{\sigma^2})^2} \leq s \leq \frac{r}{(1-\sqrt{\sigma^2})^2}} \left\{ \left(- \frac{(\sqrt{s} - \sqrt{r})^2}{\sigma^2}\right) + \frac{s}{2}\right\}.
    \end{equation*}
    Letting \(f(s) := - \frac{(\sqrt{s} - \sqrt{r})^2}{\sigma^2} + \frac{s}{2}\), observe that \(f'(s) = -\frac{1}{\sigma^2} + \frac{\sqrt{r}}{\sigma^2 \sqrt{s}} + \frac{1}{2}\). Finding the root, we see that the maximum is achieved when \(s = \left(\frac{2 \sqrt{r}}{2-\sigma^2}\right)^2\), which immediately yields \(\beta^*(r, \sigma^2) = \frac{r}{2-\sigma^2}\). Note that the two conditions \(\sqrt{r} + \sqrt{\sigma^2} \leq 1\) and \(\sqrt{r} \leq 1 + \sqrt{\sigma^2}\) are redundant in that only the first \(\sqrt{r} \leq 1 - \sqrt{\sigma^2}\) is binding. Now since we also have \(\sigma^2 < 1\), it follows that \(2\sqrt{r} + \sigma^2 \leq 2\). \newline

    Suppose now that \(\sqrt{r} + \sqrt{\sigma^2} > 1\). Then we have 
    \begin{equation*}
        \beta^*(r, \sigma^2) = \frac{1}{2} + 0 \vee \left\{F_1 \vee F_2\right\}
    \end{equation*}
    where 
    \begin{align*}
        F_1 &= \sup_{\frac{r}{(1+\sqrt{\sigma^2})^2} \leq s \leq 1} \left\{ - \frac{(\sqrt{s} - \sqrt{r})^2}{\sigma^2} + \frac{s}{2} \right\}, \\
        F_2 &= \sup_{1 \leq s \leq \frac{r}{(1-\sqrt{\sigma^2})^2}} \left\{ -\frac{(\sqrt{s} - \sqrt{r})^2}{\sigma^2} + \frac{1}{2} \right\}. 
    \end{align*}
    We examine each term separately. The analysis for \(F_1\) is exactly the same as the analysis for \(E_1\) in Case 1 above. Thus, 
    \begin{align*}
        F_1 =
        \begin{cases}
            \frac{r}{2-\sigma^2} &\text{if } 2\sqrt{r} + \sigma^2 \leq 2, \\
            \frac{1}{2} - \frac{(1-\sqrt{r})^2}{\sigma^2} &\text{if } 2\sqrt{r} + \sigma^2 > 2.
        \end{cases}
    \end{align*}
    Examining \(F_2\), let us set \(f(s) = -\frac{(\sqrt{s} - \sqrt{r})^2}{\sigma^2} + \frac{1}{2}\). Observe that \(f'(s) = -\frac{1}{\sigma^2} + \frac{\sqrt{r}}{\sigma^2 \sqrt{s}}\). Observe that for \(1 \leq s \leq \frac{r}{(1-\sqrt{\sigma^2})^2}\), we have \(f'(s) \leq -\frac{1}{\sigma^2} + \frac{\sqrt{r}}{\sigma^2}\). Therefore, if \(r \leq 1\), then \(f'(s) \leq 0\) for all \(1 \leq s \leq \frac{r}{(1-\sqrt{\sigma^2})^2}\) and so the maximum is achieved at the left endpoint \(s = 1\), yielding \(f(1) = \frac{1}{2} - \frac{(1-\sqrt{r})^2}{\sigma^2}\). If \(r \geq 1\), then the maximum is achieved at \(s = r\), which yields \(F_2 = \frac{1}{2}\). Consequently, we have \(F_2 = \frac{1}{2} - \frac{(1-\sqrt{r})_{+}^2}{\sigma^2}\) for all \(r > 0\). Therefore, as in Case 1, we have by direct comparison 
    \begin{align*}
        F_1 \vee F_2 = 
        \begin{cases}
            \frac{r}{2-\sigma^2} &\text{if } 2 \sqrt{r} + \sigma^2 \leq 2, \\
            \frac{1}{2} - \frac{(1-\sqrt{r})_{+}^2}{\sigma^2} &\text{if } 2 \sqrt{r} + \sigma^2 > 2.
        \end{cases}
    \end{align*}
    Hence, we've shown that if \(\sqrt{r} \leq 1 + \sqrt{\sigma^2}\) and \(\sqrt{r} + \sqrt{\sigma^2} > 1\), then 
    \begin{equation*}
        \beta^*(r, \sigma^2) = 
        \begin{cases}
            \frac{1}{2} + \frac{r}{2-\sigma^2} &\text{if } 2 \sqrt{r} + \sigma^2 \leq 2, \\
            1 - \frac{(1-\sqrt{r})_{+}^2}{\sigma^2} &\text{if } 2 \sqrt{r} + \sigma^2 > 2.
        \end{cases}
    \end{equation*}

    Assembling all of the pieces together, we have shown that if \(\sigma^2 < 1\), then 
    \begin{equation*}
        \beta^*(r, \sigma^2) = 
        \begin{cases}
            \frac{1}{2} + \frac{r}{2-\sigma^2} &\text{if } 2 \sqrt{r} + \sigma^2 \leq 2, \\
            1 - \frac{(1-\sqrt{r})_{+}^2}{\sigma^2} &\text{if } 2 \sqrt{r} + \sigma^2 > 2.
        \end{cases}
    \end{equation*}

    Thus, we've shown in both cases, (i.e. \(\sigma^2 > 1\) or \(\sigma^2 < 1\)), the detection boundary is given by
    \begin{equation*}
        \beta^*(r, \sigma^2) = 
        \begin{cases}
            \frac{1}{2} + \frac{r}{2-\sigma^2} & \text{if } 2\sqrt{r} + \sigma^2 \leq 2, \\
            1 - \frac{(1-\sqrt{r})_{+}^2}{\sigma^2} & \text{if } 2\sqrt{r} + \sigma^2 > 2.
        \end{cases}
    \end{equation*}
    Thus we've exactly recovered the detection boundary obtained by Cai, Jeng, and Jin \cite{caiOptimalDetectionHeterogeneous2011}. 
    
    \subsubsection{Higher Criticism}
    It can be directly checked that the rate function \(I\) is indeed convex and that the other conditions of Theorem \ref{thm:HC_tight} hold. Hence, the test \(\psi_{\HC^*_n}\) defined in (\ref{eqn:HC_test}) achieves the detection boundary.  

    \subsection{Mixture of a mixture I}
    A sparse mixture detection problem making more use of the multivariate setting is the following. Consider the testing problem (\ref{problem:sparse_mixture_detection_1})-(\ref{problem:sparse_mixture_detection_2}) with \(P_n = P = N(0, I_d)\) and \(Q_n = \frac{1}{2} N(\mu_1, I_d) + \frac{1}{2} N(\mu_2, I_d)\) where \(\mu_1 = \sqrt{2r\log n} \cdot u_1\) and \(\mu_2 = \sqrt{2r\log n} \cdot u_2\) where \(u_1, u_2\) are fixed and linearly independent unit vectors in \(\R^d\). Note that the tail condition (\ref{eqn:varadhan_tail_condition}) is easily verified. 

    \subsubsection{Finding the rate function}
    To deduce the detection boundary, we directly establish a large deviation principle and calculate the associated rate function. Consider that under the null,
    \begin{align*}
        \frac{\log \frac{q_n}{p}(X)}{\log n} &= -r - \frac{\log 2}{\log n} + \frac{\log \left(\exp\left(\langle X, \mu_1 \rangle \right) + \exp\left(\langle X, \mu_2 \rangle \right) \right)}{\log n} \\
        &= -r - \frac{\log 2}{\log n} + \frac{\langle X, \mu_1\rangle \vee \langle X, \mu_2 \rangle}{\log n} + \frac{\log \left(1 + \frac{\min_{i \in \{1, 2\}} \exp\left(\langle X, \mu_i \rangle\right)}{\max_{i \in \{1, 2\}} \exp\left(\langle X, \mu_i \rangle\right)} \right)}{\log n} \\
        &= -r + \left\langle X, \frac{\mu_1}{\log n} \right\rangle \vee \left\langle X, \frac{\mu_2}{\log n} \right\rangle + o_P(1).
    \end{align*}
    To deduce a large deviation principle, we will first deduce a large deviation principle for \(\left\langle X, \frac{\mu_1}{\log n} \right\rangle \vee \left\langle X, \frac{\mu_2}{\log n} \right\rangle\), then argue by contraction principle (see Theorem \ref{thm:contraction_principle}), exponential equivalence (see Definition \ref{def:exp_equiv}), and Theorem \ref{thm:same_ldp} to obtain a large deviation principle for the sequence of normalized log likelihood ratios.
    Under the null,
    \begin{equation*}
        Z_n := \left(\begin{matrix} \langle X, \frac{\mu_1}{\log n}\rangle \\ \langle X, \frac{\mu_2}{\log n}\rangle \end{matrix}\right) \sim N\left(0, \Sigma \right)
    \end{equation*}
    where 
    \begin{equation*}
        \Sigma = \frac{2r}{\log n} \left(\begin{matrix} 1 & \langle u_1, u_2\rangle \\ \langle u_1, u_2 \rangle & 1 \end{matrix}\right).
    \end{equation*}
    Now consider that for any Borel set \(\Gamma \subset \R^2\),
    \begin{equation*}
        \frac{\log P(Z_n \in \Gamma)}{\log n} = \frac{1}{\log n} \log \int_{\Gamma} \frac{\log n}{2\pi \cdot 2r \sqrt{1 - \langle u_1, u_2\rangle^2}} \exp\left(- \frac{z_1^2 + z_2^2 - 2z_1 z_2 \langle u_1, u_2 \rangle}{4r(1-\langle u_1, u_2 \rangle^2)} \cdot \log n \right) \, dz.
    \end{equation*}
    Applying Lemma 3 of \cite{caiOptimalDetectionSparse2014}, we immediately see that the \(Z_n\) satisfies the large deviation principle with respect to speed \(\left\{\frac{1}{\log n}\right\}\) and good rate function \(J : \R^2 \to [0, \infty]\) given by
    \begin{equation*}
        J(t) = \frac{t_1^2 + t_2^2 - 2t_1t_2 \langle u_1, u_2 \rangle}{4r(1-\langle u_1, u_2\rangle^2)}.
    \end{equation*}
    Consider that the function \(f : \R^2 \to \R\) with \(f(v,w) = v \vee w\) is continuous. Thus, by applying the contraction principle (Theorem \ref{thm:contraction_principle}) to \(f(Z_n)\), it follows that \(f(Z_n)\) satisfies the large deviation principle with respect to speed \(\left\{\frac{1}{\log n}\right\}\) and good rate function 
    \begin{align*}
        K(y) &= \inf\left\{J(t) : y = f(t) \right\} \\
        &= \inf \left\{ \frac{t_1^2 + t_2^2 - 2t_1t_2 \langle u_1, u_2 \rangle}{4r(1-\langle u_1, u_2 \rangle^2)} : y = t_1 \vee t_2 \right\} \\
        &= \inf\left\{\frac{y^2 + (t_1 \wedge t_2)^2 - 2y(t_1 \wedge t_2)\langle u_1, u_2 \rangle}{4r(1-\langle u_1, u_2 \rangle^2)} : y = t_1 \vee t_2 \right\} \\
        &= \inf\left\{\frac{y^2 + (t_1 \wedge t_2 - y \langle u_1, u_2 \rangle)^2 - y^2 \langle u_1, u_2 \rangle^2}{4r(1-\langle u_1, u_2 \rangle^2)} : y = t_1 \vee t_2 \right\} \\
        &= \inf\left\{\frac{y^2}{4r} + \frac{(t_1 \wedge t_2 - y \langle u_1, u_2 \rangle)^2}{4r(1-\langle u_1, u_2 \rangle^2)} : y = t_1 \vee t_2 \right\} \\
        &=
        \begin{cases}
            \frac{y^2}{4r} & \text{if } y \geq 0, \\
            \frac{y^2}{2r(1+\langle u_1, u_2 \rangle)} &\text{if } y < 0.
        \end{cases}
    \end{align*}
    Applying the contraction principle once again, it follows that \(f(Z_n) - r\) satisfies the large deviation principle with respect to speed \(\left\{\frac{1}{\log n}\right\}\) and good rate function 
    \begin{align*}
        I(y) &= 
        \begin{cases}
            \frac{(y+r)^2}{4r} & \text{if } y \geq -r, \\
            \frac{(y+r)^2}{2r(1+\langle u_1, u_2 \rangle)} &\text{if } y < -r.
        \end{cases}
    \end{align*}
    Now, consider that for any \(\delta > 0\), it follows that 
    \begin{align*}
        \frac{\log P\left( \left| \frac{\log \frac{q_n}{p}(X)}{\log n} - (f(Z_n) - r)\right| > \delta \right)}{\log n} &= \frac{\log P\left( \left| -\frac{\log 2}{\log n} + \frac{\log\left(1 + \frac{\min_{i \in \{1, 2\}} \exp\left(\langle X, \mu_i \rangle \right)}{\max_{i \in \{1, 2\}} \exp\left(\langle X, \mu_i \rangle\right) } \right)}{\log n}\right| > \delta \right)}{\log n}.
    \end{align*}
    Consider that 
    \begin{equation*}
        0 \leq \frac{\log\left(1 + \frac{\min_{i \in \{1, 2\}} \exp\left(\langle X, \mu_i \rangle\right) }{\max_{i \in \{1, 2\}} \exp\left(\langle X, \mu_i \rangle\right) } \right)}{\log n} \leq \frac{\log 2}{\log n}
    \end{equation*}
    almost surely. Hence, 
    \begin{equation*}
        P\left( \left| -\frac{\log 2}{\log n} + \frac{\log\left(1 + \frac{\min_{i \in \{1, 2\}} \exp\left(\langle X, \mu_i \rangle \right)}{\max_{i \in \{1, 2\}} \exp\left(\langle X, \mu_i \rangle\right) } \right)}{\log n}\right| > \delta \right) = 0
    \end{equation*}
    for all \(n\) sufficiently large, and so 
    \begin{equation*}
        \lim_{n \to \infty} \frac{\log P\left( \left| \frac{\log \frac{q_n}{p}(X)}{\log n} - (f(Z_n) - r)\right| > \delta \right)}{\log n} = -\infty.
    \end{equation*}
    Therefore, \(\frac{\log \frac{q_n}{p}(X)}{\log n}\) and \(f(Z_n) - r\) are exponentially equivalent with respect to speed \(\left\{\frac{1}{\log n}\right\}\). By Theorem \ref{thm:same_ldp}, it follows that \(\left\{\frac{\log \frac{q_n}{p}}{\log n}\right\}\) satisfies the large deviation principle under the null with good rate function \(I\). 
    
    \subsubsection{Determining the detection boundary}
    Given the form of \(I\), it immediately follows that the detection boundary is given by the Ingster-Donoho-Jin boundary, namely
    \begin{equation*}
        \beta^*(r) = 
        \begin{cases}
            \frac{1}{2} + r & \text{if } r \leq \frac{1}{4}, \\
            1 - (1-\sqrt{r})_{+}^2 & \text{if } r > \frac{1}{4}.
        \end{cases}
    \end{equation*}

    \subsubsection{Higher Criticism}
    Theorem \ref{thm:HC_tight} guarantees the optimality of \(\psi_{\HC^{*}_n}\). Writing out the testing statistic \(\HC_n^*\) reveals that knowledge of both \(u_1, u_2\) is needed to construct \(\HC^*_n\).

    \subsection{Mixture of a mixture II}
    The following sparse mixture detection problem is inspired by the testing equivalence of clustering problem considered by Gao and Ma (Section 2.3 in \cite{gaoTestingEquivalenceClustering2019}). Let \(u, v \in \R^d\) be orthogonal unit vectors, i.e. \(||u||=||v||=1\) and \(\langle u, v \rangle = 0\). Let \(\mu_n = \sqrt{2r\log n} \cdot u\) and \(\nu_n = \sqrt{2r\log n} \cdot v\) with fixed \(0 < r \leq 1\). Consider the testing problem (\ref{problem:sparse_mixture_detection_1})-(\ref{problem:sparse_mixture_detection_2}) with 
    \begin{align*}
        P_n &= \frac{1}{2} N(\mu_n, I_d) + \frac{1}{2} N(-\mu_n, I_d), \\
        Q_n &= \frac{1}{2}N(\nu_n, I_d) + \frac{1}{2}N(-\nu_n, I_d).
    \end{align*}
    In this subsection, we will show that the normalized log-likelihood ratio satisfies the large deviation principle under the null and that the testing statistic \(\HC_n^*\) furnishes an optimal test. 
    
    \subsubsection{Checking the tail condition} 
    First, we show that the tail condition (\ref{eqn:varadhan_tail_condition}) is satisfied. Consider that for any \(\gamma > 1\), 
    \begin{align*}
        \left(\frac{q_n}{p_n}(x) \right)^\gamma &= \frac{\left(\exp\left(-\frac{||x-\nu_n||^2}{2} \right) + \exp\left(-\frac{||x+\nu_n||^2}{2}\right)\right)^\gamma}{\left(\exp\left(-\frac{||x-\mu_n||^2}{2} \right) + \exp\left(-\frac{||x+\mu_n||^2}{2}\right)\right)^\gamma} \\
        &= \frac{\left(\exp\left(\langle x, -\nu_n\rangle \right) + \exp\left(\langle x, \nu_n\rangle \right)\right)^\gamma}{\left(\exp\left(\langle x, -\mu_n\rangle\right) + \exp\left(\langle x, \mu_n\rangle\right)\right)^\gamma} \\
        &= \frac{\left(\exp\left(\langle x, -\nu_n\rangle \right) + \exp\left(\langle x, \nu_n\rangle \right)\right)^\gamma}{\left(\exp\left(|\langle x, \mu_n\rangle|\right) + \exp\left(-|\langle x, \mu_n\rangle|\right)\right)^\gamma} \\
        &\leq \left(\exp\left(\langle x, -\nu_n\rangle \right) + \exp\left(\langle x, \nu_n\rangle \right)\right)^\gamma \\
        &\leq 2^{\gamma - 1} \left(\exp\left(-\gamma \langle x, \nu_n\rangle \right) + \exp\left( \gamma \langle x, \nu_n \rangle \right) \right) \\
        &= 2^{\gamma - 1}\left(\exp\left(-\gamma\sqrt{2r\log n} \langle x, v\rangle \right) + \exp\left(\gamma \sqrt{2r\log n} \langle x, v\rangle \right)\right).
    \end{align*}
    Using the moment generating function for multivariate normal random variables, we have under the null
    \begin{align*}
        & E\left[\left(\frac{q_n}{p_n}(X)\right)^\gamma \right] \\
        &\leq 2^{\gamma - 1} \left(E\left[\exp\left(-\gamma \sqrt{2r\log n} \langle X, v\rangle \right)\right] + E\left[\exp\left(\gamma \sqrt{2r\log n} \langle X, v\rangle\right) \right]\right) \\
        &= 2^{\gamma - 1}\left[ \frac{1}{2} \exp\left(-\gamma \sqrt{2r\log n} \langle \mu_n, v\rangle + \gamma^2 r \log n \right) + \frac{1}{2} \exp\left(-\gamma \sqrt{2r \log n} \langle -\mu_n, v\rangle + \gamma^2 r \log n\right) \right. \\
        &\;\; \left. + \frac{1}{2} \exp\left(\gamma \sqrt{2r \log n} \langle \mu_n, v\rangle + \gamma^2 r \log n \right) + \frac{1}{2} \exp\left(\gamma \sqrt{2r\log n} \langle -\mu_n, v\rangle + \gamma^2 r \log n \right)\right] \\
        &= 2^{\gamma} \exp\left(\gamma^2 r \log n\right).
    \end{align*}
    It immediately follows that 
    \begin{equation*}
        \limsup_{n \to \infty} \frac{1}{\log n} \log E\left[\left(\frac{q_n}{p_n}(X)\right)^\gamma \right] < \infty
    \end{equation*}
    and so the tail condition (\ref{eqn:varadhan_tail_condition}) is satisfied.
    
    \subsubsection{Finding the rate function}
    With the tail condition established, we now show that the large deviation principle under the null is satisfied. Observe that 
    \begin{align*}
        \log \frac{q_n}{p_n}(x) &= \log\left(\frac{\exp\left(-\frac{||x-\nu_n||^2}{2}\right) + \exp\left(-\frac{||x+\nu_n||^2}{2}\right)}{\exp\left(-\frac{||x-\mu_n||^2}{2}\right) + \exp\left(-\frac{||x+\mu_n||^2}{2}\right)} \right) \\
        &= \log\left(\frac{\exp\left(\langle x, -\nu_n\rangle \right) + \exp\left(\langle x, \nu_n\rangle \right)}{\exp\left(\langle x, -\mu_n\rangle\right) + \exp\left(-\langle x, \mu_n\rangle\right)} \right) \\
        &= \log\left(\frac{\exp\left(|\langle x, \nu_n\rangle| \right) + \exp\left(-|\langle x, \nu_n\rangle| \right)}{\exp\left(|\langle x, \mu_n\rangle|\right) + \exp\left(-|\langle x, \mu_n\rangle|\right)} \right) \\
        &= |\langle x, \nu_n\rangle| + \log\left(1 + \exp\left(-2|\langle x, \nu_n\rangle|\right)\right) - |\langle x, \mu_n \rangle| - \log\left(1 + \exp\left(-2|\langle x, \mu_n\rangle|\right)\right).
    \end{align*}
    Thus, under the null 
    \begin{align*}
        \frac{\log \frac{q_n}{p_n}(X)}{\log n} &= \frac{\sqrt{2r}}{\sqrt{\log n}} \left(|\langle X, v \rangle| - |\langle X, u \rangle|\right) + \frac{\log\left(\frac{1 + \exp(-2 |\langle X, v\rangle| \cdot \sqrt{2r \log n})}{1 + \exp(-2|\langle X, u \rangle| \cdot \sqrt{2r\log n})} \right)}{\log n}.
    \end{align*}
    To establish a large deviation principle, we will use the contraction principle along with exponential equivalence. First, consider that under the null 
    \begin{align*}
        \frac{\sqrt{2r}}{\sqrt{\log n}} 
        \left(\begin{matrix} \langle X, v \rangle \\ \langle X, u\rangle  \end{matrix} \right) \sim \frac{1}{2}N\left( \left(\begin{matrix} 0 \\ 2r\end{matrix} \right), \frac{2r}{\log n} I_2\right) + \frac{1}{2}N\left( \left(\begin{matrix} 0 \\ -2r \end{matrix} \right), \frac{2r}{\log n} I_2\right).
    \end{align*}
    Observe that the density is given by \(g : \R^2 \to \R\) with 
    \begin{align*}
        g(y) &= \frac{1}{2} \frac{\log n}{4\pi r} \exp\left(-\frac{y_1^2 + (y_2 - 2r)^2}{4r} \cdot \log n \right) + \frac{1}{2} \frac{\log n}{4\pi r} \exp\left(-\frac{y_1^2 + (y_2+2r)^2}{4r} \cdot \log n \right) \\
        &= \frac{1}{2} \frac{\log n}{4\pi r} \exp\left(-\frac{y_1^2}{4r} \log n\right) \exp\left(- \frac{(y_2+2r)^2 \wedge (y_2-2r)^2}{4r} \cdot \log n \right) \\
        &\;\; \cdot \left[1 + \exp\left( - \frac{\left[(y_2+2r)^2 \vee (y_2-2r)^2 \right] - \left[(y_2 + 2r)^2 \wedge (y_2-2r)^2\right]}{4r} \log n\right) \right].
    \end{align*}
    An application of Lemma 3 in \cite{caiOptimalDetectionSparse2014} implies that, under the null, \(\frac{\sqrt{2r}}{\sqrt{\log n}} \left(\langle X, v \rangle, \langle X, u\rangle \right)\) satisfies a large deviation principle with respect to speed \(\left\{\frac{1}{\log n} \right\}\) and good rate function \(J : \R^2 \to [0, \infty]\) given by
    \begin{equation*}
        J(y) = \frac{y_1^2 + (y_2 + 2r)^2 \wedge (y_2 - 2r)^2}{4r}.
    \end{equation*}
    Consider the function \(f : \R^2 \to \R\) given by \(f(a, b) = |a| - |b|\). Note that \(f\) is continuous, and so we can apply the contraction principle to derive a large deviation principle for \(\frac{\sqrt{2r}}{\sqrt{\log n}} \left(|\langle X, v\rangle| - |\langle X, u \rangle| \right)\). By the contraction principle (Theorem \ref{thm:contraction_principle}), it follows that the corresponding rate function is \(I : \R \to [0, \infty]\) given by 
    \begin{align*}
        I(t) = \inf\{J(y) : t = f(y)\}. 
    \end{align*}
    For \(t = f(y) = |y_1| - |y_2|\), we have
    \begin{align*}
        J(y) &= \frac{y_1^2 + (y_2+2r)^2 \wedge (y_2-2r)^2}{4r} \\
        &= \frac{y_1^2 + (y_2^2 + 4ry_2 + 4r^2) \wedge (y_2^2 - 4ry_2 + 4r^2)}{4r} \\
        &= \frac{y_1^2 + y_2^2 + 4r^2 + 4r\left[(-y_2)\wedge(y_2)\right]}{4r} \\
        &= \frac{y_1^2 + y_2^2 + 4r^2 - 4r|y_2|}{4r} \\
        &= \frac{y_1^2 + y_1^2 - 2|y_1|t + t^2 + 4r^2 - 4r|y_1| + 4rt}{4r} \\
        &= \frac{(t+2r)^2}{4r} + \frac{2y_1^2 - 2(t+2r)|y_1|}{4r} \\
        &= \frac{(t+2r)^2}{4r} + \frac{y_1^2 - (t+2r)|y_1|}{2r}.
    \end{align*}
    Since \(t + |y_2| = |y_1|\) and \(|y_2|, |y_1| \geq 0\), we have the constraint that \(|y_1| \geq t\). So for \(t \in \R\), we have
    \begin{align*}
        I(t) &= \inf\left\{J(y) : t = f(y) \right\} \\
        &= \inf\left\{ \frac{(t+2r)^2}{4r} + \frac{y_1^2 - (t+2r)|y_1|}{2r} : |y_1| \geq t \right\} \\
        &= \begin{cases} \frac{(t+2r)^2}{4r}& \text{if } t < -2r, \\ \frac{(t+2r)^2}{8r} & \text{if } |t| \leq 2r, \\  \frac{(t+2r)^2}{4r} - t & \text{if } t > 2r.\end{cases}
    \end{align*}
    Thus, we've showed that \(\frac{\sqrt{2r}}{\sqrt{\log n}} (|\langle X, v\rangle| - |\langle X, u \rangle|)\) satisfies a large deviation principle with respect to speed \(\left\{\frac{1}{\log n}\right\}\) and good rate function \(I\). Now, consider that for every \(\delta > 0\), we have under the null
    \begin{align*}
        &\lim_{n \to \infty} \frac{\log P\left(\left| \frac{\log \frac{q_n}{p_n}(X)}{\log n} - \frac{\sqrt{2r}}{\sqrt{\log n}} (|\langle X, v \rangle| - |\langle X, u \rangle|) \right| > \delta \right)}{\log n} \\
        &= \lim_{n \to \infty} \frac{\log P\left(\left|\frac{\log\left(\frac{1 + \exp(-2 |\langle X, v\rangle| \cdot \sqrt{2r \log n})}{1 + \exp(-2|\langle X, u \rangle| \cdot \sqrt{2r\log n})} \right)}{\log n}\right| > \delta \right)}{\log n} \\
        &\leq \lim_{n \to \infty} \frac{\log P\left(\frac{2 \log 2}{\log n} > \delta \right)}{\log n}. 
    \end{align*}
    Clearly \(\frac{2 \log 2}{\log n} \leq \delta\) for all sufficiently large \(n\), and so it immediately follows that \(\lim_{n \to \infty} \frac{\log P\left(\frac{2 \log 2}{\log n} > \delta \right)}{\log n} = -\infty\), yielding
    \begin{equation*}
        \lim_{n \to \infty} \frac{\log P\left(\left| \frac{\log \frac{q_n}{p_n}(X)}{\log n} - \frac{\sqrt{2r}}{\sqrt{\log n}} (|\langle X, v \rangle| - |\langle X, u \rangle|) \right| > \delta \right)}{\log n} = -\infty.
    \end{equation*}
    Hence, under the null \(\frac{\log \frac{q_n}{p_n}(X)}{\log n}\) and \(\frac{\sqrt{2r}}{\sqrt{\log n}}\) are exponentially equivalent (see Definition \ref{def:exp_equiv}) with respect to speed \(\left\{\frac{1}{\log n}\right\}\). By Theorem \ref{thm:same_ldp}, it follows that \(\left\{\frac{\log \frac{q_n}{p_n}}{\log n} \right\}\) satisfies the large deviation principle under the null with good rate function \(I\). 

    \subsubsection{Determining the detection boundary}
    Since the tail condition (\ref{eqn:varadhan_tail_condition}) is satisfied and the conditions of Corollary \ref{corollary:nice_tight_limit} can be straight-forwardly checked to hold, it follows that the detection boundary is given by 
    \begin{equation*}
        \beta^*(r) = \frac{1}{2} + 0 \vee \sup_{t \geq 0} \left\{t - I(t) + \frac{1 \wedge I(t)}{2} \right\}.
    \end{equation*}
    We now explicitly evaluate the detection boundary. Due to the structure of \(I\), we must break the optimization problem into two pieces, namely 
    \begin{align*}
        & \sup_{t \geq 0} \left\{t - I(t) + \frac{1 \wedge I(t)}{2} \right\} \\
        &= \left[ \sup_{t \in [0, 2r]} \left\{t - \frac{(t+2r)^2}{8r} + \frac{1 \wedge \left(\frac{(t+2r)^2}{8r}\right)}{2} \right\} \right] \vee \left[ \sup_{t > 2r} \left\{t - \left(\frac{(t+2r)^2}{4r} - t\right) + \frac{1 \wedge \left(\frac{(t+2r)^2}{4r} - t\right)}{2} \right\} \right] \\
        &=: A \vee B.
    \end{align*}
    We examine each piece separately. We first examine \(A\). Consider that \(\frac{(t+2r)^2}{8r} \leq 1\) for all \(t \in [0, 2r]\) if and only if \(r \leq \frac{1}{2}\). Therefore, for \(r \leq \frac{1}{2}\) we have
    \begin{align*}
        A &= \sup_{t \in [0, 2r]} \left\{t - \frac{(t+2r)^2}{8r} + \frac{(t+2r)^2}{16r} \right\} \\
        &= \sup_{t \in [0, 2r]} \left\{t - \frac{(t+2r)^2}{16r} \right\} \\
        &= r.
    \end{align*}
    Alternatively, if \(\frac{1}{2} < r\), then 
    \begin{align*}
        A &= \left[\sup_{t \in [0, \sqrt{8r}-2r]} \left\{t - \frac{(t+2r)^2}{16r}\right\}\right] \vee \left[\sup_{t \in [\sqrt{8r}-2r, 2r]} \left\{t - \frac{(t+2r)^2}{8r} + \frac{1}{2} \right\}\right] \\
        &=: a_1 \vee a_2.
    \end{align*}
    Observe that since \(r > \frac{1}{2}\), we have \(a_1 = \sqrt{8r}-2r-\frac{1}{2}\) and \(a_2 = \frac{1}{2}\). Therefore, \(a_1 \vee a_2 = a_2\), and so \(A = \frac{1}{2}\). In particular, it has been shown that 
    \begin{equation*}
        A = \begin{cases}
            r &\text{if } r \leq \frac{1}{2}, \\
            \frac{1}{2} & \text{if } r > \frac{1}{2}.
        \end{cases}
    \end{equation*}
    We now turn our attention to \(B\). Note that \( \frac{(t+2r)^2}{4r} - t \leq 1\) if and only if \(t \leq 2\sqrt{r-r^2}\). Thus, we can write 
    \begin{align*}
        B &= \left[\sup_{t \in [2r, 2\sqrt{r-r^2}]} \left\{t - \left(\frac{(t+2r)^2}{4r} - t \right) + \frac{\frac{(t+2r)^2}{4r} - t}{2} \right\} \right] \vee \left[\sup_{t > 2\sqrt{r-r^2}} \left\{t - \left(\frac{(t+2r)^2}{4r} - t \right) + \frac{1}{2} \right\} \right] \\
        &=: b_1 \vee b_2.
    \end{align*}
    where we use the convention that a supremum of a function over an empty set results in \(-\infty\). Observe that the interval \([2r, 2\sqrt{r-r^2}]\) is nonempty if and only if \(r \leq \frac{1}{2}\). A direct calculation reveals
    \begin{align*}
        b_1 &= \sup_{t \in [2r, 2\sqrt{r-r^2}]} \left\{ \frac{3}{2}t - \frac{(t+2r)^2}{8r} \right\} \\
        &= \begin{cases} \frac{3}{2}r & \text{if } r \leq \frac{1}{5}, \\
        2\sqrt{r-r^2} - \frac{1}{2} &\text{if } \frac{1}{5} < r \leq \frac{1}{2}, \\
        -\infty & \text{if } \frac{1}{2} < r. \end{cases}
    \end{align*}
    Examining the quantity \(b_2\), it follows that
    \begin{align*}
        b_2 &= \sup_{t > 2\sqrt{r-r^2}} \left\{ 2t - \frac{(t+2r)^2}{4r} + \frac{1}{2} \right\} \\
        &=
        \begin{cases}
            2\sqrt{r-r^2} - \frac{1}{2} & \text{if } r \leq \frac{1}{2}, \\
            \frac{1}{2} &\text{if } r > \frac{1}{2}.
        \end{cases}
    \end{align*}
    With these calculations in hand, it immediately follows that 
    \begin{align*}
        B = 
        \begin{cases}
            \frac{3}{2} r &\text{if } r \leq \frac{1}{5}, \\
            2\sqrt{r-r^2} - \frac{1}{2} &\text{if } \frac{1}{5} < r \leq \frac{1}{2}, \\
            \frac{1}{2} &\text{if } r > \frac{1}{2}.
        \end{cases}
    \end{align*}
    A direct comparison shows
    \begin{equation*}
        A \vee B = 
        \begin{cases}
            \frac{3}{2} r &\text{if } r \leq \frac{1}{5}, \\
            2\sqrt{r-r^2} - \frac{1}{2} &\text{if } \frac{1}{5} < r \leq \frac{1}{2}, \\
            \frac{1}{2} &\text{if } r > \frac{1}{2}.
        \end{cases}
    \end{equation*}
    Therefore, reexpressing terms into an equivalent yet more evocative form yields the detection boundary
    \begin{equation*}
        \beta^*(r) = \frac{1}{2} + 0 \vee (A \vee B) = 
        \begin{cases}
            \frac{3}{2}r + \frac{1}{2} & \text{if } r \leq \frac{1}{5}, \\
            \sqrt{1-(1-2r)_{+}^2} &\text{if } r > \frac{1}{5}.
        \end{cases}
    \end{equation*}
    We have exactly recovered the detection boundary stated in Theorem 2.3 of \cite{gaoTestingEquivalenceClustering2019}. In fact, this detection boundary holds in a more general setting than that originally considered by Gao and Ma \cite{gaoTestingEquivalenceClustering2019}.

    \subsubsection{Higher Criticism}
    Examining the rate function, it can directly seen that \(I\) is convex as \(I\) is differentiable and \(I'\) is a monotone increasing function. Furthermore, it is straight-forward to check that the other conditions of Theorem \ref{thm:HC_tight} hold, and so the test \(\psi_{\HC^*_n}\) given in (\ref{eqn:HC_test}) achieves the detection boundary.   
    
    \subsection{Detection of a low-rank perturbation}\label{example:low_rank_perturbation}
    Consider the testing problem (\ref{problem:sparse_mixture_detection_1})-(\ref{problem:sparse_mixture_detection_2}) with calibration (\ref{eqn:beta_sparsity}) and distributions \(P_n = N(0, I_p), Q_n = N(0, I_p + H)\) where \(H\) is a rank \(k < p\) symmetric matrix with its \(k\) nonzero eigenvalues equal to  \(r > 0\). Note that we can write \(H = Q \cdot r A_k \cdot Q^\intercal\) where \(Q \in \R^{p \times p}\) is an orthogonal matrix and \(A_k\) is a diagonal matrix with the first \(k\) entries on the diagonal equal to one and the remaining diagonal entries equal to zero.

    \subsubsection{Checking the tail condition}
    We first check that the tail condition \ref{eqn:varadhan_tail_condition} holds. Consider that 
    \begin{align*}
        p_n(x) &= (2\pi)^{-p/2} \exp\left(-\frac{||x||^2}{2} \right), \\
        q_n(x) &= (2\pi)^{-p/2} (1+r)^{-k/2} \exp\left(-\frac{1}{2} \langle x, (I_p+H)^{-1} x\rangle \right) \\
        &= (2\pi)^{-p/2} (1+r)^{-k/2} \exp\left(-\frac{1}{2} \left\langle Q^\intercal x, \left(I_p - \frac{r}{r+1} A_k \right)Q^\intercal x \right\rangle \right).
    \end{align*}
    Therefore, it follows that 
    \begin{align*}
        \frac{q_n}{p_n}(x) &= (1+r)^{-k/2} \exp\left(\frac{1}{2} \cdot \frac{r}{r+1} \langle Q^\intercal x, A_k Q^\intercal x \rangle \right).
    \end{align*}
    Note that for \(1 < \gamma < \frac{r+1}{r}\), we have for \(X_n \sim P_n\)
    \begin{align*}
        E\left[\left( \frac{q_n}{p_n}(X_n) \right)^\gamma \right] &= (r+1)^{-\gamma k/2} (2\pi)^{-p/2} \int_{\R^p} \exp\left( \frac{\gamma r}{r+1} \cdot \frac{\langle Q^\intercal x, A_k Q^\intercal x \rangle }{2} - \frac{||x||^2}{2} \right) \, dx \\
        &\leq  (r+1)^{-\gamma k/2} (2\pi)^{-p/2} \int_{\R^p} \exp\left( \left[\frac{\gamma r}{r+1} - 1 \right] \cdot \frac{||x||^2}{2} \right) \, dx \\
        & < \infty
    \end{align*}
    where the final inequality follows from the fact that \(1 < \gamma < \frac{r+1}{r}\) implies \(\frac{\gamma r}{r+1} - 1 < 0\). Hence, the tail condition (\ref{eqn:varadhan_tail_condition}) holds. 

    \subsubsection{Finding the rate function}
    We now establish that the normalized log likelihood ratios \(\left\{\frac{\log \frac{q_n}{p_n}}{\log n}\right\}\) satisfy the large deviation principle under the null and we identify the corresponding rate function. Consider that 
    \begin{align*}
        \frac{\log \frac{q_n}{p_n}(X)}{\log n} &= -\frac{k \log(1+r)}{2 \log n} + \frac{\langle Q^\intercal X, A_k Q^\intercal X \rangle}{2 \log n} \cdot \frac{r}{r+1}.
    \end{align*}
    To show that \(\left\{\frac{\log \frac{q_n}{p}}{\log n} \right\}\) satisfies a large deviation principle under the null, we will apply the contraction principle (Theorem \ref{thm:contraction_principle}). Under the null \(X_n \sim P_n\), consider that \(\frac{X_n}{\sqrt{2 \log n}} \sim N(0, (2 \log n)^{-1} I_p)\). Consider that for any Borel set \(\Gamma \subset \R^p\) we have
    \begin{align}\label{example:gaussian_ldp}
        \frac{\log P\left( \frac{X_n}{\sqrt{2 \log n}} \in \Gamma \right)}{\log n} &= \frac{1}{\log n} \cdot \log \int_{\Gamma} \frac{(2\log n)^{p/2}}{\sqrt{2\pi}} \exp\left(- ||x||^2 \cdot \log n \right) \, dx.
    \end{align}
    Applying Lemma 3 of \cite{caiOptimalDetectionSparse2014}, it follows that \(\frac{1}{\log n}\log P\left( \frac{X_n}{\sqrt{2 \log n}} \in \Gamma \right) \to - \inf_{t \in \Gamma} ||t||^2\). Thus, \(\frac{X_n}{\sqrt{2 \log n}}\) satisfies a large deviation principle with respect to speed \(\left\{\frac{1}{\log n}\right\}\) with rate function \(J(t) = ||t||^2\). Applying the contraction principle (Theorem \ref{thm:contraction_principle}) to the function \(f : \R^p \to \R\) given by \(f(x) = \frac{r}{r+1} \langle x, A_k x\rangle \), it follows that \(\frac{r}{r+1} \cdot \frac{\langle X_n, A_k X_n \rangle}{2 \log n}\) satisfies a large deviation principle with respect to speed \(\left\{\frac{1}{\log n}\right\}\) and good rate function \(I : \R \to [0, \infty]\),
    \begin{align*}
        I(y) &= \inf\{J(t) : y = f(t)\} \\
        &= \inf\left\{||t||^2 : y = \frac{r}{r+1} \langle t, A_k t\rangle \right\} \\
        &=  \inf\left\{||t||^2 : \frac{r+1}{r} y = \langle t, A_k t \rangle \right\} \\
        &= 
        \begin{cases}
            \frac{r+1}{r} y & \text{if } y \geq 0, \\
            \infty & \text{otherwise}.
        \end{cases}
    \end{align*}
    Under the null \(X_n \sim P_n\), consider that \(Q^\intercal X_n \overset{d}{=} X_n\) by the rotational invariance of \(N(0, I_p)\). Therefore, it immediately follows that \(\frac{r}{r+1} \cdot \frac{\langle Q^\intercal X_n, A_k Q^\intercal X_n\rangle }{\log n}\) satisfies the large deviation principle under the null with good rate function \(I\) and with respect to speed \(\left\{\frac{1}{\log n}\right\}\). Consider that since \(-k\log(1+r)/(2 \log n) \to 0\), it follows that \(\frac{\log \frac{q_n}{p_n}(X_n)}{\log n}\) and \(\frac{\langle Q^\intercal X_n, A_k Q^\intercal X_n\rangle}{2 \log n} \cdot \frac{r}{r+1}\) are exponentially equivalent with respect to speed \(\left\{\frac{1}{\log n}\right\}\) (see Definition \ref{def:exp_equiv}). Thus by Theorem \ref{thm:same_ldp}, it follows that \(\left\{\frac{\log \frac{q_n}{p_n}}{\log n}\right\}\) satisfies the large deviation principle under the null with good rate function \(I\). 
    
    \subsubsection{Determining the detection boundary}
    We now determine the detection boundary. It's easily checked that the conditions of Corollary \ref{corollary:tight_limit} hold. Solving the optimization problem (\ref{eqn:tight_beta}), it follows that the detection boundary is given by 
    \begin{equation*}
        \beta^*(r) = 
        \begin{cases}
            \frac{1}{2} & \text{if } r \leq 1, \\
            1 - \frac{1}{1+r} & \text{if } r > 1.
        \end{cases}
    \end{equation*}
    Interestingly, there is no dependence on the rank \(k\) of \(H\) nor the dimension \(p\). Consequently, the detection boundary is exactly the same as in the heteroscedastic normal mixture testing problem with \(\mu = 0\) and \(\sigma^2 = 1+r\). 

    \subsubsection{Higher Criticism}
    Observe that the rate function \(I\) satisfies the conditions of Theorem \ref{thm:HC_tight}, and so \(\underline{\beta}^{\HC} = \beta^*(r)\). Letting \(X \sim N(0, I_p)\) be independent of the data \(\{X_i\}_{i=1}^{n}\), observe that 
    \begin{align*}
        \HC^*_n &= \sup_{t >0} \frac{|\sum_{i=1}^{n} \mathbf{1}_{\{(q_n/p)(X_i) > t\}} - n P((q_n/p)(X) > t)|}{\sqrt{n P((q_n/p)(X) > t)P((q_n/p)(X) \leq t)}} \\
        &= \sup_{t \in \R} \frac{|\sum_{i=1}^{n} \mathbf{1}_{\{\langle Q^\intercal X_i, A_k Q^\intercal X_i\rangle > t\}} - n P( \langle Q^\intercal X, A_k Q^\intercal X \rangle > t)|}{\sqrt{n P(\langle Q^\intercal X, A_k Q^\intercal X \rangle > t)P(\langle Q^\intercal X, A_k Q^\intercal X \rangle \leq t)}} \\
        &= \sup_{t \in \R} \frac{|\sum_{i=1}^{n} \mathbf{1}_{\{\langle X_i, Q A_k Q^\intercal X_i\rangle > t\}} - n P( \langle X, A_k X \rangle > t)|}{\sqrt{n P(\langle X, A_k X \rangle > t)P(\langle X, A_k X \rangle \leq t)}} \\
        &= \sup_{t \in \R} \frac{|\sum_{i=1}^{n} \mathbf{1}_{\{\langle X_i, Q A_k Q^\intercal X_i\rangle > t\}} - n P( \chi^2_k > t)|}{\sqrt{n P(\chi^2_k > t)P(\chi^2_k \leq t) }}
    \end{align*}
    where the penultimate equality follows from the fact that \(Q^\intercal X \overset{d}{=} X\) under the null and the final equality follows from the fact that \(A_k\) is a rank \(k\) projection matrix. Consider that \(\HC^{*}_n\) is adaptive to the parameters \(r\) and \(\beta\), but requires knowledge of \(Q A_k Q^\intercal\). In other words, the testing statistic requires the knowledge of which subspace in \(\R^p\) the covariance matrix exhibits the perturbation.
    
    \subsection{Detection of sparse correlated pairs}
    In Section 5 of the review article \cite{donohoHigherCriticismLargeScale2015}, Donoho and Jin consider the problem of detecting the presence of a small collection of correlated pairs. More specifically, the testing problem (\ref{problem:sparse_mixture_detection_1})-(\ref{problem:sparse_mixture_detection_2}) with \(P_n = P = N(0, I_2)\) and \(Q_n = N(\mu_n \mathbf{1}_2, \Sigma)\) is considered where \(\mu_n = \sqrt{r \log n}\) for \(r > 0\) and 
    \begin{equation*}
        \mathbf{1}_2 = \left( \begin{matrix} 1 \\ 1 \end{matrix} \right), 
        \;\;\;\; \Sigma = 
        \left(
        \begin{matrix}
            1 & \rho \\
            \rho & 1
        \end{matrix}
        \right)
    \end{equation*}
    with \(-1 < \rho < 1\). Without loss of generality, we take \(\rho \neq 0\). The case \(\rho = 0\) reduces to a special case of the sparse multivariate normal mixture studied in Example \ref{example:mvn}. Donoho and Jin illustrate the applicability of their original formulation of the Higher Criticism statistic and perform some simulations. We will deduce the detection boundary and investigate the behavior of Gao and Ma's Higher Criticism type testing statistic. It turns out that the detection boundary is highly related to the detection boundary in Example \ref{example:heteroscedastic}.

    \subsubsection{Finding the rate function}
    First, consider that \(\Sigma\) can be diagonalized as \(\Sigma = Q^\intercal W Q\) where
    \begin{equation*}
        Q = \left(\begin{matrix} \frac{1}{\sqrt{2}} & \frac{1}{\sqrt{2}} \\ \frac{1}{\sqrt{2}} & -\frac{1}{\sqrt{2}} \end{matrix} \right), 
        \;\;\;\; W = 
        \left(
        \begin{matrix}
            1+\rho & 0 \\
            0 & 1 - \rho
        \end{matrix}
        \right).
    \end{equation*}
    Thus, it follows that 
    \begin{align*}
        \frac{q_n}{p}(x) &= \frac{1}{\sqrt{1-\rho^2}} \exp\left(-\frac{1}{2} (Qx)^\intercal (W^{-1} - I_2) Qx + (Qx)^\intercal W^{-1} Q (\mu_n \mathbf{1}_2) \right) \exp\left( -\frac{\mu_n^2 (Q\mathbf{1})^\intercal W^{-1} (Q\mathbf{1})}{2} \right) \\
        &= \frac{1}{\sqrt{1-\rho^2}} \exp\left(-\frac{1}{2} (Qx)^\intercal (W^{-1} - I_2) Qx + (Qx)^\intercal W^{-1} Q (\mu_n \mathbf{1}_2) \right) \exp\left( - \frac{r}{1+\rho} \cdot \log n \right) \\
        &= \frac{1}{\sqrt{1-\rho^2}} \exp\left(-\frac{1}{2} \left(Qx -  M(\mu_n \cdot \mathbf{1}_2) \right)(W^{-1} - I_2)\left(Qx - M(\mu_n \cdot \mathbf{1}_2) \right) \right) \\
        &\;\; \cdot \exp\left(\frac{\mu_n^2 \mathbf{1}_2^\intercal Q^\intercal W^{-1}(W^{-1}-I_2)^{-1}W^{-1} Q\mathbf{1}_2}{2} \right) \exp\left( - \frac{r}{1+\rho} \cdot \log n \right) \\
        &= \frac{1}{\sqrt{1-\rho^2}} \exp\left(-\frac{1}{2} \left(Qx -  M(\mu_n \cdot \mathbf{1}_2) \right)(W^{-1} - I_2)\left(Qx - M(\mu_n \cdot \mathbf{1}_2) \right) \right) \\
        &\;\; \cdot \exp\left(-\frac{r}{\rho(1+\rho)} \cdot \log n\right) \exp\left( - \frac{r}{1+\rho} \cdot \log n \right) \\
        &= \frac{1}{\sqrt{1-\rho^2}} \exp\left(-\frac{1}{2} \left(Qx -  M(\mu_n \cdot \mathbf{1}_2) \right)(W^{-1} - I_2)\left(Qx - M(\mu_n \cdot \mathbf{1}_2) \right) \right) \exp\left(-\frac{r}{\rho} \cdot \log n \right)
    \end{align*}
    where \(M = (W^{-1} - I_2)^{-1}W^{-1}Q\). Further simplification yields 
    \begin{align*}
        & \frac{q_n}{p}(x) \\
        &= \frac{1}{\sqrt{1-\rho^2}} \exp\left(-\frac{1}{2} \left(Qx + \frac{\sqrt{2r\log n}}{\rho} e_1 \right) (W^{-1} - I_2) \left(Qx + \frac{\sqrt{2r\log n}}{\rho} e_1 \right) \right) \exp\left(-\frac{r}{\rho} \cdot \log n \right)
    \end{align*}
    where \(e_1 = (1, 0)\) is the first standard basis vector in \(\R^2\). A multivariate analogue of the argument in Example \ref{example:heteroscedastic} establishes that the tail condition (\ref{eqn:varadhan_tail_condition}) is satisfied. Thus, we can leverage Corollay \ref{corollary:tight_limit} to obtain the detection boundary. 
    
    Consider that 
    \begin{align*}
        \frac{\log \frac{q_n}{p}(X)}{\log n} &= -\frac{\log(1-\rho^2)}{2 \log n} - \frac{\left(QX + \frac{\sqrt{2r\log n}}{\rho} e_1 \right) (W^{-1} - I_2) \left(QX + \frac{\sqrt{2r\log n}}{\rho} e_1 \right)}{2 \log n} - \frac{r}{\rho}
    \end{align*}
    To deduce a large deviation principle for \(\left\{\frac{\log \frac{q_n}{p}}{\log n}\right\}\), we apply the contraction principle (Theorem \ref{thm:contraction_principle}). Consider that under the null, 
    \begin{equation*}
        \frac{\left(QX + \frac{\sqrt{2r \log n}}{\rho} e_1 \right)}{\sqrt{2 \log n}} \sim N\left(\frac{\sqrt{r}}{\rho} e_1, \frac{1}{2\log n} I_2 \right).
    \end{equation*}
    Following an argument similar to that of (\ref{example:gaussian_ldp}) in Example \ref{example:low_rank_perturbation}, it follows that \(\frac{\left(QX + \frac{\sqrt{2r \log n}}{\rho} e_1 \right)}{\sqrt{2 \log n}}\) satisfies a large deviation principle with respect to speed \(\left\{\frac{1}{\log n}\right\}\) and good rate function \(J : \R^2 \to [0, \infty]\) given by 
    \begin{align*}
        J(t) = \left| \left|t - \frac{\sqrt{r}}{\rho} e_1 \right|\right|^2. 
    \end{align*}
    Consider the function \(f : \R^2 \to \R\) given by \(f(x) = -\frac{r}{\rho} - t^\intercal (W^{-1} - I_2) t\). Applying the contraction principle (Theorem \ref{thm:contraction_principle}), it follows that \(-\frac{\left(QX + \frac{\sqrt{2r\log n}}{\rho} e_1 \right) (W^{-1} - I_2) \left(QX + \frac{\sqrt{2r\log n}}{\rho} e_1 \right)}{2 \log n} -\frac{r}{\rho}\) satisfies a large deviation principle with good rate function 
    \begin{align*}
        I(y) &:= \inf\{J(t) : y = f(t)\} \\
        &= \inf\left\{\left|\left| t - \frac{\sqrt{r}}{\rho}e_1 \right|\right|^2 : y = -\frac{r}{\rho} - t^\intercal (W^{-1} - I_2) t \right\}
    \end{align*}
    We can expand to obtain 
    \begin{equation}\label{eqn:correlated_pairs_I}
        I(y) = \inf\left\{t_1^2 + t_2^2 - \frac{2\sqrt{r}}{\rho}t_1 + \frac{r}{\rho^2} : y = -\frac{r}{\rho} + \frac{\rho}{1+\rho}t_1^2 - \frac{\rho}{1-\rho}t_2^2 \right\}
    \end{equation}
    Since \(-\frac{\log(1-\rho^2)}{2 \log n}\) is a deterministic sequence converging to zero, it follows that, under the null, \(\frac{\log \frac{q_n}{p}(X)}{\log n}\) is exponentially equivalent to \(-\frac{\left(QX + \frac{\sqrt{2r\log n}}{\rho} e_1 \right) (W^{-1} - I_2) \left(QX + \frac{\sqrt{2r\log n}}{\rho} e_1 \right)}{2 \log n} -\frac{r}{\rho}\) with respect to speed \(\left\{\frac{1}{\log n}\right\}\). Hence, \(\left\{\frac{\log \frac{q_n}{p}(X)}{\log n}\right\}\) satisfies the large deviation principle under the null with good rate function \(I\). To obtain an explicit expression for \(I\), we consider the two cases \(\rho > 0\) and \(\rho < 0\) separately. 

    \textbf{Case 1:} Suppose \(\rho > 0\). Examining the optimization problem (\ref{eqn:correlated_pairs_I}), note that the constraint implies that 
    \begin{equation*}
        t_2^2 = -\frac{1-\rho}{\rho} \left(y + \frac{r}{\rho} - \frac{\rho}{1+\rho}t_1^2\right). 
    \end{equation*}
    Since \(t_2^2 \geq 0\), we have the constraint \(-\frac{1-\rho}{\rho} \left(y + \frac{r}{\rho} - \frac{\rho}{1+\rho}t_1^2\right) \geq 0\). Since \(0 < \rho < 1\), it follows that we must have \(t_1^2 \geq \frac{1+\rho}{\rho^2}(\rho y + r)\). In other words, the problem (\ref{eqn:correlated_pairs_I}) is equivalently written as 
    \begin{equation}\label{eqn:correlated_pairs_opt}
        \min_{t_1 \in \R} \, \left\{ t_1^2 - \frac{1-\rho}{\rho} \left(y + \frac{r}{\rho} - \frac{\rho}{1+\rho}t_1^2\right) - \frac{2\sqrt{r}}{\rho}t_1 + \frac{r}{\rho^2}\right\} \text{ s.t. } t_1^2 \geq \frac{1+\rho}{\rho^2}\left(\rho y + r\right).
    \end{equation}
    If \(\rho y + r < 0\), then the constraint is trivial as we always have \(t_1^2 \geq 0\). Then (\ref{eqn:correlated_pairs_opt}) is an unconstrained optimization problem. Taking a derivative of the objective function with respect to \(t_1\) and finding the zero shows that the minimum is achieved at \(t_1 = \frac{1+\rho}{\rho} \frac{\sqrt{r}}{2}\). Consequently, we have \(I(y) = \frac{(\rho - 1)(r+2\rho y)}{2\rho^2}\). On the other hand, if \(\rho y + r \geq 0\), then the constraint is no longer trivial. Letting \(f(t_1)\) denote the objective function, observe that 
    \begin{equation*}
        f'(t_1) = \frac{4}{1+\rho} t_1 - \frac{2\sqrt{r}}{\rho}.
    \end{equation*}
    The unique zero of \(f'\) is \(\frac{1+\rho}{\rho}\frac{\sqrt{r}}{2}\). Noting that our constraint is \(t_1^2 \geq \frac{1+\rho}{\rho^2}(\rho y + r)\), observe that \(\frac{(1+\rho)^2}{\rho^2} \frac{r}{4} \geq \frac{1+\rho}{\rho^2}(\rho y + r)\) if and only if \(\rho y + r \leq \frac{(1 + \rho)r}{4}\). Therefore, it follows that if \(\rho y + r \leq \frac{(1+\rho)r}{4}\), then (\ref{eqn:correlated_pairs_opt}) is achieved at \(t_1 = \frac{1+\rho}{\rho} \frac{\sqrt{r}}{2}\) and so \(I(y) = \frac{(\rho-1)(r+2\rho y)}{2\rho^2}\). If, however, \(\rho y + r > \frac{(1+\rho) r}{4}\), then it follows that \(\frac{(1+\rho)}{\rho} \frac{\sqrt{r}}{2} < \sqrt{\frac{1+\rho}{\rho^2}(\rho y + r)}\). Consequently \(f'(t_1) \geq f'\left(\frac{1+\rho}{\rho} \frac{\sqrt{r}}{2}\right) = 0\) for all \(t_1 \geq \sqrt{\frac{1+\rho}{\rho^2}(\rho y + r)}\). Furthermore, since \(\frac{(1+\rho)}{\rho} \frac{\sqrt{r}}{2} \geq 0 > -\sqrt{\frac{1+\rho}{\rho^2}(\rho y + r)}\), it follows that \(f'(t_1) \leq 0\) for all \(t_1 \leq -\sqrt{\frac{1+\rho}{\rho^2}(\rho y + r)}\). Direct comparison shows that (\ref{eqn:correlated_pairs_opt}) is achieved at \(t_1 = \sqrt{\frac{(1+\rho)}{\rho^2}(\rho y + r)}\) which yields \(I(y) = \frac{1+\rho}{\rho^2} \left(\sqrt{\rho y + r} - \sqrt{\frac{r}{1+\rho}}\right)^2\). To summarize, we have shown that if \(\rho > 0\), then 
    \begin{equation}\label{eqn:positive_corr_rate}
        I(y) = 
        \begin{cases}
            \frac{(\rho-1)(r+2\rho y)}{2\rho^2} & \text{if } \rho y + r \leq \frac{(1+\rho)r}{4}, \\
            \frac{1+\rho}{\rho^2} \left(\sqrt{\rho y + r} - \sqrt{\frac{r}{1+\rho}}\right)^2 & \text{if } \rho y + r > \frac{(1+\rho) r}{4}.
        \end{cases}
    \end{equation}

    \textbf{Case 2:} Suppose \(\rho < 0\). As in Case 1, we have \(t_2^2 = -\frac{1-\rho}{\rho} \left(y + \frac{r}{\rho} - \frac{\rho}{1+\rho} t_1^2\right)\). Likewise, the constraint \(t_2^2 \geq 0\) along with \(-1 < \rho < 0\) implies that we have the constraint \(t_1^2 \geq \frac{1+\rho}{\rho^2}(\rho y + r)\). Therefore, the problem (\ref{eqn:correlated_pairs_I}) is again equivalently given by (\ref{eqn:correlated_pairs_opt}). Of course, the difference here compared to Case 1 is that \(\rho < 0\). Nonetheless, the analysis proceeds similarly. If \(\rho y + r < 0\), then the constraint is trivial, and so \(I(y) = \frac{(\rho - 1)(r+2\rho y)}{2\rho^2}\) as argued in Case 1. On the other hand, if \(\rho y + r \geq 0\), then the constraint is nontrivial. As in Case 1, if we additionally have \(\rho y + r \leq \frac{(1+\rho) r}{4}\), then it follows that \(I(y) = \frac{(\rho - 1)(r+ 2\rho y)}{2\rho^2}\). 
    
    Now suppose we instead have \(\rho y + r > \frac{(1+\rho) r}{4}\). Let \(f(t_1)\) denote the objective function of (\ref{eqn:correlated_pairs_opt}) and note that \(f'(t_1) = \frac{4}{1+\rho} t_1 - \frac{2\sqrt{r}}{\rho}\). Then it follows per the analysis in Case 1 that \(\frac{(1+\rho)^2}{\rho^2} \frac{r}{4} < \frac{1+\rho}{\rho^2} (\rho y + r)\). Since \(-1 < \rho < 0\), it follows that \(0 > \frac{1+\rho}{\rho} \frac{\sqrt{r}}{2} > - \sqrt{\frac{1+\rho}{\rho^2}(\rho y + r)}\). Consequently, we have \(f'(t_1) \leq 0\) for all \(t_1 \leq -\sqrt{\frac{(1+\rho)}{\rho^2} (\rho y + r)}\) and \(f'(t_1) \geq 0\) for all \(t_1 \geq \sqrt{\frac{(1+\rho)}{\rho^2}(\rho y + r)}\). Hence, (\ref{eqn:correlated_pairs_opt}) is achieved at either \(t_1 = \pm \sqrt{\frac{(1+\rho)}{\rho^2}(\rho y + r)}\). Since \(\rho < 0\), direct comparison indicates that the minimum is achieved at \(t_1 = -\sqrt{\frac{(1+\rho)}{\rho^2}(\rho y + r)}\), yielding \(I(y) = \frac{1+\rho}{\rho^2}\left(\sqrt{\rho y + r} - \sqrt{\frac{r}{1+\rho}}\right)^2\). To summarize, we have shown that if \(\rho < 0\), then 
    \begin{equation} \label{eqn:negative_corr_rate}
        I(y) = 
        \begin{cases}
            \frac{(\rho-1)(r+2\rho y)}{2\rho^2} & \text{if } \rho y + r \leq \frac{(1+\rho)r}{4}, \\
            \frac{1+\rho}{\rho^2} \left(\sqrt{\rho y + r} - \sqrt{\frac{r}{1+\rho}}\right)^2 & \text{if } \rho y + r > \frac{(1+\rho) r}{4}.
        \end{cases}
    \end{equation}

    We see that the rate function is exactly the same in both cases, that is, we have exact equality of (\ref{eqn:positive_corr_rate}) and (\ref{eqn:negative_corr_rate}).

    \subsubsection{Determining the detection boundary}
    We can now deduce the detection boundary. We consider the cases \(\rho > 0\) and \(\rho < 0\) separately.
    
    \textbf{Case 1:} Suppose \(\rho > 0\). Note that \(t \geq 0\) implies \(\rho t + r \geq r > \frac{(1+\rho) r}{4}\) because \(r > 0\) and \(0 < \rho < 1\). Referring to (\ref{eqn:positive_corr_rate}), we see that \(I(t) = \frac{1+\rho}{\rho^2} \left(\sqrt{\rho t + r} - \sqrt{\frac{r}{1+\rho}}\right)^2\) for \(t \geq 0\). Remarkably, this is precisely the rate function obtained in Example \ref{example:heteroscedastic} with \(\rho = \sigma^2 - 1\). Therefore, we obtain the detection boundary of Cai, Jeng, and Jin \cite{caiOptimalDetectionHeterogeneous2011} with \(\rho = \sigma^2 - 1\), that is,
    \begin{equation*}
        \beta^*(r, \rho) =
        \begin{cases}
            \frac{1}{2} + \frac{r}{1 - \rho} & \text{if } 2\sqrt{r} + \rho \leq 1, \\
            1 - \frac{(1-\sqrt{r})_{+}^2}{1+\rho} &\text{if } 2\sqrt{r} + \rho > 1.
        \end{cases}
    \end{equation*}
    We warn the reader that we have parametrized the mean \(\mu = \sqrt{r \log n}\) here, whereas the mean has parametrization \(\sqrt{2r \log n}\) in Example \ref{example:heteroscedastic}. This difference in a factor of \(\sqrt{2}\) reflects the difference in dimension between the two settings, namely that \(||\mathbf{1}_2|| = \sqrt{2}\).
        
    \textbf{Case 2:} Suppose \(\rho < 0\). The course of the calculation will reveal that Corollary \ref{corollary:tight_limit} holds and so
    \begin{equation*}
        \beta^*(r, \rho) = \frac{1}{2} + 0 \vee \sup_{t \geq 0} \left\{t - I(t) + \frac{1 \wedge I(t)}{2} \right\}.
    \end{equation*}
    Examining (\ref{eqn:negative_corr_rate}), observe that \(\rho y + r \leq \frac{(1+\rho)r}{4}\) if and only if \(t \geq \frac{\rho - 3}{\rho} \frac{r}{4}\). Consequently, we can write 
    \begin{equation*}
        \beta^*(r, \rho) = \frac{1}{2} + 0 \vee E \vee F
    \end{equation*}
    where 
    \begin{align*}
        E &= \sup_{t \geq \frac{\rho-3}{\rho} \frac{r}{4}} \left\{ t - \frac{(\rho - 1)(r + 2\rho t)}{2\rho^2} + \frac{1 \wedge \left[\frac{(\rho - 1)(r + 2\rho t)}{2\rho^2} \right]}{2} \right\}, \\
        F &= \sup_{0 \leq t < \frac{\rho-3}{\rho} \frac{r}{4}} \left\{ t -  \left(\frac{1+\rho}{\rho^2} \left(\sqrt{\rho t + r} - \sqrt{\frac{r}{1+\rho}}\right)^2\right) + \frac{1 \wedge\left[\frac{1+\rho}{\rho^2} \left(\sqrt{\rho t + r} - \sqrt{\frac{r}{1+\rho}}\right)^2 \right]}{2}\right\}.
    \end{align*}
    We examine each optimization problem separately. Examining \(E\) first, it's clear that the supremum is always achieved at the left endpoint, that is, at \(t = \frac{\rho-3}{\rho} \frac{r}{4}\). Examining (\ref{eqn:negative_corr_rate}), it is easily checked that \(I\) is continuous. Since the supremum is achieved at the left endpoint for \(E\), it immediately follows that 
    \begin{equation}\label{eqn:neg_corr_opt}
        E \vee F = \sup_{0 \leq t \leq \frac{\rho - 3}{\rho} \frac{r}{4}} \left\{ t -  \left(\frac{1+\rho}{\rho^2} \left(\sqrt{\rho t + r} - \sqrt{\frac{r}{1+\rho}}\right)^2\right) + \frac{1 \wedge\left[\frac{1+\rho}{\rho^2} \left(\sqrt{\rho t + r} - \sqrt{\frac{r}{1+\rho}}\right)^2 \right]}{2}\right\}.
    \end{equation}
    Following our approach in Example \ref{example:heteroscedastic}, let us make the change of variable for \(0 \leq t \leq \frac{\rho - 3}{\rho} \frac{r}{4}\)
    \begin{equation*}
        s = I(t) = \frac{1+\rho}{\rho^2} \left(\sqrt{\rho t + r} - \sqrt{\frac{r}{1+\rho}}\right)^2. 
    \end{equation*}
    Rearranging, we have 
    \begin{align*}
        t &= \frac{1}{\rho} \left( \frac{\rho \sqrt{s}}{\sqrt{1+\rho}} + \frac{\sqrt{r}}{\sqrt{1+\rho}} \right)^2 - \frac{r}{\rho} \\
          &= \frac{\rho^2 s + 2\rho \sqrt{sr} + r - r (1+\rho)}{\rho (1+\rho)} \\
          &= \frac{\rho s + 2\sqrt{sr} - r}{1+\rho} \\
          &= s - \frac{s - 2\sqrt{sr} + r}{1+\rho} \\
          &= s - \frac{(\sqrt{s} - \sqrt{r})^2}{1+\rho}.
    \end{align*}
    Since \(\rho < 0\) and \(s > 0\), it follows by some rearrangement that \(t \geq 0\) if and only if \(s \geq r \frac{(1-\sqrt{1+\rho})^2}{\rho^2}\). Likewise, we have that \(t \leq \frac{\rho - 3}{\rho} \frac{r}{4}\) if and only if \(s \geq \frac{(\rho+3)^2}{4\rho^2} r\) or \(s \leq \frac{(\rho-1)^2}{4\rho^2} r\). Consequently, the optimization problem (\ref{eqn:neg_corr_opt}) can be written as
    \begin{equation}\label{eqn:neg_final_opt}
        \left[\sup_{\frac{(1-\sqrt{1+\rho})^2}{\rho^2}r \leq s \leq \frac{(\rho-1)^2}{4\rho^2} r} \left\{ \left(s - \frac{(\sqrt{s} - \sqrt{r})^2}{1+\rho}\right) - s + \frac{1 \wedge s}{2}  \right\}\right] \vee \left[ \sup_{s \geq \frac{(\rho+3)^2}{4\rho^2}r} \left\{ \left(s - \frac{(\sqrt{s} - \sqrt{r})^2}{1+\rho}\right) - s + \frac{1 \wedge s}{2}  \right\} \right].
    \end{equation}
    To solve this optimization problem, we will actually show (\ref{eqn:neg_final_opt}) achieves the same value as 
    \begin{equation}\label{eqn:full_opt}
        \sup_{s \geq 0} \left\{ - \frac{(\sqrt{s} - \sqrt{r})^2}{1+\rho} + \frac{1 \wedge s}{2} \right\}.
    \end{equation}
    Clearly the quantity (\ref{eqn:full_opt}) larger than or equal to (\ref{eqn:neg_final_opt}) as the maximization is over a superset. In fact, we will show equality. This is useful in deducing the detection boundary because (\ref{eqn:full_opt}) is precisely the optimization problem that yields the Cai, Jeng, and Jin detection boundary \cite{caiOptimalDetectionHeterogeneous2011} with \(1+\rho = \sigma^2\) as Cai and Wu showed in Section V.C of \cite{caiOptimalDetectionSparse2014}. In particular, we have (Section V.C \cite{caiOptimalDetectionSparse2014})
    \begin{equation*}
        \sup_{s \geq 0} \left\{ - \frac{(\sqrt{s} - \sqrt{r})^2}{1+\rho} + \frac{1 \wedge s}{2} \right\} = 
        \begin{cases}
            \frac{r}{1-\rho} &\text{if } 2\sqrt{r} + (1+\rho)\leq 2, \\
            \frac{1}{2} - \frac{(1-\sqrt{r})_{+}^2}{1+\rho} &\text{if } 2\sqrt{r} + (1+\rho) > 2.
        \end{cases}
    \end{equation*}
    It's clear from this that if \(2\sqrt{r} + (1+\rho) \leq 2\), then a maximizer of (\ref{eqn:full_opt}) is \(s = \frac{4r}{(1-\rho)^2}\). We claim that \(s = \frac{4r}{(1-\rho)^2}\) is also a maximizer of (\ref{eqn:neg_final_opt}) when \(2\sqrt{r} + (1+\rho) \leq 2\). It suffices to show that \(\frac{(1-\sqrt{1+\rho})^2}{\rho^2} r \leq \frac{4r}{(1-\rho)^2} \leq \frac{(\rho-1)^2}{4\rho^2} r\). To see the upper bound, consider that \(16 \leq \frac{(\rho - 1)^4}{\rho^2}\) for all \(\rho < 0\). Consequently, we have \(16\rho^2 \leq (\rho-1)^4\), which yields \(\frac{4}{(1-\rho)^2} \leq \frac{(\rho - 1)^2}{4\rho^2}\). This clearly implies \(\frac{4r}{(1-\rho)^2} \leq \frac{(\rho - 1)^2}{4\rho^2}r\) and so we have the upper bound. To see the lower bound, consider that \(\frac{\rho^2}{(1-\rho)^2(1-\sqrt{1+\rho})^2} \geq \frac{1}{4}\) for all \(\rho > -1\). Consequently, we have \(\frac{4}{(1-\rho)^2} \geq \frac{(1-\sqrt{1+\rho})^2}{\rho^2}\), which immediately implies \(\frac{(1-\sqrt{1+\rho})^2}{\rho^2} r \leq \frac{4r}{(1-\rho)^2}\). Hence, we have shown that (\ref{eqn:full_opt}) and (\ref{eqn:neg_final_opt}) are equal if \(2\sqrt{r} + (1+\rho) \leq 2\).

    Now suppose \(2 \sqrt{r} + (1 + \rho) > 2\) and \(r < 1\). It's clear that a maximizer of (\ref{eqn:full_opt}) is \(s =  1\). We claim \(s = 1\) is also a maximizer of (\ref{eqn:neg_final_opt}). It suffices to show \(\frac{(1-\sqrt{1+\rho})^2}{\rho^2}r \leq 1 \leq \frac{(\rho-1)^2}{4\rho^2} r\). To see the upper bound, consider that \(2\sqrt{r} + (1+\rho) > 2\) implies \(\sqrt{r} \geq \frac{1-\rho}{2}\), and so \(r \geq \frac{(1-\rho)^2}{4}\). Since \(\rho < 0\), it follows that \(1 \leq \frac{(\rho-1)^4}{16\rho^2} \leq \frac{(\rho-1)^2 r}{4\rho^2}\) and so we have the upper bound. To see the lower bound, note that \(r < 1\) implies \(\frac{(1-\sqrt{1+\rho})^2}{\rho^2} r \leq \frac{(1-\sqrt{1+\rho})^2}{\rho^2}\). Since \(-1 < \rho\), it follows that \(\frac{(1-\sqrt{1+\rho})^2}{\rho^2} \leq 1\), and so we have the lower bound. Hence, we have shown that (\ref{eqn:full_opt}) and (\ref{eqn:neg_final_opt}) are equal if \(2\sqrt{r} + (1+\rho) > 2\) and \(r < 1\).

    The remaining case is when \(2 \sqrt{r} + (1+\rho) > 2\) and \(r \geq 1\). It's clear that a maximizer of (\ref{eqn:full_opt}) is \(s = r\). We claim that \(s = r\) is also a maximizer of (\ref{eqn:neg_final_opt}). It suffices to show \(\frac{(1-\sqrt{1+\rho})^2}{\rho^2}r \leq r \leq \frac{(\rho-1)^2}{4\rho^2}r\). The upper bound follows immediately from the fact that \(1 \leq \frac{(\rho-1)^2}{4\rho^2}\) for all \(-1 < \rho < 0\). The lower bound follows immediately from the fact that \(1 \geq \frac{(1-\sqrt{1+\rho})^2}{\rho^2}\) for all \(\rho > -1\). Hence, we have shown that (\ref{eqn:full_opt}) and (\ref{eqn:neg_final_opt}) are equal if \(2\sqrt{r} + (1+\rho) > 2\) and \(r \geq 1\). 

    Thus, we obtain the detection boundary 
    \begin{equation*}
        \beta^*(r, \rho) = 
        \begin{cases}
            \frac{1}{2} + \frac{r}{1-\rho} &\text{if } 2\sqrt{r} + \rho \leq 1, \\
            1 - \frac{(1-\sqrt{r})_{+}^2}{1+\rho} &\text{if } 2\sqrt{r} + \rho > 1
        \end{cases}
    \end{equation*}
    when \(\rho < 0\). Note that this is precisely the detection boundary of Cai, Jeng, and Jin \cite{caiOptimalDetectionHeterogeneous2011} with \(1+\rho = \sigma^2\). Furthermore, the detection boundaries exactly between the two cases \(\rho > 0\) and \(\rho < 0\). Again, we warn the reader that we have parametrized the mean \(\mu = \sqrt{r \log n}\) here, whereas the mean has parametrization \(\sqrt{2r \log n}\) in Example \ref{example:heteroscedastic}. As mentioned earlier, this difference in a factor of \(\sqrt{2}\) is due to the difference of dimension between the two settings, that is, \(||\mathbf{1}_2|| = \sqrt{2}\).
    
    \subsubsection{Higher Criticism}
    It is straight-forward to check that the rate function \(I\) is indeed convex and that the other conditions of Theorem \ref{thm:HC_tight} hold. The test \(\psi_{\HC_n^*}\) given by (\ref{eqn:HC_test}) achieves the detection boundary.

    \subsection{Stochastic Block Model}
    In some areas (such as sociology, political science, and neuroscience), the observational units exhibit relationships amongst one another thereby forming a network. The field of network analysis deals with addressing statistical questions regarding the observed network, such as determining whether there exist latent communities in the network, identifying the communities if they do exist, and possibly estimating parameters of a statistical model. We refer the reader to a recent survey \cite{gaoMinimaxRatesNetwork2021} covering fundamental statistical limits in a number of estimation and testing tasks in network analysis.

    The stochastic block model (SBM) \cite{hollandStochasticBlockmodelsFirst1983} is a popular model for capturing the presence of communities in a network. We refer the reader to the survey \cite{abbeCommunityDetectionStochastic2018} for further background. In the simplest case of two communities, consider \(n\) nodes and let \(z \in \{-1, 1\}^n\) denote the community membership for the \(n\) nodes, i.e. \(z_i\) denotes to which community node \(i\) belongs. The statistician's observation is the random symmetric matrix \(A \in \{0, 1\}^{n \times n}\), where \(A_{ij} = A_{ji} = 1\) denotes the presence of an edge between nodes \(i\) and \(j\). Likewise, \(A_{ij} = A_{ji} = 0\) denotes the absence of an edge. The data generating process is determined by the community structure, namely for \(1 \leq i < j \leq n\), we have independent draws
    \begin{equation*}
        A_{ij} \sim 
        \begin{cases}
            \Bernoulli(p) &\text{if } z_i = z_j, \\
            \Bernoulli(q) &\text{if } z_i \neq z_j. 
        \end{cases}
    \end{equation*}
    We set \(A_{ji} = A_{ij}\) to enforce symmetry and \(A_{ii} = 0\) to disallow self-loops. Here, \(p,q \in (0, 1)\) are parameters. Note that \(p\) gives the probability of an edge between two nodes in the same community and \(q\) gives the probability of an edge between two nodes in different communities. When \(p > q\) the SBM is said to be assortative; otherwise, the SBM is said to be disassortative. There is a litany of statistical tasks associated with the SBM; perhaps the most popular is the community detection problem (that is, estimation of \(z\) under a suitable loss function) \cite{abbeCommunityDetectionStochastic2018, zhangMinimaxRatesCommunity2016}.

    To illustrate an application of our results, we will consider a sparse mixture detection problem related to two sample testing of SBMs. Suppose we observe two independent SBMs on the same set of \(n\) nodes. We have a node of interest, say node \(i^*\), and the statistical problem is to determine whether \(i^*\) is in a community with the same members in each observed network. As an example of a practical application, suppose we observe brain networks consisting of \(n\) neurons from healthy individuals and diseased individuals, and the scientific question of interest is to determine whether the connectivity of neuron \(i^*\) to neighboring neurons differs between healthy and diseased individuals. Another example entails detecting whether the connectivity of gene \(i^*\) in a gene regulatory network is different between healthy and diseased individuals (or control and treatment groups, etc.). An extension of this question is to test whether the community structure of the full network is the same between the two observed SBMs rather than exclusively focusing on node \(i^*\). In this line of work, \cite{gangradeTwoSampleTestingCan2018} has addressed the two community case. 
    
    \subsubsection{Two Sample Testing}
    Let \(z, \sigma \in \{-1, 1\}^{n+1}\) denote community membership for a common set of \(n+1\) nodes (considering \(n+1\) nodes versus \(n\) nodes is only for convenience). Let \(A, B \in \R^{n+1 \times n+1}\) be random symmetric matrices where \(A_{ii} = B_{ii} = 0\) for all \(1 \leq i \leq n+1\) and 
    \begin{align*}
        A_{ij} &\sim \begin{cases}
            \Bernoulli(p) & \text{if } \sigma_i = \sigma_j, \\
            \Bernoulli(q) & \text{if } \sigma_i \neq \sigma_j,
        \end{cases}
        \\
        B_{ij} &\sim 
        \begin{cases}
            \Bernoulli(p) &\text{if } z_i = z_j, \\
            \Bernoulli(q) &\text{if } z_i \neq z_j.
        \end{cases}
    \end{align*}
    for \(1 \leq i < j \leq n+1\) are drawn independently. Here, \(p, q \in (0, 1)\) are parameters. To ensure symmetry, set \(A_{ji} := A_{ij}\) and \(B_{ji} := B_{ij}\). As mentioned, the problem of interest is to determine whether node \(i^*\) has the same community members between the two SBMs \(A\) and \(B\). Without loss of generality, let us take \(i^* = 1\). To formally state the testing problem, let us define 
    \begin{align*}
        S_A &:= \left\{ 2 \leq j \leq n+1 : \sigma_j = \sigma_1 \right\}, \\
        S_B &:= \left\{ 2 \leq j \leq n+1 : z_j = z_1 \right\}.
    \end{align*}
    Concretely, the testing problem is to test, for \(\varepsilon > 0\),
    \begin{align}\label{problem:sbm_original}
        H_0 &: S_A = S_B, \\
        H_1 &:  \frac{|S_A \,\Delta\, S_B|}{n} > \varepsilon.
    \end{align}
    Here, \(\Delta\) denotes symmetric difference and so \(S_A \,\Delta\, S_B = \{ 2 \leq j \leq n+1 : \sigma_j = \sigma_1, z_j \neq z_1 \text{ or } \sigma_j \neq \sigma_1, z_j = z_1\}\). For ease of notation, let us denote \(B(\pi) := \Bernoulli(\pi)\) for \(\pi \in (0, 1)\). Under the null hypothesis, 
    \begin{equation*}
        \left( \begin{matrix} A_{1j} \\ B_{1j} \end{matrix}\right) \sim \mathbf{1}_{\{\sigma_1 = \sigma_j\}} B(p) \otimes B(p)  + \mathbf{1}_{\{\sigma_1 \neq \sigma_j\}}B(q) \otimes B(q)
    \end{equation*}
    independently for \(2 \leq j \leq n+1\). Under the alternative, there are two cases,
    \begin{equation*}
        \left(\begin{matrix} A_{1j} \\ B_{1j} \end{matrix} \right) \sim 
        \begin{cases}
            \mathbf{1}_{\{\sigma_1 = \sigma_j\}}  B(p) \otimes B(p) + \mathbf{1}_{\{\sigma_{1} \neq \sigma_j\}} B(q) \otimes B(q) &\text{if } j \in (S_A \, \Delta \, S_B)^c,\\
            \mathbf{1}_{\{\sigma_1 = \sigma_j\}}  B(p) \otimes B(q) + \mathbf{1}_{\{\sigma_1 \neq \sigma_j\}} B(q) \otimes B(p) &\text{if } j \in S_A \, \Delta \, S_B.
        \end{cases}
    \end{equation*}
    Recall that under the alternative, \(n^{-1} |S_A \, \Delta \, S_B| > \varepsilon\). Assuming the two communities in \(A\) are of roughly equal size, we can formulate the related sparse mixture detection problem
    \begin{align}
        H_0 &: \left( \begin{matrix} A_{1j} \\ B_{1j} \end{matrix}\right) \overset{iid}{\sim} \frac{1}{2} \cdot B(p) \otimes B(p) + \frac{1}{2} \cdot B(q) \otimes B(q), \label{problem:sbm_problem1} \\
        H_1 &: \left( \begin{matrix} A_{1j} \\ B_{1j} \end{matrix}\right) \overset{iid}{\sim} (1-\varepsilon) \left[\frac{1}{2} \cdot B(p) \otimes B(p) + \frac{1}{2} \cdot B(q) \otimes B(q) \right] + \varepsilon \left[\frac{1}{2} \cdot B(p) \otimes B(q) + \frac{1}{2} \cdot B(q) \otimes B(p)\right]. \label{problem:sbm_problem2}  
    \end{align}
    Note that the indices run over \(2 \leq j \leq n+1\). Adopting an asymptotic perspective, we use the calibration \(\varepsilon = n^{-\beta}\) as in (\ref{eqn:beta_sparsity}). Furthermore, set \(p = \frac{n^r}{1+n^r}\) and \(q = \frac{1}{1+n^r}\) where \(r > 0\). Taking \(P_n = \frac{1}{2}B(p) \otimes B(p) + \frac{1}{2} B(q) \otimes B(q)\) and \(Q_n = \frac{1}{2} B(p) \otimes B(q) + \frac{1}{2}B(q) \otimes B(p)\), we are exactly in the setting of testing (\ref{problem:sparse_mixture_detection_1})-(\ref{problem:sparse_mixture_detection_2}). 

    \subsubsection{Checking the tail condition}
    We first check that the tail condition (\ref{eqn:varadhan_tail_condition}) holds. Consider that for \((a, b) \in \{0, 1\}^2\),
    \begin{align*}
        \frac{q_n}{p_n}(a, b) &= \frac{\frac{1}{2}\left(\frac{p}{1-p}\right)^a\left(\frac{q}{1-q}\right)^b(1-p)(1-q) + \frac{1}{2}\left(\frac{q}{1-q}\right)^a \left(\frac{p}{1-p}\right)^b(1-p)(1-q)}{\frac{1}{2}\left(\frac{p}{1-p}\right)^{a+b}(1-p)^2 + \frac{1}{2}\left(\frac{q}{1-q}\right)^{a+b}(1-q)^2} \\
        &= \frac{p^a(1-p)^{1-a} q^b(1-q)^{1-b} + p^b (1-p)^{1-b}q^a(1-q)^{1-a}}{p^{a+b}(1-p)^{2-a-b} + q^{a+b}(1-q)^{2-a-b}}.
    \end{align*}
    It immediately follows that 
    \begin{align*}
        \frac{q_n}{p_n}(0, 0) &= \frac{2(1-p)(1-q)}{(1-p)^2 + (1-q)^2} \leq 1,\\
        \frac{q_n}{p_n}(0, 1) &= \frac{(1-p)q + p(1-q)}{p(1-p) + q(1-q)} = \frac{n^{2r}+1}{2n^r}, \\
        \frac{q_n}{p_n}(1, 0) &= \frac{p(1-q) + (1-p)q}{p(1-p) + q(1-q)} = \frac{n^{2r}+1}{2n^r}, \\
        \frac{q_n}{p_n}(1, 1) &= \frac{2pq}{p^2 + q^2} \leq 1.
    \end{align*}
    Thus for any \(\gamma > 1\), we have for \((a, b) \sim P_n\), 
    \begin{align*}
        E\left[ \left(\frac{q_n}{p_n}(a,b)\right)^\gamma \right] &= \left(\frac{q_n}{p_n}(0, 0)\right)^\gamma \left(\frac{p^2 + q^2}{2} \right) + \left(\frac{q_n}{p_n}(0, 1)\right)^\gamma \left( \frac{p(1-p)+q(1-q)}{2} \right) \\
        &\;\; + \left(\frac{q_n}{p_n}(1, 0)\right)^\gamma \left( \frac{p(1-p)+q(1-q)}{2} \right) + \left(\frac{q_n}{p_n}(1, 1)\right)^\gamma \left( \frac{(1-p)^2 + (1-q)^2}{2} \right) \\
        &\leq 2 + \left(\frac{n^{2r}+1}{2n^r}\right)^\gamma \cdot \left(p(1-p) + q(1-q) \right) \\
        &= 2 + \left(\frac{n^{2r}+1}{2n^r}\right)^\gamma \cdot \frac{n^{2r}+1}{(1+n^r)^2} \\
        &= 2 + \frac{(n^{2r}+1)^{\gamma+1}}{2^\gamma n^{\gamma r}(1+n^r)^2}.
    \end{align*}
    It immediately follows that 
    \begin{align*}
        \limsup_{n \to \infty} \frac{\log E\left[ \left(\frac{q_n}{p_n}(a,b)\right)^\gamma \right] }{\log n} \leq \gamma r < \infty
    \end{align*}
    and so the tail condition (\ref{eqn:varadhan_tail_condition}) is satisfied.

    \subsubsection{Finding the rate function}
    We first establish that the normalized log-likelihood ratio satisfies a large deviation principle under the null. Consider that for \((a, b) \in \{0, 1\}^2\),
    \begin{align*}
        \frac{\log \frac{q_n}{p_n}(a,b)}{\log n} &= \frac{\log\left(\frac{n^r}{(1+n^r)^2}\right)}{\log n} + \frac{\log\left(n^{(a-b)r} + n^{(b-a)r}\right)}{\log n} - \frac{\log\left(n^{(a+b)r}(1+n^r)^{-2} + n^{-(a+b)r}n^{2r}(1+n^r)^{-2} \right)}{\log n} \\
        &= r + |a-b| r + \frac{\log\left(1 + n^{-|a-b|r}\right)}{\log n} - \frac{\log(n^{(a+b)r} + n^{-(a+b-2)r})}{\log n} \\
        &= r\left(1 + |a-b| - (a+b)\vee (2-a-b)\right) + \frac{\log\left(1 + n^{-|a-b|r}\right)}{\log n} \\
        &\;\; 
        - \frac{\log\left(1 + n^{[(a+b)r \wedge (2-a-b)r] - [(a+b)r \vee (2-a-b)r]}\right)}{\log n}.
    \end{align*}
    Let us focus on the sequence of random variables \(r(1+|U_n - V_n| - (U_n+V_n) \vee (2 - U_n - V_n))\) under the null \((U_n, V_n) \sim P_n\). The random variables are supported on \(\{-r, r\}\), in particular we have 
    \begin{align*}
        r(1+|U_n - V_n| - (U_n+V_n) \vee (2 - U_n - V_n)) = 
        \begin{cases}
            r &\text{with probability } \frac{2n^r}{(1+n^r)^2}, \\
            -r &\text{with probability } \frac{1 + n^{2r}}{(1+n^r)^2}.
        \end{cases}
    \end{align*}
    It immediately follows that \(r(1+|U_n - V_n| - (U_n+V_n) \vee (2 - U_n - V_n))\) satisfies a large deviation principle with respect to speed \(\left\{ \frac{1}{\log n} \right\}\) and good rate function \(I : \R \to [0, \infty]\) 
    \begin{equation*}
        I(t) =
        \begin{cases}
            0 &\text{if } t = -r, \\
            r &\text{if } t = r, \\
            \infty &\text{otherwise}.
        \end{cases}
    \end{equation*}
    Note that \(I\) is indeed a good rate function since the sublevel sets for \(\alpha \in [0, \infty)\),
    \begin{equation*}
        \{t \in \R : I(t) \leq \alpha\} =
        \begin{cases}
            \{-r, r\} &\text{if } \alpha \geq r, \\
            \{-r\} &\text{otherwise}
        \end{cases}
    \end{equation*}
    are compact. By an argument similar to the ones in previous examples appealing to exponential equivalence (Definition \ref{def:exp_equiv}) and Theorem \ref{thm:same_ldp}, it follows that \(\left\{\frac{\log \frac{q_n}{p_n}}{\log n}\right\}\) satisfies the large deviation principle under the null with good rate function \(I\).

    \subsubsection{Determining the detection boundary}
    We are now able to determine the detection boundary by an application of Corollary \ref{corollary:tight_limit}. Solving the optimization problem (\ref{eqn:tight_beta}) yields the detection boundary 
    \begin{equation}\label{eqn:sbm_beta}
        \beta^*(r) = \frac{1 + 1\wedge r}{2}.
    \end{equation}
    Note Theorem \ref{thm:HC_tight} cannot be applied to deduce optimality of \(\HC^*_n\) since the rate function \(I\) is nonconvex.

    \subsubsection{Univariate reduction without information loss}
    Examining the testing problem once again, a natural idea is to examine the univariate statistic \(U_{j} = A_{1j} + B_{1j} \bmod 2\) for \(2 \leq j \leq n+1\). Note that \(U_{j} = 0\) if and only if \(A_{1j} = B_{1j}\) and \(U_{j} = 1\) if and only if \(A_{1j} \neq B_{1j}\). In other words, \(U_{j}\) indicates whether the matrices \(A\) and \(B\) match at entry \((1, j)\). Observing numerous instances of \(U_{j}\) equal to \(1\) constitutes as evidence against the null. The sparse mixture detection problem corresponding to (\ref{problem:sbm_problem1})-(\ref{problem:sbm_problem2}) is 
    \begin{align}
        H_0 &: U_{j} \overset{iid}{\sim} \frac{1}{2} B(2p(1-p)) + \frac{1}{2} B(2q(1-q)), \label{problem:sbm_reduce1}\\
        H_1 &: U_{j} \overset{iid}{\sim} (1-\varepsilon) \left[\frac{1}{2} B(2p(1-p)) + \frac{1}{2}B(2q(1-q))\right] + \varepsilon B(p(1-q) + q(1-p)) \label{problem:sbm_reduce2}
    \end{align}
    where the indices run over \(2 \leq j \leq n+1\). As before, we take \(\varepsilon = n^{-\beta}\), \(p = \frac{n^r}{1+n^r},\) and \(q = \frac{1}{1+n^r}\) where \(r > 0\). Furthermore, take \(P_n = \frac{1}{2}B(2p(1-p)) + \frac{1}{2}B(2q(1-q))\) and \(Q_n = B(p(1-q) + q(1-p))\). 
    
    The principal question is whether the reduction to \(U_{j}\) incurs a loss of information. Note that \(U_{j}\) is not an invertible function of the vector \((A_{1j}, B_{1j})\), so it is not clear that the detection boundary for testing (\ref{problem:sbm_reduce1})-(\ref{problem:sbm_reduce2}) is equal to the detection boundary (\ref{eqn:sbm_beta}). In fact, one must derive the detection boundary (\ref{eqn:sbm_beta}) in the context of testing (\ref{problem:sbm_reduce1})-(\ref{problem:sbm_reduce2}) in order to even check that a univariate reduction does not incur information loss; without knowing the detection boundary for the testing problem (\ref{problem:sbm_problem1})-(\ref{problem:sbm_problem2}), how can one verify that the reduction to \(U_{j}\) is safe?
    
    It turns out that reducing to \(U_{j}\) does not incur a loss of information, namely the detection boundary corresponding to (\ref{problem:sbm_problem1})-(\ref{problem:sbm_problem2}) is exactly the same as the detection boundary (\ref{eqn:sbm_beta}). We briefly sketch the derivation and omit details for brevity. The analysis follows the typical roadmap showcased in all of our examples. Firstly, it can be shown through direct calculation that for \(u \in \{0, 1\}\)
    \begin{equation*}
        \frac{\log \frac{q_n}{p_n}(u)}{\log n} = (2u-1) \left[\frac{\log 2}{\log n} + \frac{\log(1+n^{-2r})}{\log n}\right] + (2u-1)r.
    \end{equation*}
    Under the null \(U \sim P_n\), it can be shown that 
    \begin{equation*}
        (2U-1) r =
        \begin{cases}
            -r & \text{with probability } \frac{n^{2r} + 1}{(1+n^r)^2}, \\
            r &\text{with probability } \frac{2n^r}{(1+n^r)^2}.
        \end{cases}
    \end{equation*}
    It immediately follows that \((2U-1)r\) satisfies the large deviation principle with respect to speed \(\left\{\frac{1}{\log n}\right\}\) and with good rate function \(I : \R \to [0, \infty]\) given by 
    \begin{equation*}
        I(t) =
        \begin{cases}
            0 &\text{if } t = -r, \\
            r &\text{if } t = r, \\
            \infty &\text{otherwise}.
        \end{cases}
    \end{equation*}
    It can be shown that \((2U-1)r\) and \(\frac{\log \frac{q_n}{p_n}(U)}{\log n}\) are exponentially equivalent with respect to speed \(\left\{\frac{1}{\log n}\right\}\), and so it follows that \(\left\{\frac{\log \frac{q_n}{p_n}}{\log n}\right\}\) satisfies the large deviation principle under the null with good rate function \(I\). After checking the conditions of Corollary \ref{corollary:tight_limit}, we obtain exactly the detection boundary given by (\ref{eqn:sbm_beta}), thus showing that the reduction to \(U_{j}\) does not incur a loss of information.
    
    \subsection{Detection with side information}
    Occasionally in applications, there is additional side information that may be useful in testing the global null hypothesis (\ref{problem:sparse_mixture_detection_1})-(\ref{problem:sparse_mixture_detection_2}). The testing problem (\ref{problem:sparse_mixture_detection_1})-(\ref{problem:sparse_mixture_detection_2}) admits a Bayesian interpretation as the two-groups model \cite{efronMicroarraysEmpiricalBayes2008}. In particular, we have \(n\) individual hypotheses \(\{H_i\}\) corresponding to each observation \(X_i\). We are testing the global null, i.e. whether the individual hypotheses \(H_i\) are null with probability one (that is, \(X_i \sim P_n\)) for all \(1 \leq i \leq n\), against the alternative in which each \(H_i\) has probability \(\varepsilon\) of being non-null (that is, \(X_i \sim (1-\varepsilon)P_n + \varepsilon Q_n\)). With this interpretation, we are in the setting of multiple testing with \(n\) individual hypotheses. This reformulation is usually the situation in which researchers find themselves; for example, each individual hypothesis might correspond to a different gene in a microarray or a different study in a meta-analysis. In such applications, there is usually detailed contextual information attached to each hypothesis and it's desirable to leverage this side information for various hypothesis testing tasks \cite{liAccumulationTestsFDR2017,liMultipleTestingStructureadaptive2019,ignatiadisCovariatePoweredCrossweighted2017,leiAdaPTInteractiveProcedure2018,ferkingstadUnsupervisedEmpiricalBayesian2008,lewingerHierarchicalBayesPrioritization2007,zablockiCovariatemodulatedLocalFalse2014}. 

    We consider a stylized problem to investigate how side information affects fundamental statistical limits. Consider a sequence of \(z\)-scores \(W_i\) for \(1 \leq i \leq n\) which are \(N(0, 1)\) under the null. Under the alternative, an \(\varepsilon\) fraction of the \(W_i\) exhibit elevated mean, that is, \(W_i \sim N(\mu, 1)\) where \(\mu > 0\). The signal detection problem (\ref{problem:sparse_mixture_detection_1})-(\ref{problem:sparse_mixture_detection_2}) with this setup is precisely the sparse normal mixture detection problem considered by Ingster \cite{ingsterProblemsHypothesisTesting1996} as well as Donoho and Jin \cite{donohoHigherCriticismDetecting2004}. However, suppose for each \(1 \leq i \leq n\), we have additional side information (independent of the \(z\)-scores) that provides a clue as to whether \(W_i\) follows the null distribution or the signal distribution. To represent this side information, we will let \(A_i\) denote a Bernoulli random variable in which the outcome \(A_i = 1\) denotes evidence that \(W_i\) follows the signal distribution and the outcome \(A_i = 0\) denotes evidence that \(W_i\) follows the null distribution. For simplicity, say that under the null \(A_i \sim \Bernoulli(1-p)\) and under the alternative \(A_i \sim \Bernoulli(p)\). In other words, the side information correctly identifies both the null and the signal with probability \(p\). While this is a stylized setup, one can think of \(A_i\) as the outcome of a well-trained classifier applied to the side information or an expert's judgement derived from existing scientific knowledge. Stated formally, we have the testing problem
    \begin{align}
        H_0 : \left(\begin{matrix} A_i \\ W_i \end{matrix}\right) &\overset{iid}{\sim} B(1-p) \otimes N(0, 1), \\
        H_1 : \left(\begin{matrix} A_i \\ W_i \end{matrix}\right) &\overset{iid}{\sim} (1-\varepsilon) B(1-p) \otimes N(0, 1) + \varepsilon B(p) \otimes N(\mu, 1)
    \end{align}
    for \(1 \leq i \leq n\). Here, we use the notation \(B(\pi) = \Bernoulli(\pi)\) for \(\pi \in (0, 1)\). Adopting the asymptotic perspective, let us calibrate \(\varepsilon = n^{-\beta}\) as in (\ref{eqn:beta_sparsity}), let us take \(p = \frac{n^r}{1+n^r}\) for \(r > 0\), and let us take \(\mu = \sqrt{2\rho \log n}\) for \(0 < \rho \leq 1\). Furthermore, let us take \(P_n = B(1-p) \otimes N(0, 1)\) and \(Q_n = B(p) \otimes N(\mu, 1)\). Thus, we are in the setting of testing (\ref{problem:sparse_mixture_detection_1})-(\ref{problem:sparse_mixture_detection_2}). We now determine the phase transition. 

    \subsubsection{Checking the tail condition}
    We first check that the tail condition (\ref{eqn:varadhan_tail_condition}) holds. Consider that for \(a \in \{0, 1\}\) and \(w \in \R\),
    \begin{align*}
        \frac{q_n}{p_n}(a, w) &= \frac{\left(\frac{p}{1-p}\right)^{a}(1-p)\exp\left(-\frac{(w-\mu)^2}{2}\right)}{\left(\frac{1-p}{p}\right)^{a}p\exp\left(-\frac{w^2}{2}\right)} \\
        &= \left(\frac{p}{1-p}\right)^{2a-1} \exp\left(w\mu - \frac{\mu^2}{2}\right).
    \end{align*}
    For any \(\gamma > 1\) and when \((a_n, w_n) \sim P_n\), it follows by independence
    \begin{align*}
        E\left[ \left( \frac{q_n}{p_n}(a_n, w_n) \right)^\gamma \right] &= E\left[ \left(\frac{p}{1-p}\right)^{\gamma(2a-1)}\right] E\left[ \exp\left(\gamma w\mu - \frac{\gamma \mu^2}{2}\right) \right] \\
        &= \left[ \left(\frac{p}{1-p}\right)^{-\gamma} \cdot p + \left(\frac{p}{1-p}\right)^{\gamma}\cdot(1-p)\right] \cdot \exp\left(-\frac{\gamma \mu^2}{2}\right) \exp\left( \frac{\gamma^2\mu^2}{2} \right) \\
        &= p \left[\left(\frac{p}{1-p}\right)^{-\gamma} + \left(\frac{p}{1-p}\right)^{\gamma-1} \right] \cdot \exp\left(-\frac{\gamma \mu^2}{2}\right) \exp\left( \frac{\gamma^2\mu^2}{2} \right) \\
        &= \frac{n^r}{1+n^r} \left[ n^{-\gamma r} + n^{\gamma r - r} \right] \cdot \exp\left((\gamma^2 - \gamma) \rho \log n \right).
    \end{align*}
    Since \(\gamma > 1\) and \(r > 0\), it follows that 
    \begin{equation*}
        \limsup_{n \to \infty} \frac{1}{\log n} \log  E\left[ \left( \frac{q_n}{p_n}(a_n, w_n) \right)^\gamma \right] = \gamma r - r + (\gamma^2 - \gamma)\rho < \infty,
    \end{equation*}
    and so the tail condition (\ref{eqn:varadhan_tail_condition}) is satisfied. 

    \subsubsection{Finding the rate function}
    We now establish a large deviation principle under the null and find the corresponding rate function. Consider that for \(a \in \{0, 1\}\) and \(w \in \R\), we have from our above calculation
    \begin{align*}
        \frac{q_n}{p_n}(a, w) &= \left(\frac{p}{1-p}\right)^{2a-1} \exp\left(w\mu - \frac{\mu^2}{2}\right) \\
        &= \exp\left(w\mu - \frac{\mu^2}{2} + (2ar-r)\log n\right).
    \end{align*}
    Therefore, 
    \begin{equation*}
        \frac{\log \frac{q_n}{p_n}(a, w)}{\log n} = \frac{w\mu - \frac{\mu^2}{2}}{\log n} + 2ar - r.
    \end{equation*}
    To derive the large deviations principle, we follow the same method of argument as in previous examples. Recall that if \((a, w) \sim P_n\), then \(w \sim N(0, 1)\) and \(a \sim B(1-p)\) are independent. Since \(\mu = \sqrt{2\rho \log n}\), it follows that 
    \begin{equation*}
        \frac{w\mu - \frac{\mu^2}{2}}{\log n} \sim N\left(-\rho, \frac{2\rho}{\log n}\right).
    \end{equation*}
    By an argument similar to the one in Example \ref{example:brownian_motion}, it follows that \(\left\{\frac{w\mu - \frac{\mu^2}{2}}{\log n}\right\}\) satisfies a large deviation principle with good rate function \(J_1(t) = \frac{(t+\rho)^2}{4\rho}\) and with respect to speed \(\left\{\frac{1}{\log n}\right\}\). Now, observe that \(2ar-r \in \{-r, r\}\) almost surely. Furthermore, since \(a \sim B(1-p)\),
    \begin{align*}
        \frac{\log P\left\{ 2ar-r = r \right\}}{\log n} &= \frac{\log (1-p)}{\log n} = -\frac{\log(1+n^r)}{\log n} \\
        \frac{\log P\left\{ 2ar-r = -r \right\}}{\log n} &= \frac{\log p}{\log n} = r - \frac{\log(1+n^r)}{\log n}
    \end{align*}
    It immediately follows that \(\{2ar-r\}\) satisfies a large deviations principle with respect to speed \(\left\{\frac{1}{\log n}\right\}\) and good rate function \(J_2 : \R \to [0, \infty]\) given by
    \begin{equation*}
        J_2(t) = 
        \begin{cases}
            r &\text{if } t = r, \\
            0 &\text{if } t = -r, \\
            \infty &\text{otherwise}.
        \end{cases}
    \end{equation*}
    To show that \(\left\{\frac{\log \frac{q_n}{p_n}(a, w)}{\log n} \right\}\) satisfies a large deviations principle under the null, consider the function \(f : \R^2 \to \R\) given by \(f(u, v) = u+v\). Consider that \(f\) is a continuous function. Furthermore, recall that \(\frac{w\mu - \frac{\mu^2}{2}}{\log n}\) and \(2ar-r\) are independent. Since \(\left\{\frac{w\mu - \frac{\mu^2}{2}}{\log n}\right\}\) and \(\left\{2ar-r \right\}\) satisfy large deviation principles with good rate functions (and so are exponentially tight by Exercise 1.2.19 of \cite{demboLargeDeviationsTechniques2010}), it follows by Exercise 4.2.7 of \cite{demboLargeDeviationsTechniques2010} that \(\left\{f\left(\frac{w\mu - \frac{\mu^2}{2}}{\log n} , 2ar-r\right)\right\}\) satisfies a large deviation principle with respect to speed \(\left\{\frac{1}{\log n}\right\}\) and good rate function \(I : \R \to [0, \infty]\) given by 
    \begin{align*}
        I(t) &=\inf\left\{J_1(x) + J_2(y) : t = x+y \right\} \\
        &= (J_1(t-r)+r) \wedge J_1(t+r) \\
        &= \left[\frac{(t-r+\rho)^2}{4\rho}+r \right] \wedge \frac{(t+r+\rho)^2}{4\rho}.
    \end{align*}
    It can be directly checked that \(I(t) = \frac{(t-r+\rho)^2}{4\rho} + r\) whenever \(t \geq 0\) since \(r, \rho > 0\). Since \(\frac{\log \frac{q_n}{p_n}}{\log n}(a, w) = f\left(\frac{w\mu - \frac{\mu^2}{2}}{\log n} , 2ar-r\right)\), we have shown that \(\left\{\frac{\log \frac{q_n}{p_n}}{\log n}\right\}\) satisfies a large deviation principle under the null with good rate function \(I\). 

    \subsubsection{Determining the detection boundary}
    We can now determine the detection boundary. By Corollary \ref{corollary:nice_tight_limit}, it follows that the detection boundary is given by
    \begin{equation*}
        \beta^*(r, \rho) = \frac{1}{2} + 0 \vee \sup_{t \geq 0} \left\{t -\frac{(t-r+\rho)^2}{4\rho} - r + \frac{1 \wedge \left(\frac{(t-r+\rho)^2}{4\rho} + r\right)}{2} \right\}
    \end{equation*}
    since \(I(t) = \frac{(t-r+\rho)^2}{4\rho} + r\) whenever \(t \geq 0\). First, consider that if \(r \geq 1\), then we must have \(\frac{(t-r+\rho)^2}{4\rho} + r \geq 1\). Thus, if \(r \geq 1\), then 
    \begin{equation*}
        \beta^*(r, \rho) = \frac{1}{2} + 0 \vee \sup_{t \geq 0} \left\{ t - \frac{(t-r+\rho)^2}{4\rho} - r + \frac{1}{2} \right\}.    
    \end{equation*}
    Direct calculation shows that the maximum is achieved at \(t = \rho + r\). Hence, if \(r \geq 1\), then \(\beta^*(r, \rho) = 1\).\newline
    
    Now let us consider the case \(r < 1\). Note that \(\frac{(t-r+\rho)^2}{4\rho} + r \leq 1\) if and only if \(-2\sqrt{\rho(1-r)} + r - \rho \leq t \leq 2\sqrt{\rho(1-r)} + r - \rho\). Note that since \(0 < r, \rho \leq 1\), we have \(2\sqrt{\rho(1-r)} + r - \rho \geq 0\). We now consider two cases. 
    
    \textbf{Case 1:} Suppose \(-2\sqrt{\rho(1-r)}+r-\rho \geq 0\). Then \(\beta^*(r, \rho) = \frac{1}{2} + 0 \vee \left[ E_1 \vee E_2 \vee E_3 \right]\) where 
    \begin{align*}
        E_1 &:= \sup_{-2\sqrt{\rho(1-r)} + r - \rho \leq t \leq 2\sqrt{\rho(1-r)} + r - \rho} \left\{t -\frac{(t-r+\rho)^2}{8\rho} - \frac{r}{2} \right\} \\
        E_2 &:= \sup_{0 \leq t \leq -2\sqrt{\rho(1-r)} + r - \rho} \left\{t -\frac{(t-r+\rho)^2}{4\rho} - r + \frac{1}{2} \right\} \\
        E_3 &:= \sup_{t \geq 2\sqrt{\rho(1-r)} + r - \rho} \left\{t -\frac{(t-r+\rho)^2}{4\rho} - r + \frac{1}{2} \right\}.
    \end{align*}
    We examine each term separately. Looking at \(E_1\) first, let us define the function \(f\) with \(f(t) := t - \frac{(t-r+\rho)^2}{8\rho} - \frac{r}{2}\). Taking a derivative, we see that \(f'(t) = \frac{3}{4} - \frac{t}{4\rho} + \frac{r}{4\rho}\). Hence, for \(-2\sqrt{\rho(1-r)} + r - \rho \leq t \leq 2\sqrt{\rho(1-r)} + r - \rho\), we have that \(f'(t) \geq 1 - \frac{2\sqrt{\rho(1-r)}}{4\rho}\). Consequently, if \(\rho > \frac{1-r}{4}\), then \(f'(t) > 0\) for all \(t\) which we maximize over. Thus, the maximum is achieved at the right endpoint and so is given by \(f(2\sqrt{\rho(1-r)}+r-\rho) = 2\sqrt{\rho(1-r)} + r - \rho - \frac{1}{2}\). On the other hand, if \(\rho \leq \frac{1-r}{4}\), then the maximum is achieved at the point \(t = 3\rho + r\) which yields \(f(3\rho+r) = \rho + \frac{r}{2}\). To summarize, 
    \begin{equation*}
        E_1 = 
        \begin{cases}
            2\sqrt{\rho(1-r)} + r - \rho - \frac{1}{2} & \text{if } \rho > \frac{1-r}{4}, \\
            \rho + \frac{r}{2} &\text{if } \rho \leq \frac{1-r}{4}.
        \end{cases}
    \end{equation*}

    We now examine \(E_2\). Define the function \(f(t) = t - \frac{(t-r+\rho)^2}{4\rho} - r + \frac{1}{2}\) and consider that \(f'(t) = \frac{1}{2} - \frac{t}{2\rho} + \frac{r}{2\rho}\). Hence, for \(0 \leq t \leq -2\sqrt{\rho(1-r)}+r-\rho\), we have \(f'(t) \geq \frac{1}{2} - \frac{-2\sqrt{\rho(1-r)}+r-\rho}{2\rho} + \frac{r}{2\rho} = 1 + \sqrt{\frac{1-r}{\rho}} > 0\). Consequently, \(f\) is maximized at the right endpoint yielding \(f(-2\sqrt{\rho(1-r)} + r - \rho) = -2\sqrt{\rho(1-r)} - \rho + r - \frac{1}{2}\). To summarize, 
    \begin{equation*}
        E_2 = -2\sqrt{\rho(1-r)} - \rho + r - \frac{1}{2}.
    \end{equation*}
    
    We now turn our attention to \(E_3\). As we mentioned in the analysis of \(E_1\), note that \(2\sqrt{\rho(1-r)} + r - \rho \geq 0\) since \(0 < r, \rho \leq 1\). Thus, the maximization is over \(t \geq 2\sqrt{\rho(1-r)} + r - \rho\) in the definition of \(E_3\). Define \(f(t) = t - \frac{(t-r + \rho)^2}{4\rho} - r + \frac{1}{2}\) and observe that \(f'(t) = \frac{1}{2} - \frac{t}{2\rho} + \frac{r}{2\rho}\). Consequently, we have \(f'(t) \leq \frac{1}{2} - \frac{2\sqrt{\rho(1-r)}+r-\rho}{2\rho} + \frac{r}{2\rho} = 1 - \sqrt{\frac{1-r}{\rho}}\) for all \(t \geq 2\sqrt{\rho(1-r)}+r-\rho\). So if \(\rho \leq 1-r\), then \(f'(t) \leq 0\) for all \(t \geq 2\sqrt{\rho(1-r)}+r-\rho\), meaning that the maximum is achieved at the left endpoint. The maximum is then \(f(2\sqrt{\rho(1-r)}+r-\rho) = 2\sqrt{\rho(1-r)} + r - \rho - \frac{1}{2}\). On the other hand, if \(\rho > 1-r\), the \(f\) is maximized at \(t = \rho + r\), which yields \(f(\rho + r) = \frac{1}{2}\). To summarize,
    \begin{equation*}
        E_3 = 
        \begin{cases}
            2\sqrt{\rho(1-r)} + r - \rho - \frac{1}{2} & \text{if } \rho \leq 1-r,\\
            \frac{1}{2} &\text{if } \rho > 1-r.
        \end{cases}
    \end{equation*}
    Direct comparison yields 
    \begin{equation*}
        E_1 \vee E_3 = 
        \begin{cases}
            \frac{1}{2} & \text{if } \rho > 1-r, \\
            2\sqrt{\rho(1-r)} + r - \rho - \frac{1}{2} & \text{if } \frac{1-r}{4} < \rho \leq 1-r, \\
            \rho + \frac{r}{2} &\text{if } \rho \leq \frac{1-r}{4}. 
        \end{cases}
    \end{equation*}
    Another direct comparison yields 
    \begin{equation*}
        E_1 \vee E_2 \vee E_3 = E_2 \vee (E_1 \vee E_3) = 
        \begin{cases}
            \frac{1}{2} & \text{if } \rho > 1-r, \\
            2\sqrt{\rho(1-r)} + r - \rho - \frac{1}{2} & \text{if } \frac{1-r}{4} < \rho \leq 1-r, \\
            \rho + \frac{r}{2} &\text{if } \rho \leq \frac{1-r}{4}. 
        \end{cases}
    \end{equation*}
    Consequently, if \(-2\sqrt{\rho(1-r)}+r-\rho \geq 0\) and \(r < 1\), then 
    \begin{equation*}
        \beta^*(r, \rho) = 
        \begin{cases}
            1 & \text{if } \rho > 1 - r, \\
            2\sqrt{\rho(1-r)}+r-\rho &\text{if } \frac{1-r}{4} < \rho \leq 1-r, \\
            \frac{1}{2} + \rho+\frac{r}{2} &\text{if } \rho \leq \frac{1-r}{4}.
        \end{cases}
    \end{equation*}
    We move on to the remaining case where \(r < 1\) and \(-2\sqrt{\rho(1-r)} + r - \rho < 0\).

    \textbf{Case 2:} Suppose \(-2\sqrt{\rho(1-r)}+r-\rho < 0\). Then \(\beta^*(r, \rho) = \frac{1}{2} + 0 \vee \left[ F_1 \vee F_2 \right]\) where 
    \begin{align*}
        F_1 &:= \sup_{0\leq t \leq 2\sqrt{\rho(1-r)} + r - \rho} \left\{t -\frac{(t-r+\rho)^2}{8\rho} - \frac{r}{2} \right\} \\
        F_3 &:= \sup_{t \geq 2\sqrt{\rho(1-r)} + r - \rho} \left\{t -\frac{(t-r+\rho)^2}{4\rho} - r + \frac{1}{2} \right\}.
    \end{align*}
    Repeating the analysis in Case 1 immediately yields 
    \begin{equation*}
        F_1 = 
        \begin{cases}
            2\sqrt{\rho(1-r)} + r - \rho - \frac{1}{2} & \text{if } \rho > \frac{1-r}{4}, \\
            \rho + \frac{r}{2} &\text{if } \rho \leq \frac{1-r}{4}.
        \end{cases}
    \end{equation*}
    and 
    \begin{equation*}
        F_2 = 
        \begin{cases}
            2\sqrt{\rho(1-r)} + r - \rho - \frac{1}{2} & \text{if } \rho \leq 1-r,\\
            \frac{1}{2} &\text{if } \rho > 1-r.
        \end{cases}
    \end{equation*}
    A direct comparison gives 
    \begin{equation*}
        F_1 \vee F_2 =
        \begin{cases}
            \frac{1}{2} &\text{if } \rho > 1-r,\\
            2\sqrt{\rho(1-r)}+r-\rho-\frac{1}{2} &\text{if } \frac{1-r}{4} < \rho \leq 1-r, \\
            \rho + \frac{r}{2} &\text{if } \rho \leq \frac{1-r}{4}. 
        \end{cases}
    \end{equation*}
    Consequently, if \(-2\sqrt{\rho(1-r)} + r- \rho < 0\) and \(r < 1\), then 
    \begin{equation*}
        \beta^*(r, \rho) = 
        \begin{cases}
            1 & \text{if } \rho > 1 - r, \\
            2\sqrt{\rho(1-r)}+r-\rho &\text{if } \frac{1-r}{4} < \rho \leq 1-r, \\
            \frac{1}{2} + \rho+\frac{r}{2} &\text{if } \rho \leq \frac{1-r}{4}.
        \end{cases}
    \end{equation*}

    Putting all of the cases together, the final detection boundary (recalling \(r > 0\) and \(0 < \rho \leq 1\)) is given by
    \begin{equation*}
        \beta^*(r, \rho) = 
        \begin{cases} 
            1 & \text{if } r \geq 1, \\
            1 & \text{if } r < 1, \rho > 1 - r, \\
            2\sqrt{\rho(1-r)}+r-\rho &\text{if } r < 1, \frac{1-r}{4} < \rho \leq 1-r, \\
            \frac{1}{2} + \rho+\frac{r}{2} &\text{if } r < 1, \rho \leq \frac{1-r}{4}.
        \end{cases}
    \end{equation*}
    We can reexpress the detection boundary in a more reminiscent form
    \begin{equation}\label{eqn:side_info_boundary}
        \beta^*(r, \rho) =
        \begin{cases}
            1 - (\sqrt{(1-r)_{+}} - \sqrt{\rho})_{+}^2 & \text{if } \frac{1-r}{4} < \rho \leq 1, \\
            \frac{1}{2} + \rho + \frac{r}{2} & \text{if } 0 < \rho \leq \frac{1-r}{4}.
        \end{cases}
    \end{equation}

    The detection boundary \(\beta^*(r, \rho)\) given by (\ref{eqn:side_info_boundary}) is exactly the Ingster-Donoho-Jin detection boundary when we naively ``plug-in'' \(r = 0\) without care (see Example \ref{example:idj}). This is entirely as expected since the case \(r = 0\) corresponds to the setting where the sequence of Bernoulli random variables \(\{A_i\}\) does not provide strong enough information on the location of the sparse signals. Only the sequence of Gaussian variables \(\{W_i\}\) exhibits strong enough signal, and so one can simply throw out the \(\{A_i\}\) sequence without loss of power. On the other hand, consider that simply ``plugging in" \(\rho = 0\) into the detection boundary formula yields \(\beta^*(r, 0) = \frac{1 + 1\wedge r}{2}\). In this setting, the detection boundary is exactly the boundary one would obtain by throwing out the Gaussian variables \(\{W_i\}\) (which now have no signal) and only using the Bernoulli sequence \(\{A_i\}\) for detection. Indeed, \(\beta^*(r, 0) = \frac{1 + 1\wedge r}{2}\) is precisely the boundary one obtains through the rate function \(J_2\) (which is the rate function of the large deviations principle under the null associated with \(\{A_i\}\)).

    In the intermediate regimes \(0 < r < 1\) and \(0 < \rho < 1\), the detection boundary \(\beta^*(r, \rho)\) is larger than the Ingster-Donoho-Jin detection boundary and larger than the boundary \(\frac{1+1\wedge r}{2}\) corresponding to using only the Bernoulli sequence \(\{A_i\}\). In words, using both the Gaussian and the Bernoulli data yields higher detection boundaries (meaning we are able to detect weaker signals in sparser settings) compared to using only the Gaussian data or only the Bernoulli data. This result is intuitively unsurprising, and it's comforting to see this phenomenon play out in the detection boundary. Furthermore, \(\beta^*(r, \rho)\) gives a precise description of how the signal strengths \(r\) and \(\rho\) in the Bernoulli and Gaussian sequences relate to one another and affect the phase transition.

    \subsection{Curie-Weiss Model}\label{example:curie_weiss_mean}
    In this subsection, we briefly cover a sparse mixture detection problem in the larger context of Ising models. There is an existing literature focused on inferential tasks given multiple independent and identically distributed samples from an unknown probabilistic graphical model, such as graph selection \cite{santhanamInformationTheoreticLimitsSelecting2012a,barberHighdimensionalIsingModel2015} and property testing/goodness-of-fit testing \cite{canonneTestingBayesianNetworks2020,bezakovaLowerBoundsTesting2019}. Note that in the context of the sparse mixture detection problems (\ref{problem:sparse_mixture_detection_1})-(\ref{problem:sparse_mixture_detection_2}), we have \(n\) observations in which possibly \((1-\varepsilon)n\) are drawn from the Ising model \(P_n\) and \(\varepsilon n\) are drawn from the Ising model \(Q_n\). The problem, of course, is to detect the presence of observations drawn from the Ising model \(Q_n\).

    We will consider the sample space \(\mathcal{X} = \bigcup_{n=1}^{\infty} \{-1, 1\}^n\). Consider the testing problem (\ref{problem:sparse_mixture_detection_1})-(\ref{problem:sparse_mixture_detection_2}) where for \(x \in \mathcal{X}\),
    \begin{align*}
        p_n(x) &= \frac{1}{Z_N(\theta, 0)} \exp\left( \frac{\theta}{N} \sum_{1 \leq i < j \leq N} x_ix_j \right) \cdot \mathbf{1}_{\{x \in \{-1, 1\}^N\}}, \\
        q_n(x) &= \frac{1}{Z_N(\theta, \mu)} \exp\left(\frac{\theta}{N} \sum_{1 \leq i < j \leq N} x_i x_j + \theta \mu \sum_{i=1}^{N} x_i \right) \cdot \mathbf{1}_{\{x \in \{-1, 1\}^N\}}
    \end{align*}
    where \(N = \lceil \log n \rceil\). Further, \(\theta > 0\) and \(\mu > 0\) are parameters we take to be fixed and unchanging with \(n\). In the parlance of statistical mechanics, \(p_n\) and \(q_n\) are the Curie-Weiss model on \(N\) particles. Each particle takes one of two states \(x_i \in \{-1, 1\}\) for all \(1 \leq i \leq N\). Additionally \(q_n\) models the existence of an external magnetic field with strength \(\mu\) and \(p_n\) models the absence of an external magnetic field. The quantity \(Z_N(\theta, \mu)\) is the normalizing constant or ``partition function'' in the statistical mechanics convention. Note that under both \(p_n\) and \(q_n\), all of the random variables \(x_1,...,x_N\) are all correlated with each other and the parameter \(\theta\) controls the strength of this correlation. Intuitively, the correlation is strong when \(\theta\) is large and the correlation is weak when \(\theta\) is small. Additionally, the parameter \(\mu\) modulates the probability that each particle takes value \(1\) instead of \(-1\). Further background about the Curie-Weiss model (and related Ising models) can be found in Chapter 2 of \cite{mezardInformationPhysicsComputation}. We will refer to some results found there.

    \subsubsection*{Checking the tail condition}
    To apply Corollary \ref{corollary:tight_limit}, we first check that the tail condition (\ref{eqn:varadhan_tail_condition}) is satisfied. Letting \(X_n \sim P_n\), consider that for \(\gamma > 1\), 
    \begin{align*}
        E\left[\left(\frac{q_n}{p_n}(X_n)\right)^\gamma \right] &= \sum_{x \in \{-1, 1\}^N} \frac{Z_N(\theta, 0)^{\gamma-1}}{Z_N(\theta, \mu)^\gamma} \exp\left( \frac{\theta}{N} \sum_{1 \leq i < j \leq N} x_i x_j + \gamma \theta \mu \sum_{i=1}^{N} x_i \right)\\
        &= \frac{Z_N(\theta, 0)^{\gamma - 1}}{Z_N(\theta, \mu)^\gamma} \cdot Z_N(\theta, \gamma \mu).
    \end{align*}
    Hence, 
    \begin{align*}
        &\limsup_{n \to \infty} \frac{1}{\log n} \log E\left[\left(\frac{q_n}{p_n}(X_n)\right)^\gamma \right] \\
        &= \limsup_{n \to \infty} \frac{(\gamma - 1) \log Z_N(\theta, 0)}{\log n} - \frac{\gamma \log Z_N(\theta, \mu)}{\log n} + \frac{\log Z_N(\theta, \gamma \mu)}{\log n} \\
        &= \limsup_{n \to \infty} \frac{N}{\log n} \left[\frac{(\gamma - 1) \log Z_N(\theta, 0)}{N} - \frac{\gamma \log Z_N(\theta, \mu)}{N} + \frac{\log Z_N(\theta, \gamma \mu)}{N} \right] \\
        &< \infty
    \end{align*}
    since \(N = \lceil \log n \rceil\) and by (2.79) in Section 2.5.2 of \cite{mezardInformationPhysicsComputation}. Thus, the tail condition is satisfied.

    \subsubsection*{Finding the rate function}
    We now deduce a large deviation principle under the null. Observe that for \(X \in \{-1, 1\}^N\),
    \begin{align*}
        \frac{\log \frac{q_n}{p_n}(X)}{\log n} &= \frac{\log Z_N(\theta, 0)}{\log n} - \frac{\log Z_N(\theta, \mu)}{\log n} + \frac{\theta \mu \sum_{i=1}^{N} X_i}{\log n} \\
        &= \frac{N}{\log n} \cdot \frac{\log Z_N(\theta, 0)}{N} - \frac{N}{\log n} \cdot \frac{\log Z_N(\theta, \mu)}{N} + \frac{N}{\log n} \cdot \theta \mu \bar{X}.
    \end{align*}
    First, consider that \(\lim_{n \to \infty} \frac{N}{\log n} = 1\). Second, consider that by (2.79) in Section 2.5.2 of \cite{mezardInformationPhysicsComputation} it follows that 
    \begin{equation*}
        \lim_{N \to \infty} \frac{\log Z_N(\theta, \mu)}{N} = \max_{m \in [-1, 1]} \varphi_{mf}(m;\theta, \mu) 
    \end{equation*}
    where 
    \begin{equation*}
        \varphi_{mf}(m;\theta, \mu) = -\frac{\theta}{2} (1 - m^2) + \theta \mu m - \frac{1+m}{2} \log\left(\frac{1+m}{2}\right) - \frac{1-m}{2}\log \left(\frac{1-m}{2} \right).
    \end{equation*}
    For ease of notation, we set \(M^*(\theta, \mu) = \max_{m \in [-1, 1]} \varphi_{mf}(m;\theta,\mu)\). Lastly, consider that by (4.35) in Section 4.3.2 of \cite{mezardInformationPhysicsComputation}, it follows that, under the null, \(\bar{X}\) satisfies a large deviation principle with respect to speed \(\left\{\frac{1}{N}\right\}\) and rate function 
    \begin{equation*}
        J(m) =
        \begin{cases} 
            M^*(\theta, 0) - \varphi_{mf}(m; \theta, 0) &\text{if } |m| \leq 1, \\
            \infty &\text{otherwise}.
        \end{cases}
    \end{equation*}
    Note that \(J\) is a good rate function because \(J\) is continuous for \(|m| \leq 1\) and so the sublevel sets \(J^{-1}([0, \alpha])\) are compact because \([0, \alpha]\) is compact for \(0 \leq \alpha < \infty\).

    Thus, it follows by the contraction principle (Theorem \ref{thm:contraction_principle}) that, under the null, \(M^*(\theta, 0) - M^*(\theta, \mu) + \theta \mu \bar{X}\) satisfies a large deviation principle with good rate function 
    \begin{align*}
        I(t) &= \inf\left\{J(m) : t = M^*(\theta, 0) - M^*(\theta, \mu) + \theta \mu m \right\} \\
        &= \begin{cases}
            M^*(\theta, 0) - \varphi_{mf}\left(\frac{t - M^*(\theta, 0) + M^*(\theta, \mu)}{\theta \mu}; \theta, 0\right) &\text{if } \left|\frac{t - M^*(\theta, 0) + M^*(\theta, \mu)}{\theta \mu} \right| \leq 1, \\
            \infty &\text{otherwise}.
        \end{cases}
    \end{align*}
    Since \(\lim_{n \to \infty} \frac{N}{\log n} = 1\) and \(-1 \leq \bar{X} \leq 1\) almost surely, a straight forward argument via the notion of exponential equivalence (\ref{def:exp_equiv}) and Theorem \ref{thm:same_ldp} establishes that \(\left\{\frac{\log \frac{q_n}{p_n}}{\log n}\right\}\) satisfies a large deviation principle under the null with good rate function \(I\). 

    With the large deviation principle under the null established, one can apply Theorems \ref{thm:beta_upper} and \ref{thm:beta_lower} to obtain \(\overline{\beta}^\#\) and \(\underline{\beta}^\#\) respectively. Then, one must check whether the conditions of Corollary \ref{corollary:tight_limit} or Corollary \ref{corollary:nice_tight_limit} hold to deduce \(\beta^*\). We do not undertake the analysis. It should be noted that the behavior of the Curie-Weiss model has an effect on the rate function \(I\) (and consequently the detection boundary and the applicability of Theorem \ref{thm:HC_tight}). As noted in Example 4.11 of \cite{mezardInformationPhysicsComputation}, the rate function \(J\) is convex when \(\theta < 1\) and non-convex when \(\theta > 1\). Interestingly, this phase transition in the qualitative behavior of the Curie-Weiss model in terms of \(\theta\) has an effect on the detection limits of the sparse mixture detection problem. Similarly, the applicability of Theorem \ref{thm:HC_tight} regarding \(\HC^*_n\) also depends on the phase transition driven by \(\theta\).

    \section{Discussion}\label{section:discussion}
    We have offered a unified perspective on deriving phase transitions in general sparse mixture detection problems via the large deviations theory. The fundamental object determining the phase transition is the the rate function associated to the large deviation principle of the normalized log likelihood ratios. The core phenomenon behind the phase transition lies in the asymptotics of the Hellinger distance between \(P_n\) and \((1-n^{-\beta}) P_n + n^{-\beta}Q_n\) as identified by Cai and Wu \cite{caiOptimalDetectionSparse2014}; the large deviations theory provides suitable machinery to relate Hellinger asymptotics to phase transitions beyond the univariate sparse mixture case.
    
    Additionally, we have obtained sufficient conditions on the rate function to guarantee the optimality of a sequence of tests based on a Higher Criticism type statistic formulated by Gao and Ma (Section 3.2 of \cite{gaoTestingEquivalenceClustering2019}). This statistic \(\HC^*_n\) adapts to the signal sparsity \(\beta\) and can be used ``off-the-shelf''; careful and delicate constructions of univariate \(p\)-values tailored to the detection problem at hand are not needed. Moreover, as we discussed in Section \ref{section:HC}, computation of \(\HC^*_n\) need not require full knowledge of the signal distributions \(\{Q_n\}\). Rather, in some problems it may suffice to consider a certain statistic of the data; of course, considerations will vary on a problem-to-problem basis. 

    We imagine that the large deviations perspective offered here will be useful in deriving phase transitions in more complicated and structured sparse mixture detection problems beyond what can be derived with the existing univariate theory. Further, we imagine that Gao and Ma's testing statistic will be practically useful in light of Theorem \ref{thm:HC_tight}. We conclude with a few remarks. 

    \subsection{Deriving large deviation principles and rate functions}
    The main results regarding detection boundaries we've presented only specify how the detection boundary is determined by the rate function when the normalized log likelihood ratios satisfy a suitable large deviation principle. These results have nothing to say about how to deduce a large deviation principle and calculate the associated rate function. This is not so surprising given the broad setting and the fundamental role of the rate function. Indeed, the main technical work in specific problems is to deduce the large deviation principle and the associated rate function. In the few examples we presented in Section \ref{section:examples}, we have illustrated a a small number of techniques useful to establishing the large deviation principle. Indispensable are the contraction principle (Theorem \ref{thm:contraction_principle}), exponential equivalence (Definition \ref{def:exp_equiv}), and the indistinguishability of the large deviation principle for exponentially equivalent probability measures (Theorem \ref{thm:same_ldp}). Lemma 3 in \cite{caiOptimalDetectionSparse2014} was also quite useful in our examples in calculating the order of some exponential integrals. 

    \subsection{Sparse mixture of exponentials: necessity of a tail condition} \label{example:sparse_exponential}
    Unfortunately, the tail condition (\ref{eqn:varadhan_tail_condition}) in Theorem \ref{thm:beta_upper} can preclude calculation of a detection boundary in some problems. For example, consider the testing problem (\ref{problem:sparse_mixture_detection_1})-(\ref{problem:sparse_mixture_detection_2}) with calibration (\ref{eqn:beta_sparsity}) and \(P_n = P = \Exp(1)\), \(Q_n = \Exp(1+n^r)\) for \(r > 0\) under the scale parameterization (i.e. \(P_n\) has mean \(1\) and \(Q_n\) has mean \(1+n^r)\). With \(X \sim P\), observe that for any \(\gamma > 1\)
    \begin{align*}
        E\left[\left(\frac{q_n}{p}(X) \right)^\gamma \right] &= \left(\frac{1}{1+n^r}\right)^\gamma \int_{0}^{\infty} \exp\left(-x\gamma \left(\frac{1}{1+n^r} - 1\right) \right) \cdot \exp(-x) \, dx \\
        &= \left(\frac{1}{1+n^r}\right)^\gamma \int_{0}^{\infty} \exp\left(x \left(\gamma \frac{n^r}{1+n^r} - 1 \right) \right) \, dx \\
        &= \left(\frac{1}{1+n^r}\right)^\gamma \int_{0}^{\infty} \exp\left(x \left[\gamma - 1 + \gamma \left(\frac{n^r}{1+n^r} - 1\right) \right]\right)\, dx.
    \end{align*}
    Since \(\gamma > 1\) and \(\frac{n^r}{1+n^r} \to 1\), it follows that \(\gamma - 1 + \gamma \left(\frac{n^r}{1+n^r} - 1 \right) > 0\) for all \(n\) sufficiently large. Therefore, 
    \begin{equation*}
        \int_{0}^{\infty} \exp\left(x \left[\gamma - 1 + \gamma \left(\frac{n^r}{1+n^r} - 1\right) \right]\right)\, dx = \infty
    \end{equation*}
    for all \(n\) sufficiently large. Thus, \(\limsup_{n \to \infty} \frac{1}{\log n} \log E\left[\left(\frac{q_n}{p}(X) \right)^\gamma \right] = \infty\) for all \(\gamma > 1\), and so the tail condition (\ref{eqn:varadhan_tail_condition}) fails to hold. However, the testing problem (\ref{problem:sparse_mixture_detection_1})-(\ref{problem:sparse_mixture_detection_2}) with \(P = \Exp(1)\) and \(Q_n = \Exp(1+n^r)\) indeed exhibits a detection boundary. Corollary 4.4 of \cite{ditzhausSignalDetectionPhidivergences2019} indicates that \(\beta^* = \frac{1 + 1 \wedge r}{2}\).

    One might argue that the tail condition is simply too strong and that the results of Theorems \ref{thm:beta_upper}, \ref{thm:beta_lower} and Corollary \ref{corollary:tight_limit} might still hold when \(P_n = P = \Exp(1)\) and \(Q_n = \Exp(1+n^r)\). We show that this is not the case. Consider that 
    \begin{align*}
        \frac{\log \frac{q_n}{p}(X)}{\log n} = -\frac{\log(1+n^r)}{\log n} + \frac{X}{\log n} \cdot \frac{n^r}{1+n^r}.
    \end{align*}
    Consider that \(-\frac{\log(1+n^r)}{\log n} = -r + o(1)\) as \(n \to \infty\) since \(r > 0\). Consider further that \(\frac{X}{\log n}\) and \(\frac{X}{\log n} \frac{n^r}{1+n^r}\) are exponentially equivalent with respect to speed \(\left\{\frac{1}{\log n}\right\}\) (see Definition \ref{def:exp_equiv}) when \(X \sim P\). Indeed for any \(\delta > 0\),
    \begin{align*}
        \limsup_{n \to \infty} \frac{\log P\left(\left|\frac{X}{\log n} - \frac{X}{\log n} \cdot \frac{n^r}{1+n^r} \right| > \delta \right)}{\log n} &= \limsup_{n \to \infty} \frac{\log P\left(\frac{X}{(1+n^r) \log n} > \delta\right)}{\log n} \\
        &= \limsup_{n \to \infty} \frac{- \delta (1+n^r) \log n}{\log n} \\
        &= -\infty.
    \end{align*}
    By Lemma \ref{lemma:ldp_Yn}, it follows that \(\left\{\frac{X}{\log n}\right\}\) satisfies a large deviation principle with speed \(\left\{\frac{1}{\log n}\right\}\) and rate function \(J\) with \(J(t) = t\) for \(t \geq 0\) and \(J(t) = \infty\) for \(t < 0\). A similar argument as the one above shows that \(-r + \frac{X}{\log n}\) is exponentially equivalent to \(- \frac{\log(1+n^r)}{\log n} + \frac{X}{\log n} \cdot \frac{n^r}{1+n^r}\) with respect to speed \(\left\{\frac{1}{\log n}\right\}\). Applying the contraction principle (Theorem \ref{thm:contraction_principle}) establishes that \(-r + \frac{X}{\log n}\) satisfies a large deviation principle with rate function \(I\) with \(I(t) = t+r\) for \(t \geq -r\) and \(I(t) = \infty\) otherwise. When \(r \leq \frac{1}{2}\), the results of Theorems \ref{thm:beta_upper}, \ref{thm:beta_lower} and Corollary \ref{corollary:tight_limit} indicate \(\beta^* = 1 - r\). This is not sensible since \(\beta^*\) decreases as \(r\) increases, yet the testing problem should be easier with larger \(r\). Thus, a tail condition like (\ref{eqn:varadhan_tail_condition}) is indeed necessary.

    The tail condition is crucial to our large deviations approach as the core of our approach relies on Varadhan's integral lemma (Theorem \ref{thm:Varadhan}). It's not clear to us how to treat the sparse exponential mixture testing problem through our large deviations approach.

    \subsection{Further generalizations}
    Theorems \ref{thm:beta_upper} and \ref{thm:beta_lower} only present upper and lower bounds on \(\overline{\beta}^*\) and \(\underline{\beta}^*\) respectively. In this paper, we were only interested in when these bounds meet (Corollaries \ref{corollary:tight_limit} and \ref{corollary:nice_tight_limit}). It is an open problem to give tight characterizations of \(\overline{\beta}^*\) and \(\underline{\beta}^*\). Likewise, it's of interest to furnish an example where normalized log likelihood ratios satisfy a large deviation principle under the null and where \(\underline{\beta}^*\) and \(\overline{\beta}^*\) do not meet. 

    Finally, it's of interest to develop results in the setting where the observations are correlated rather than independent and identically distributed. The approach of characterizing the Hellinger asymptotics is no longer tenable as this method exploited the tensorization property of the Hellinger distance over product measures. Both problems of determining the phase transitions and developing optimal procedures are open. In the normal mixture setting, Hall and Jin \cite{hallInnovatedHigherCriticism2010} develop the Innovated Higher Criticism. Remarkably, Hall and Jin show that signal detection can actually be easier in some cases; the independent noise case is statistically the hardest. We refer the interested reader to further discussion in that paper as well as the review article \cite{jinRareWeakEffects2016}. 

    \section{Proofs}\label{section:proofs}

    Proofs for the results presented in the main body of the paper are stated in this section.

    \subsection{Useful results}\label{subsection:tools}
    A major tool used in establishing the detection limits of Section \ref{section:detection} and investigating the higher criticism type statistic of Section \ref{section:HC} is Varadhan's integral lemma from the theory of large deviations. Varadhan's integral lemma is essentially a generalized version of Laplace's method, and is useful in characterizing the asymptotics of certain integrals in the presence of probability measures satisfying the large deviation principle. The following formulation of Varadhan's integral lemma is found in \cite{demboLargeDeviationsTechniques2010} (Theorem 4.3.1), with specialization to suit our setting. 

    \begin{theorem}(Varadhan)\label{thm:Varadhan}
        Suppose that \(\{P_n\}\) and \(\{Q_n\}\) are probability measures satisfying Assumption 1. Suppose further that that the sequence of (normalized) log-likelihood ratios \(\left\{\frac{\log \frac{q_n}{p_n}}{\log n}\right\}\) satisfies the large deviation principle under the null with good rate function \(I : \R \to [0, \infty]\). For ease of notation, set \(Z_n := \frac{\log \frac{q_n}{p_n}(X_n)}{\log n}\) where \(X_n \sim P_n\). Let \(\varphi : \R \to \R\) be a continuous function. Assume the following moment condition for some \(\gamma > 1\),
        \begin{equation*}\label{eqn:varadhan_moment}
            \limsup_{n \to \infty} \frac{1}{\log n} \log E\left[n^{\gamma \varphi\left(Z_n\right)}  \right] < \infty.
        \end{equation*}
        Then
        \begin{equation*}
            \lim_{n \to \infty} \frac{1}{\log n} \log E\left[n^{\varphi\left(Z_n\right)} \right] = \sup_{t \in \R} \left\{ \varphi(t) - I(t) \right\}.
        \end{equation*}
    \end{theorem}

    The following lemma is useful when integrating over a subset of \(\R\) in Varadhan's integral lemma (Exercise 4.3.11 in \cite{demboLargeDeviationsTechniques2010} and Exercise 2.1.24 in \cite{deuschelLargeDeviations2001}).
    \begin{lemma}\label{lemma:Varadhan_subset}
        Consider the setting of Theorem \ref{thm:Varadhan}. If the conditions of Theorem \ref{thm:Varadhan} hold, then for any open set \(G \subset \R\) and any closed set \(F \subset \R\) it follows that 
        \begin{align*}
            &\liminf_{n \to \infty} \frac{1}{\log n} \log E\left[ n^{\varphi(Z_n)} \mathbf{1}_{\{Z_n \in G\}} \right] \geq \sup_{t \in G} \{\varphi(t) - I(t)\}, \\
            &\limsup_{n \to \infty} \frac{1}{\log n} \log E\left[n^{\varphi(Z_n)} \mathbf{1}_{\{Z_n \in F\}}\right] \leq \sup_{t \in F} \{\varphi(t) - I(t)\}.
        \end{align*}
    \end{lemma}
    \begin{proof}
        As per the hint in Exercise 2.1.24 of \cite{deuschelLargeDeviations2001}, set
        \begin{equation*}
            \varphi_G(x) =
            \begin{cases}
                \varphi(x) & \text{if } x \in G \\
                -\infty & \text{otherwise}
            \end{cases}
        \end{equation*}
        and 
        \begin{equation*}
            \varphi_F(x) =
            \begin{cases}
                \varphi(x) & \text{if } x \in F \\
                -\infty & \text{otherwise}.
            \end{cases}
        \end{equation*}
        To deduce the result, Lemmas 2.1.7 and 2.1.8 in \cite{deuschelLargeDeviations2001} will be applied. First, we focus on \(G\). To apply Lemma 2.1.7, we must show that \(\varphi_G\) is lower semi-continuous. Let \(\{x_n\}\) be a sequence in \(\R\) such that \(x_n \to x\) for some \(x \in \R\). If \(x \not \in G\), then we trivially have \(\liminf_{n \to \infty} \varphi_G(x_n) \geq \varphi_G(x)\) since \(\varphi_G(x) = -\infty\). If \(x \in G\), then all but finitely many of the \(x_n\) lie in \(G\) since \(G\) is open and \(x_n \to x\). Thus, \(\varphi_G(x_n) = \varphi(x_n)\) for all but finitely many \(n\). Since \(\varphi\) is continuous, it immediately follows that \(\liminf_{n \to \infty} \varphi_G(x_n) = \liminf_{n \to \infty} \varphi(x_n) = \varphi(x) = \varphi_G(x)\). Hence, \(\varphi_G\) is lower semi-continuous. Thus, by Lemmma 2.1.7 in \cite{deuschelLargeDeviations2001}, it follows that
        \begin{align*}
            \liminf_{n \to \infty} \frac{1}{\log n} \log E\left[ n^{\varphi(Z_n)} \mathbf{1}_{\{Z_n \in G\}} \right] &= \liminf_{n \to \infty} \frac{1}{\log n} \log E\left[ n^{\varphi_G(Z_n)} \right] \\
            &\geq \sup_{t \in \R} \{\varphi_G(t) - I(t) : \varphi_G(t) \wedge I(t) < \infty\} \\
            &= \sup_{t \in G} \{\varphi(t) - I(t)\}
        \end{align*}
        as desired. 
        
        Turning our attention to \(F\), we need to show that \(\varphi_F\) is upper semi-continuous to apply Lemma 2.1.8. To do so, let \(\{x_n\}\) be a sequence converging to a point \(x \in \R\). If \(x \not \in F\), then all but finitely many of the \(x_n\) lie in \(F^c\) since \(F^c\) is open. Thus, we immediately have \(\limsup_{n \to \infty} \varphi_F(x_n) = -\infty = \varphi_F(x)\). Suppose \(x \in F\). If all but finitely many of the \(x_n\) lie in \(F^c\), then we immediately have \(\limsup_{n \to \infty} \varphi_F(x_n) = -\infty \leq \varphi_F(x)\). Instead, if \(x_n \in F\) for infinitely many \(n\), then for every subsequence \(x_{n_k} \in F\) we have that \(\lim_{k \to \infty} \varphi_F(x_{n_k}) = \lim_{k \to \infty} \varphi(x_{n_k}) = \varphi(x) = \varphi_F(x)\) by the continuity of \(\varphi\). Since this holds for all subsequences that lie in \(F\), it follows immediately that \(\limsup_{n \to \infty} \varphi_F(x_n) \leq \varphi_F(x)\). Thus, \(\varphi_F\) is upper semi-continuous. Applying Lemma 2.1.8 and applying a similar argument as in the display above yields the desired result.
    \end{proof}

    The following technical lemma (Lemma 4 from \cite{caiOptimalDetectionSparse2014}) is used in the analysis of the Hellinger asymptotics in the proof of Theorems \ref{thm:beta_upper} and \ref{thm:beta_lower}.

    \begin{lemma}[Lemma 4 - \cite{caiOptimalDetectionSparse2014}] \label{lemma:hellinger_lemma}
        \leavevmode
        \begin{enumerate}[label=\roman*)]
            \item For any \(b > 0\), the function \(s \mapsto (\sqrt{1+b(s-1)}-1)^2\) is strictly convex on \(\R_{+}\) and strictly decreasing and increasing on \([0, 1]\) and \([1, \infty)\), respectively.
            \item For any \(t \geq 0\), it follows that
            \begin{equation}
                (\sqrt{2} - 1)^2 \cdot (t \wedge t^2) \leq (\sqrt{1+t} - 1)^2 \leq t\wedge t^2.
            \end{equation}
        \end{enumerate}
    \end{lemma}

    \subsection{Proofs of Theorem \ref{thm:beta_upper}, Theorem \ref{thm:beta_lower}, Corollary \ref{corollary:tight_limit}, Corollary \ref{corollary:nice_tight_limit}}
    The proofs of Theorems \ref{thm:beta_upper} and \ref{thm:beta_lower} follow the same roadmap as the proof of Theorem 3 in \cite{caiOptimalDetectionSparse2014} in that Hellinger distance asymptotics are examined. The main difference lies in the fact that we do not impose a uniform convergence condition as in \cite{caiOptimalDetectionSparse2014}. Rather, we assume that the normalized log likelihood ratios satisfy a large deviation principle. Consequently, we apply Varadhan's integral lemma (Theorem \ref{thm:Varadhan}) instead of applying a version of Laplace's method as in \cite{caiOptimalDetectionSparse2014}. 

    \begin{proof}[Proof of Theorem \ref{thm:beta_upper}]
        We follow the approach of \cite{caiOptimalDetectionSparse2014}. To prove Theorem \ref{thm:beta_upper}, it suffices to show that if \(\beta > \overline{\beta}^\#\), then \(H_n^2(\beta) = o(n^{-1})\) by Lemma \ref{lemma:Hellinger}. Let \(\beta > \overline{\beta}^\#\). Let \(h_n = (1-\varepsilon)p_n + \varepsilon q_n\) be the density of the alternative hypothesis in testing problem (\ref{problem:sparse_mixture_detection_1})-(\ref{problem:sparse_mixture_detection_2}) with calibration (\ref{eqn:beta_sparsity}) and observe that the likelihood ratio can be expressed as 
        \begin{align*}
            \frac{h_n}{p_n} &= 1 + n^{-\beta} \left(\frac{q_n}{p_n} - 1\right) \\
            &= 1 + n^{-\beta} \left(\exp\left( \log \frac{q_n}{p_n} \right) - 1\right).
        \end{align*}
        Let \(X_n \sim P_n\) and \(\ell_n := \log \frac{q_n}{p_n}(X_n)\). Consider 
        \begin{align*}
            H_n^2(\beta) &= E\left[\left(\sqrt{\frac{h_n}{p_n}(X_n)} - 1\right)^2 \right] \\
            &= E\left[\left(\sqrt{1 + n^{-\beta} \left(\exp\left(\ell_n\right) - 1\right)} - 1\right)^2 \right] \\
            &= E\left[\left(\sqrt{1 + n^{-\beta} \left( \exp\left(\ell_n \right) - 1\right)} - 1\right)^2 \cdot \mathbf{1}_{\left\{\ell_n/\log n \geq 0\right\}} \right] \\
            &+ E\left[\left(\sqrt{1 + n^{-\beta} \left(\exp\left(\ell_n\right) - 1\right)} - 1\right)^2 \cdot \mathbf{1}_{\{\ell_n/\log n < 0\}} \right].
        \end{align*}
        Examining the second term, consider that \(\exp(\ell_n) < 1\) when \(\ell_n/\log n < 0\). Thus, by Lemma \ref{lemma:hellinger_lemma}, it follows that 
        \begin{align*}
            E\left[\left(\sqrt{1 + n^{-\beta} \left(\exp\left(\ell_n\right) - 1\right)} - 1\right)^2 \cdot \mathbf{1}_{\{\ell_n/\log n < 0\}} \right] &\leq (\sqrt{1 + n^{-\beta}(0-1)} - 1)^2 \\
            &= (\sqrt{1-n^{-\beta}}-1)^2 \\
            &\leq n^{-2\beta} \\
            &= o(n^{-1})
        \end{align*}
        since \(\beta > \overline{\beta}^\# \geq \frac{1}{2}\). 
        
        Turning our attention to the first term in the expansion of \(H_n^2(\beta)\), consider that \(\exp(\ell_n) \geq 1\) when \(\ell_n \geq 0\). So, by Lemma \ref{lemma:hellinger_lemma},
        \begin{align*}
            E\left[\left(\sqrt{1 + n^{-\beta} \left( \exp\left(\ell_n \right) - 1\right)} - 1\right)^2 \cdot \mathbf{1}_{\{\ell_n/\log n \geq 0\}} \right] &\leq E\left[\left(\sqrt{1 + n^{-\beta}\exp(\ell_n)} - 1\right)^2 \cdot \mathbf{1}_{\{\ell_n/\log n \geq 0\}} \right]  \\ 
            &= E\left[\left(\sqrt{1 + n^{-\beta + \frac{\ell_n}{\log n}}} - 1\right)^2 \cdot \mathbf{1}_{\{\ell_n/\log n \geq 0\}} \right] \\
            &\leq E\left[n^{2\left(\frac{\ell_n}{\log n} - \beta\right) \wedge \left(\frac{\ell_n}{\log n} - \beta \right)} \cdot \mathbf{1}_{\{\ell_n/\log n \geq 0\}}\right]
        \end{align*}
        where the final inequality follows from Lemma \ref{lemma:hellinger_lemma}. We will now apply Varadhan's lemma (actually, we apply Lemma \ref{lemma:Varadhan_subset} but the content of the mathematics we utilize is attributed to Varadhan). First, consider that by assumption there exists \(\gamma > 1\) such that 
        \begin{equation*}
            \limsup_{n \to \infty} \frac{1}{\log n} \cdot \log E\left[ \left( \frac{q_n}{p_n}(X_n) \right)^\gamma \right] < \infty
        \end{equation*}
        where \(X_n \sim P_n\). Let \(\varphi_\beta(x) = 2(x-\beta) \wedge (x-\beta)\) and note that \(\varphi_\beta\) is continuous. Then, observe
        \begin{align*}
            \limsup_{n \to \infty} \frac{1}{\log n} \log E\left[ n^{\gamma \varphi_\beta(\ell_n/\log(n))} \right] &= \limsup_{n \to \infty} \frac{1}{\log n} \log E \left[ \exp\left(\gamma \cdot \log (n) \cdot \varphi_\beta(\ell_n/\log(n)) \right) \right] \\
            &\leq \limsup_{n \to \infty} \frac{1}{\log n} \log E\left[ \exp\left(\gamma \ell_n - \beta \cdot \log(n) \right) \right] \\
            &\leq \limsup_{n \to \infty} \frac{1}{\log n} \log E\left[ \left(\frac{q_n}{p_n}(X_n) \right)^\gamma \cdot n^{-\beta} \right] \\
            &\leq \limsup_{n \to \infty} \frac{1}{\log n} \log E\left[ \left(\frac{q_n}{p_n}(X_n) \right)^\gamma \right] \\
            &< \infty
        \end{align*}
        where the penultimate inequality follows from \(\beta > \beta^\# \geq \frac{1}{2} > 0\). Hence, the moment condition of Varadhan's integral lemma is satisfied, and so an application of Lemma \ref{lemma:Varadhan_subset} yields 
        \begin{equation*}
            \limsup_{n \to \infty} \frac{1}{\log n} \log E\left[n^{2\left(\frac{\ell_n}{\log n} - \beta\right) \wedge \left(\frac{\ell_n}{\log n} - \beta \right)} \cdot \mathbf{1}_{\{\ell_n/\log n \geq 0\}}\right] \leq \sup_{t \geq 0} \{\varphi_\beta(t) - I(t)\}.
        \end{equation*}
        Thus, it follows that if 
        \begin{equation} \label{eqn:thm1_sufficiency}
            \sup_{t \geq 0} \{\varphi_\beta(t) - I(t)\} < -1,
        \end{equation}
        then \(E\left[n^{2\left(\frac{\ell_n}{\log n} - \beta\right) \wedge \left(\frac{\ell_n}{\log n} - \beta \right)} \cdot \mathbf{1}_{\{\ell_n/\log n \geq 0\}}\right] = o(n^{-1})\). In order for condition (\ref{eqn:thm1_sufficiency}) to hold, we need either \(2(t-\beta) - I(t) < - 1\) or \(t-\beta - I(t) < -1\) for all \(t \geq 0\). Equivalently, we need either \(\beta > \frac{1}{2} - \frac{I(t)}{2} + t\) or \(\beta > t - I(t) + 1\) for all \(t \geq 0\). Equivalently, we require 
        \begin{align*}
            \beta &> \sup_{t \geq 0} \left\{ \left(\frac{1}{2} - \frac{I(t)}{2} + t\right) \wedge \left(t - I(t) + 1\right) \right\} \\
            &= \sup_{t\geq 0} \left\{t + \frac{1}{2} + \left( - \frac{I(t)}{2}\right)\wedge\left(\frac{1}{2} - I(t) \right) \right\} \\
            &= \sup_{t\geq 0} \left\{t + \frac{1}{2} - I(t) + \frac{1 \wedge I(t)}{2} \right\}.
        \end{align*}
        which is equivalent to requiring 
        \begin{equation}\label{eqn:thm1_beta_condition_term1}
           \beta > \frac{1}{2} + \sup_{t\geq 0} \left\{t - I(t) + \frac{1 \wedge I(t)}{2} \right\}.
        \end{equation}
        Since \(\beta > \underline{\beta}^\# = \frac{1}{2} + 0\vee \sup_{t \geq 0}\left\{t - I(t) + \frac{1 \wedge I(t)}{2} \right\}\), it follows that \(\beta\) satisfies condition (\ref{eqn:thm1_beta_condition_term1}). Hence, \(E\left[n^{2\left(\frac{\ell_n}{\log n} - \beta\right) \wedge \left(\frac{\ell_n}{\log n} - \beta \right)} \cdot \mathbf{1}_{\{\ell_n/\log n \geq 0\}}\right] = o(n^{-1})\). Thus, it's been established that \(H_n^2(\beta) = o(n^{-1})\). Therefore, \(\overline{\beta}^\# \geq \overline{\beta}^*\), as desired. 
    \end{proof}

    \begin{proof}[Proof of Theorem \ref{thm:beta_lower}]
        To prove Theorem \ref{thm:beta_lower}, it suffices to show that if \(\beta < \underline{\beta}^\#\), then \(H_n^2(\beta) = \omega(n^{-1})\). Let \(\delta > 0\). Let \(X_n \sim P_n\) and \(\ell_n = \log \frac{q_n}{p_n}(X_n)\). From the proof of Theorem \ref{thm:beta_upper}, 
        \begin{align*}
            H_n^2(\beta) &= E\left[\left(\sqrt{\frac{h_n}{p_n}(X_n)} - 1\right)^2 \right] \\
            &= E\left[\left(\sqrt{1 + n^{-\beta} \left(\exp\left(\ell_n\right) - 1\right)} - 1\right)^2 \right] \\
            &\geq E\left[\left(\sqrt{1 + n^{-\beta} \left( \exp\left(\ell_n \right) - 1\right)} - 1\right)^2 \cdot \mathbf{1}_{\{\ell_n/\log n > \delta\}} \right].
        \end{align*}
        From the lower bound in Lemma \ref{lemma:hellinger_lemma}, when \(\ell_n/\log n > \delta\) we have 
        \begin{align*}
            & \left(\sqrt{1 + n^{-\beta}(\exp(\ell_n) - 1)} - 1 \right)^2 \\
            &\geq (\sqrt{2} - 1)^2 \left[ \left(n^{-\beta + \ell_n/\log(n)} - n^{-\beta}\right) \wedge (n^{-\beta+\ell_n/\log(n)} - n^{-\beta})^2 \right] \\
            &= (\sqrt{2} - 1)^2 \left[\left[ (n^{-\beta + \ell_n/\log(n)})\left(1 - n^{-\ell_n/\log(n)}\right)\right] \wedge \left[ (n^{-2\beta + 2\ell_n/\log(n)})(1-n^{-\ell_n/\log(n)})^2\right] \right] \\
            &\geq (\sqrt{2} - 1)^2 (1 - n^{-\ell_n/\log(n)})^2 \cdot n^{\left(\ell_n/\log(n) - \beta \right) \wedge 2(\ell_n/\log(n) - \beta)} \\
            &\geq (\sqrt{2}-1)^2(1-n^{-\delta})^2 \cdot n^{\left(\ell_n/\log(n) - \beta \right) \wedge 2(\ell_n/\log(n) - \beta)}.
        \end{align*}
        Thus, it follows that 
        \begin{align*}
            & E\left[\left(\sqrt{1 + n^{-\beta} \left( \exp\left(\ell_n \right) - 1\right)} - 1\right)^2 \cdot \mathbf{1}_{\{\ell_n/\log n > \delta\}} \right] \\
            &\geq (\sqrt{2}-1)^2(1-n^{-\delta})^2 \cdot E\left[n^{\left(\ell_n/\log(n) - \beta \right) \wedge 2(\ell_n/\log(n) - \beta)} \cdot \mathbf{1}_{\{\ell_n/\log n > \delta\}}\right].
        \end{align*}
        From the proof of Theorem \ref{thm:beta_upper}, it follows that the moment condition of Lemma \ref{lemma:Varadhan_subset} is satisfied, and so Lemma \ref{lemma:Varadhan_subset} can be applied to obtain 
        \begin{equation*}
            \liminf_{n \to \infty} \frac{1}{\log n} \log E\left[n^{2\left(\frac{\ell_n}{\log n} - \beta\right) \wedge \left(\frac{\ell_n}{\log n} - \beta \right)} \cdot \mathbf{1}_{\{\ell_n/\log n > \delta\}}\right] \geq \sup_{t > \delta} \{\varphi_\beta(t) - I(t)\}
        \end{equation*}
        where \(\varphi_\beta(x) = (x-\beta)\wedge 2(x-\beta)\). Thus, it has been established that 
        \begin{align*}
            &\liminf_{n \to \infty} \frac{\log H_n^2(\beta)}{\log n} \\
            &\geq \liminf_{n \to \infty} \frac{\log \left[(\sqrt{2}-1)^2(1-n^{-\delta})^2 \cdot E\left[n^{\left(\ell_n/\log(n) - \beta \right) \wedge 2(\ell_n/\log(n) - \beta)} \cdot \mathbf{1}_{\{\ell_n/\log n > \delta\}}\right] \right]}{\log n} \\
            &= \liminf_{n \to \infty} \frac{\log[(\sqrt{2} - 1)^2]}{\log n} + \frac{\log[(1-n^{-\delta})^2]}{\log n} + \frac{\log E\left[n^{2\left(\frac{\ell_n}{\log n} - \beta\right) \wedge \left(\frac{\ell_n}{\log n} - \beta \right)} \cdot \mathbf{1}_{\{\ell_n/\log n > \delta\}}\right] }{\log n} \\
            &\geq \sup_{t > \delta} \{\varphi_\beta(t) - I(t)\}.
        \end{align*}
        Since this holds for all \(\delta > 0\), it follows that 
        \begin{equation*}
            \liminf_{n \to \infty} \frac{\log H_n^2(\beta)}{\log n} \geq \sup_{t > 0} \{\varphi_\beta(t) - I(t)\}.
        \end{equation*}
        Thus, a sufficient condition that \(H_n^2(\beta) = \omega(n^{-1})\) is that 
        \begin{equation*}
            \sup_{t > 0} \{\varphi_\beta(t) - I(t)\} > -1
        \end{equation*}
        or equivalently
        \begin{equation*}
            \sup_{t > 0} \{(t-\beta)\wedge2(t-\beta) - I(t)\} > -1.
        \end{equation*}
        The same reasoning at the end of the proof of Theorem \ref{thm:beta_upper} yields the equivalent condition 
        \begin{equation*}
            \beta < \frac{1}{2} + \sup_{t > 0} \left\{t - I(t) + \frac{1 \wedge I(t)}{2}\right\}
        \end{equation*}
        which is exactly equivalent to \(\beta < \underline{\beta}^\#\). Since we indeed had \(\beta < \underline{\beta}^\#\) by assumption, it follows that \(H_n^2(\beta) = \omega(n^{-1})\). Therefore, \(\underline{\beta}^\# \leq \underline{\beta}^*\), as desired. 
    \end{proof}

    \begin{proof}[Proof of Corollary \ref{corollary:tight_limit}]
        Consider that \(\underline{\beta}^\# = \overline{\beta}^\# = \beta^*\). Since Theorems \ref{thm:beta_upper} and \ref{thm:beta_lower} imply    
        \begin{equation*}
            \underline{\beta}^\# \leq \underline{\beta}^* \leq \overline{\beta}^* \leq \overline{\beta}^\#,
        \end{equation*}
        it immediately follows that \(\underline{\beta}^* = \overline{\beta}^* = \beta^*\).
    \end{proof}

    \begin{proof}[Proof of Corollary \ref{corollary:nice_tight_limit}]
        With the condition that \(I\) be right-continuous at \(0\), it follows that 
        \begin{equation*}
            \sup_{t > 0} \left\{t - I(t) + \frac{1 \wedge I(t)}{2}\right\} = \sup_{t \geq 0}\left\{t - I(t) + \frac{1 \wedge I(t)}{2}\right\}.
        \end{equation*}
        Since there exists \(t^* \geq 0\) such that \(t^* - I(t^*) + \frac{1 \wedge I(t^*)}{2} \geq 0\), it follows that 
        \begin{equation*}
            \sup_{t > 0} \left\{t - I(t) + \frac{1 \wedge I(t)}{2}\right\} = \sup_{t \geq 0}\left\{t - I(t) + \frac{1 \wedge I(t)}{2}\right\} \geq 0
        \end{equation*}
        and so Corollary \ref{corollary:tight_limit} holds. Thus, \(\underline{\beta}^* = \overline{\beta}^* = \beta^*\).
    \end{proof}

    \subsection{Proof of Proposition \ref{prop:weaker_than_cai_wu}}

    A few key definitions and results in the theory of large deviations will be used in the proof of Proposition \ref{prop:weaker_than_cai_wu}. These results are stated here for completeness. We follow the presentation found in Chapter 4 of \cite{demboLargeDeviationsTechniques2010}. 

    \begin{theorem}[Contraction principle, Theorem 4.2.1 - \cite{demboLargeDeviationsTechniques2010}]\label{thm:contraction_principle}
        Let \(\mathcal{X}\) and \(\mathcal{Y}\) be Hausdorff topological spaces and \(f : \mathcal{X} \to \mathcal{Y}\) a continuous function. Consider a good rate function \(I : \mathcal{X} \to [0, \infty]\). 
        \begin{enumerate}[label=(\alph*)]
            \item For each \(y \in \mathcal{Y}\), define
            \begin{equation*}
                I'(y) := \inf\{I(x) : x \in \mathcal{X}, \, y=f(x)\}.
            \end{equation*}
            Then \(I'\) is a good rate function on \(\mathcal{Y}\). Here, we adopt the convention \(\inf \emptyset = \infty\).
            
            \item If \(I\) controls the large deviation principle with a family of probability measures \(\{\mu_n\}\) on \(\mathcal{X}\), then \(I'\) controls the large deviation principle associated with the family of probability measures \(\{\mu_n \circ f^{-1}\}\) on \(\mathcal{Y}\). 
        \end{enumerate}
    \end{theorem}

    \begin{definition}[Exponential equivalence, Definition 4.2.10 - \cite{demboLargeDeviationsTechniques2010}]\label{def:exp_equiv}
        Let \((\mathcal{Y}, d)\) be a metric space. The probability measures \(\{\mu_n\}\) and \(\{\widetilde{\mu}_n\}\) are called \textit{exponentially equivalent} with respect to speed \(\{a_n\}\) if there exist probability spaces \(\{\Omega, \mathcal{B}_n, P_n\}\) and two families of \(\mathcal{Y}\)-valued random variables \(\{Z_n\}\) and \(\{\widetilde{Z}_n\}\) with joint laws \(\{P_n\}\) and marginals \(\{\mu_n\}\) and \(\{\widetilde{\mu}_n\}\) respectively such that the following condition is satisfied. For each \(\delta > 0\), the set \(\{\omega \in \Omega : d(\widetilde{Z}_n, Z_n) > \delta\}\) is \(\mathcal{B}_n\) measurable, and 
        \begin{equation*}
            \limsup_{n \to \infty} a_n \log P_n \left( d(\widetilde{Z}_n, Z_n) > \delta \right) = -\infty.
        \end{equation*}
        Here, \(\{a_n\}\) is a sequence of reals with \(a_n \to 0\).
    \end{definition}

    \begin{theorem}[Indistinguishability of the large deviation principle, Theorem 4.2.13 - \cite{demboLargeDeviationsTechniques2010}]\label{thm:same_ldp}
        If a large deviation principle with speed \(\{a_n\}\) and good rate function \(I\) holds for the probability measures \(\{\mu_n\}\), which are exponentially equivalent to \(\{\widetilde{\mu}_n\}\), then the same large deviation principle holds for \(\{\widetilde{\mu}_n\}\).
    \end{theorem}

    The following small lemma will be used in the proof of Proposition \ref{prop:weaker_than_cai_wu}. 
    
    \begin{lemma}\label{lemma:ldp_Yn}
        Let \(Y_n \sim \frac{1}{\log n} \Exp(1)\) for \(n \geq 2\). Then \(\{Y_n\}\) satisfies the large deviation principle with speed \(\left\{\frac{1}{\log n}\right\}\) and good rate function
        \begin{equation*}
            J(t) = 
            \begin{cases}
                t &\text{if } t \geq 0 \\
                \infty &\text{if } t < 0.
            \end{cases}
        \end{equation*}
    \end{lemma}
    \begin{proof}
        The density of \(Y_n\) is \(f_n(t) = (\log n) \exp(-t \log n) \cdot \mathbf{1}_{\{t \geq 0\}}\). For a Borel set \(\Gamma \subset \R\), it follows that
        \begin{align*}
            \frac{\log P(Y_n \in \Gamma)}{\log n} &= \frac{1}{\log n} \cdot \log \int_{\Gamma \cap [0, \infty)} (\log n) \exp(-t \log n) \, dt \\
            &= \frac{\log(\log n)}{\log n} + \frac{1}{\log n} \cdot \log \int_{\Gamma \cap [0, \infty)} \exp(-t \log n) \, dt
        \end{align*}
        Noting that \(\int_{0}^{\infty} e^{-t} \, dt < \infty\), Lemma 3 of \cite{caiOptimalDetectionSparse2014} can be applied to obtain 
        \begin{align*}
            \lim_{n \to \infty} \frac{1}{\log n} \cdot \log \int_{\Gamma \cap [0, \infty)} \exp(-t \log n) \, dt &= \ess \sup_{t \in \Gamma \cap [0, \infty)} -t \\
            &= \sup_{t \in \Gamma \cap [0, \infty)} -t \\
            &= -\inf_{t \in \Gamma \cap [0, \infty)} t.
        \end{align*}
        It immediately follows that 
        \begin{align*}
            - \inf_{t \in \Gamma^\circ} J(t) &\leq \liminf_{n \to \infty} \frac{\log P(Y_n \in \Gamma)}{\log n} \\
            &\leq \limsup_{n \to \infty} \frac{\log P(Y_n \in \Gamma)}{\log n} \leq -\inf_{t \in \overline{\Gamma}} J(t).
        \end{align*}
        It is clear that \(J\) is a good rate function. Hence, \(\{Y_n\}\) satisfies the large deviation principle with speed \(\left\{\frac{1}{\log n}\right\}\) and good rate function \(J\). 
    \end{proof}

    \begin{proof}[Proof of Proposition \ref{prop:weaker_than_cai_wu}]
        Let \(X_n \sim P_n\). Define \(\nu_0^{(n)}\) and \(\nu_1^{(n)}\) to be the probability measures such that for any Borel set \(B \subset \R\), 
        \begin{align*}
            \nu_0^{(n)}(B) &= \frac{P(X_n \in (-\infty, z^{(n)}(1/2)) \cap B)}{P(X_n \in (-\infty, z^{(n)}(1/2))},\\
            \nu_1^{(n)}(B) &= \frac{P(X_n \in [z^{(n)}(1/2), \infty) \cap B)}{P(X_n \in [z^{(n)}(1/2), \infty))},
        \end{align*}
        where \(z^{(n)}\) is the quantile function of \(X_n\). Let \(z_0^{(n)}, z_1^{(n)}\) be the quantile functions of \(\nu_0^{(n)}, \nu_1^{(n)}\) respectively. Let \(Y_n, \widetilde{Y}_n \sim \frac{1}{\log n} \cdot \Exp(1)\) be independent. Observe that \(z_0^{(n)}(n^{-Y_n}) \sim \nu_0^{(n)}\) and \(z_1^{(n)}(1-n^{-\widetilde{Y}_n}) \sim \nu_1^{(n)}\). Let \(Z_n \sim \Bernoulli\left(\frac{1}{2}\right)\) be independent of \(Y_n, \widetilde{Y}_n\) and observe that we can write 
        \begin{equation*}
            X_n \overset{d}{=} z_0^{(n)}(n^{-Y_n}) \cdot \mathbf{1}\{Z_n = 0\} + z_1^{(n)}(1-n^{-\widetilde{Y}_n}) \cdot \mathbf{1}\{Z_n = 1\}.
        \end{equation*}
        Consider further that 
        \begin{equation*}
            \log\frac{q_n}{p_n}(X_n) \overset{d}{=} \left(\log \frac{q_n}{p_n}(z_0^{(n)}(n^{-Y_n}))\right) \cdot \mathbf{1}\{Z_n = 0\} + \left(\log \frac{q_n}{p_n}(z_1^{(n)}(1-n^{-\widetilde{Y}_n})) \right) \cdot \mathbf{1}\{Z_n = 1\}. 
        \end{equation*}
        For a Borel set \(\Gamma \subset \R\), observe that 
        \begin{align*}
            & \log P\left( \frac{\log \frac{q_n}{p_n}(X_n)}{\log n} \in \Gamma \right) \\
            &= \log\left[ \frac{1}{2} P\left( \frac{\log \frac{q_n}{p_n}(z_0^{(n)}(n^{-Y_n}))}{\log n} \in \Gamma \right) + \frac{1}{2} P\left( \frac{\log \frac{q_n}{p_n}(z_1^{(n)}(1-n^{-\widetilde{Y}_n}))}{\log n} \in \Gamma \right) \right] \\
            &= -\log(2) + \log\left[P\left( \frac{\log \frac{q_n}{p_n}(z_0^{(n)}(n^{-Y_n}))}{\log n} \in \Gamma \right) + P\left( \frac{\log \frac{q_n}{p_n}(z_1^{(n)}(1-n^{-\widetilde{Y}_n}))}{\log n} \in \Gamma \right) \right].
        \end{align*}
        For ease of notation, set \(r_0 = P\left( \frac{\log \frac{q_n}{p_n}(z_0^{(n)}(n^{-Y_n}))}{\log n} \in \Gamma \right)\) and \(r_1 =  P\left( \frac{\log \frac{q_n}{p_n}(z_1^{(n)}(1-n^{-\widetilde{Y}_n}))}{\log n} \in \Gamma \right)\). Then, note that 
        \begin{align*}
            \frac{1}{\log n} \cdot \log P\left( \frac{\log \frac{q_n}{p_n}(X_n)}{\log n} \in \Gamma \right) &= -\frac{\log(2)}{\log n} + \frac{1}{\log n}\log\left(r_0 + r_1\right) \\
            &= -\frac{\log(2)}{\log n} + \frac{1}{\log n} \cdot \log\left( (r_0 \vee r_1) \cdot \left( 1 + \frac{r_0 \wedge r_1}{r_0 \vee r_1} \right)\right) \\
            &= -\frac{\log(2)}{\log n} + \left(\frac{\log r_0}{\log n}\right) \vee \left(\frac{\log r_1}{\log n}\right) + \frac{1}{\log n} \cdot \log\left(1 + \frac{r_0 \wedge r_1}{r_0 \vee r_1}\right). 
        \end{align*}
        Consider that \(0 \leq \log\left(1 + \frac{r_0 \wedge r_1}{r_0 \vee r_1}\right) \leq \log(2)\), and so we have
        \begin{equation} \label{eqn:cai_wu_ldp}
            \frac{1}{\log n} \cdot \log P\left( \frac{\log \frac{q_n}{p_n}(X_n)}{\log n} \in \Gamma \right) = \left(\frac{\log r_0}{\log n}\right) \vee \left(\frac{\log r_1}{\log n}\right) + o(1)
        \end{equation}
        as \(n \to \infty\). Thus, to establish that \(\frac{\log \frac{q_n}{p_n}}{\log n}\) satisfies the large deviation principle under the null, it suffices to study whether \(\frac{\log \frac{q_n}{p_n}(z_0^{(n)}(n^{-Y_n}))}{\log n}\) and \(\frac{\log \frac{q_n}{p_n}(z_1^{(n)}(1-n^{-\widetilde{Y}_n}))}{\log n}\) satisfy large deviation principles. 

        Consider that 
        \begin{align*}
            z_0^{(n)}(n^{-s + \log_n 2}) &= z^{(n)}(n^{-s}),\\
            z_1^{(n)}(1-n^{-s + \log_n 2}) &= z^{(n)}(1-n^{-s})
        \end{align*}
        for all \(s \geq \frac{1}{\log_2 n}\). Therefore, the uniform convergence conditions (\ref{eqn:unif_convergence_0}) and (\ref{eqn:unif_convergence_1}) can be equivalently written as 
        \begin{align}
            \lim_{n \to \infty} \sup_{r \geq 0} \left|\frac{\log \frac{q_n}{p_n}(z_0^{(n)}(n^{-r}))}{\log n} - \alpha_0(r+\log_n 2) \right| = 0, \label{eqn:unif_conv_modified_0}\\
            \lim_{n \to \infty} \sup_{r \geq 0} \left| \frac{\log \frac{q_n}{p_n}(z_1^{(n)}(1-n^{-r}))}{\log n} - \alpha_1(r+\log_n 2) \right| = 0. \label{eqn:unif_conv_modified_1}
        \end{align}
        Let \(W_n := Y_n + \log_n 2\) and \(\widetilde{W}_n := \widetilde{Y}_n + \log_n 2\). Note that for all \(\delta > 0\),
        \begin{align*}
            \limsup_{n \to \infty} \frac{1}{\log n} \log P\left( |W_n - Y_n| > \delta \right) &= \limsup_{n \to \infty} \frac{1}{\log n} \log P\left(\frac{\log 2}{\log n} > \delta \right) \\
            &= -\infty   
        \end{align*}
        and so \(W_n\) and \(Y_n\) are exponentially equivalent (recall Definition \ref{def:exp_equiv}) with respect to speed \(\left\{\frac{1}{\log n}\right\}\). Likewise, \(\widetilde{W}_n\) and \(\widetilde{Y}_n\) are exponentially equivalent with respect to the same speed. By Lemma \ref{lemma:ldp_Yn}, it follows that \(Y_n\) and \(\widetilde{Y}_n\) both satisfy a large deviation principle with speed \(\left\{\frac{1}{\log n}\right\}\) and with the good rate function \(J\) specified in Lemma \ref{lemma:ldp_Yn}. Then, it follows by Theorem \ref{thm:same_ldp} that \(W_n\) and \(\widetilde{W}_n\) both satisfy a large deviation principle with speed \(\left\{\frac{1}{\log n}\right\}\) and with the good rate function \(J\).

        By conditions (\ref{eqn:unif_conv_modified_0}) and (\ref{eqn:unif_conv_modified_1}), it follows that for any \(\delta > 0\),
        \begin{align*}
            \left|\frac{\log \frac{q_n}{p_n}(z_0^{(n)}(n^{-Y_n}))}{\log n} - \alpha_0(W_n) \right| < \delta, \\
            \left| \frac{\log \frac{q_n}{p_n}(z_1^{(n)}(1-n^{-\widetilde{Y}_n}))}{\log n} - \alpha_1(\widetilde{W}_n)\right| < \delta 
        \end{align*}
        almost surely for all \(n\) sufficiently large. Therefore, for all \(\delta > 0\), it follows that 
        \begin{align*}
            \limsup_{n \to \infty} \frac{1}{\log n} \log P\left(\left|\frac{\log \frac{q_n}{p_n}(z_0^{(n)}(n^{-Y_n}))}{\log n} - \alpha_0(W_n) \right| > \delta \right) = -\infty,
        \end{align*}
        and so \(\frac{\log \frac{q_n}{p_n}(z_0^{(n)}(n^{-Y_n}))}{\log n}\) and \(\alpha_0(W_n)\) are exponentially equivalent with respect to speed \(\left\{\frac{1}{\log n}\right\}\). The same argument yields that \(\frac{\log \frac{q_n}{p_n}(z_1^{(n)}(1-n^{-\widetilde{Y}_n}))}{\log n} \) and \(\alpha_1(\widetilde{W}_n)\) are exponentially equivalent with respect to the same speed. 

        Since \(\alpha_0\) and \(\alpha_1\) are continuous, the contraction principle (Theorem \ref{thm:contraction_principle}) implies that \(\alpha_0(W_n)\) and \(\alpha_1(\widetilde{W}_n)\) satisfy large deviation principles with respect to speed \(\left\{\frac{1}{\log n}\right\}\) and with good rate functions
        \begin{align*}
            I_0(t) &= \inf\{J(w) : t = \alpha_0(w)\}, \\
            I_1(t) &= \inf\{J(w) : t = \alpha_1(w)\}
        \end{align*}
        respectively. Since \(J(w) = w\) for \(w \geq 0\) and \(J(w) = \infty\) for \(w < 0\), it can equivalently be written as 
        \begin{align*}
            I_0(t) &= \inf\{w \geq 0 : t = \alpha_0(w)\} \\
            I_1(t) &= \inf\{w \geq 0 : t = \alpha_1(w)\}.
        \end{align*}
        Recall that we use the convention that \(\inf \emptyset = \infty\). By the exponential equivalence, it follows that \(\frac{\log \frac{q_n}{p_n}(z_0^{(n)}(n^{-Y_n}))}{\log n}\) and \(\frac{\log \frac{q_n}{p_n}(z_1^{(n)}(1-n^{-\widetilde{Y}_n}))}{\log n} \) both satisfy the large deviation principle under the null. The respective good rate functions are \(I_0\) and \(I_1\). 

        Turning our attention back to (\ref{eqn:cai_wu_ldp}), we immediately see that for any Borel set \(\Gamma \subset \R\),
        \begin{align*}
            \liminf_{n \to \infty} \frac{1}{\log n} \cdot \log P\left(\frac{\log \frac{q_n}{p_n}(X_n)}{\log n} \in \Gamma \right) &\geq \left(-\inf_{t \in \Gamma^\circ} I_0(t)\right) \vee \left(-\inf_{t \in \Gamma^\circ} I_1(t)\right) \\
            &= \sup_{t \in \Gamma^\circ} \left\{ (-I_0(t)) \vee (-I_1(t)) \right\} \\
            &= \sup_{t \in \Gamma^\circ} - \left( I_0(t) \wedge I_1 (t)\right) \\
            &= - \inf_{t \in \Gamma^\circ} (I_0(t) \wedge I_1(t)).
        \end{align*}
        Likewise, 
        \begin{align*}
            \limsup_{n \to \infty} \frac{1}{\log n} \cdot \log P\left(\frac{\log \frac{q_n}{p_n}(X_n)}{\log n} \in \Gamma \right) &\leq \left(-\inf_{t \in \overline{\Gamma}} I_0(t)\right) \vee \left(-\inf_{t \in \overline{\Gamma}} I_1(t)\right) \\
            &= \sup_{t \in \overline{\Gamma}} (-I_0(t)) \vee (-I_1(t)) \\
            &= \sup_{t \in \overline{\Gamma}} - \left( I_0(t) \wedge I_1(t)\right) \\
            &= - \inf_{t \in \overline{\Gamma}} (I_0(t) \wedge I_1(t)).
        \end{align*}
        Therefore, \(\left\{ \frac{\log \frac{q_n}{p_n}}{\log n} \right\}\) satisfies the large deviation principle under the null with good rate function \(I_0 \wedge I_1\), as desired.
    \end{proof}

    \subsection{Proof of Proposition \ref{prop:HC_bound_formula}}

    \begin{proof}[Proof of Proposition \ref{prop:HC_bound_formula}]
        The inequality \(\underline{\beta}^{\HC} \leq \underline{\beta}^*\) follows immediately from Proposition \ref{prop:HC_upper_bound}. We now turn our attention to the lower bound. For each \(n\), define the event
        \begin{equation*}
            A_n(c) := \left\{ x \in \mathcal{X} : \frac{q_n}{p_n}(x) > n^c \right\}
        \end{equation*}
        where \(c \geq 0\). Observe that \(A_n(c) \in \A_n^*\) for each \(n\). By definition, it follows 
        \begin{equation}\label{eqn:beta_hc_lln_bound}
            \underline{\beta}^{\HC} \geq \frac{1}{2} + \liminf_{n \to \infty} \left\{ \frac{\log Q_n(A_n(c))}{\log n} + \frac{1}{2} \min\left(1, - \frac{\log P_n(A_n(c))}{\log n} \right) \right\}
        \end{equation}
        for all \(c \geq 0\). Since \(\left\{\frac{\log \frac{q_n}{p_n}}{\log n}\right\}\) satisfies the large deviation principle under the null with good rate function \(I\), it follows that 
        \begin{align*}
            \limsup_{n \to \infty} \frac{\log P_n(A_n(c))}{\log n} &= \limsup_{n \to \infty} \frac{\log P\left( \frac{\log \frac{q_n}{p_n}(X_n)}{\log n} > c\right)}{\log n} \leq -\inf_{t \geq c} I(t)
        \end{align*}
        where \(X_n \sim P_n\). Hence, 
        \begin{equation*}\label{eqn:log_P_bound}
            \liminf_{n \to \infty} -\frac{\log P_n(A_n(c))}{\log n} \geq \inf_{t \geq c} I(t).
        \end{equation*}
        Turning our attention to the term \(\frac{\log Q_n(A_n(c))}{\log n}\), we will apply Varahdan's integral lemma. First, note that condition (\ref{eqn:HC_bound_moment_condition}) implies 
        \begin{equation*}
            \limsup_{n \to \infty} \frac{1}{\log n} \cdot \log E\left[ n^{\gamma \frac{\log \frac{q_n}{p_n}(X_n)}{\log n}} \right] < \infty.
        \end{equation*}
        Therefore, the moment condition (\ref{eqn:varadhan_moment}) is satisfied with the identity function \(\varphi(t) = t\). It follows from Lemma \ref{lemma:Varadhan_subset} that
        \begin{align*}
            \liminf_{n \to \infty} \frac{\log Q_n(A_n(c))}{\log n} &= \liminf_{n \to \infty} \frac{1}{\log n} \cdot \log E\left[ \frac{q_n}{p_n}(X_n) \cdot \mathbf{1}_{\left\{A_n(c) \right\}}\right] \\
            &= \liminf_{n \to \infty} \frac{1}{\log n} \cdot \log E\left[ n^{\frac{\log \frac{q_n}{p_n}(X_n)}{\log n}} \cdot \mathbf{1}_{\left\{A_n(c) \right\}}\right] \\
            &= \liminf_{n \to \infty} \frac{1}{\log n} \cdot \log E\left[ n^{\frac{\log \frac{q_n}{p_n}(X_n)}{\log n}} \cdot \mathbf{1}_{\left\{\frac{\log \frac{q_n}{p_n}(X_n)}{\log n} > c \right\}}\right] \\
            &\geq \sup_{t > c} \left\{ t - I(t) \right\}.
        \end{align*}
        Combining these lower bounds with (\ref{eqn:beta_hc_lln_bound}) yields
        \begin{equation*}
            \underline{\beta}^{\HC} \geq \frac{1}{2} + \sup_{t > c} \left\{ t - I(t) \right\} + \frac{1}{2}\min\left(1, \inf_{t \geq c} I(t) \right)
        \end{equation*}
        for all \(c \geq 0\). Maximizing over nonnegative \(c\) yields 
        \begin{equation*}
            \underline{\beta}^{\HC} \geq \frac{1}{2} + \sup_{c \geq 0} \left\{ \sup_{t > c} \left\{t - I(t)\right\} + \frac{1}{2} \min\left(1, \inf_{t \geq c} I(t) \right) \right\},
        \end{equation*}
        as claimed.
    \end{proof}

    \subsection{Proof of Theorem \ref{thm:HC_tight}}
    To prove Theorem \ref{thm:HC_tight}, we use some basic facts about the subdifferential calculus of real-valued convex functions \cite{rockafellarConvexAnalysis1970}. We first state a few definitions from convex analysis \cite{rockafellarConvexAnalysis1970}.

    \begin{definition}
        Let \(f : \R \to [-\infty, \infty]\) be a convex function. We say that \(f\) is \textit{proper} if \(f(x) < \infty\) for some \(x \in \R\) and \(f(x) > -\infty\) for all \(x \in \R\). 
    \end{definition}

    \begin{definition}
        Let \(f : \R \to [-\infty, \infty]\) be a proper convex function. Define the \textit{right derivative} 
        \begin{equation*}
            f_{+}'(x) := \lim_{h \downarrow 0} \frac{f(x+h) - f(x)}{h}
        \end{equation*}
        and \textit{left derivative}
        \begin{equation*}
            f_{-}'(x) := \lim_{h \uparrow 0} \frac{f(x+h) - f(x)}{h}.
        \end{equation*}
    \end{definition}

    The following theorem indeed establishes that the left and right derivatives exist at points where \(f\) is finite when \(f\) is a proper, convex function.

    \begin{theorem}[Theorem 23.1 - \cite{rockafellarConvexAnalysis1970}]\label{thm:rockafellar23.1}
        If \(f : \R \to [-\infty, \infty]\) is a proper convex function and if \(x \in \R\) such that \(f(x)\) finite, then the right and left derivatives \(f_{+}'(x)\) and \(f_{-}'(x)\) exist. Moreover, 
        \begin{align}
            f_{+}'(x) &= \inf_{h > 0} \frac{f(x+h) - f(x)}{h},\\
            f_{-}'(x) &= \sup_{h > 0} \frac{f(x-h) - f(x)}{-h}.
        \end{align}
    \end{theorem}

    The notion of the subgradient and its related properties are useful in our arguments.

    \begin{definition}\label{def:left_right_deriv}
        Let \(f : \R \to [-\infty, \infty]\) be a convex function. A real number \(x^*\) is said to be a \textit{subgradient} of \(f\) at \(x\) if 
        \begin{equation*}
            f(z) \geq f(x) + x^* \cdot (z-x) 
        \end{equation*}
        for all \(z \in \R\). The set of all subgradients of \(f\) at \(x\) is called the \textit{subdifferential of \(f\) at \(x\)}, and is denoted by \(\partial f(x)\). The mapping \(x \mapsto \partial f(x)\) is called the \textit{subdifferential} of \(f\). If \(\partial f(x)\) is not empty, \(f\) is said to be \textit{subdifferentiable} at \(x\).
    \end{definition}

    \begin{theorem}[Theorem 24.1 - \cite{rockafellarConvexAnalysis1970}]\label{thm:rockafellar24.1}
        Let \(f : \R \to [-\infty, \infty]\) be a closed proper convex function. For convenience, extend the right and left derivatives \(f_{+}'\) and \(f_{-}'\) beyond the interval \(D\) on which \(f\) is finite as follows. For points to the right of \(D\), set \(f_+'\) and \(f_-'\) equal to \(\infty\). For points the left of \(D\), set \(f_{+}'\) and \(f_-'\) equal to \(-\infty\). Then \(f_{+}'\) and \(f_{-}'\) are increasing functions on \(\R\), finite on the interior of \(D\), such that
        \begin{equation*}
            f_{+}'(z_1) \leq f_{-}'(x) \leq f_{+}'(x) \leq f_{-}'(z_2)
        \end{equation*}
        when \(z_1 < x < z_2\). Moreover, for every \(x\), 
        \begin{align*}
            \lim_{z \downarrow x} f_{+}'(z) &= f_{+}'(x), \\
            \lim_{z \uparrow x} f_{+}'(z) &= f_{-}'(x), \\
            \lim_{z \downarrow x} f_{-}'(z) &= f_{+}'(x), \\
            \lim_{z \uparrow x} f_{-}'(z) &= f_{-}'(x).
        \end{align*}
    \end{theorem}

    \begin{lemma}[pg. 229 of \cite{rockafellarConvexAnalysis1970}]
        Under the conditions of Theorem \ref{thm:rockafellar24.1}, it follows that \(\partial f(x) = [f_-'(x), f_+'(x)]\) for all \(x \in \R\).
    \end{lemma}

    \begin{lemma}\label{lemma:convex_rate_is_proper}
        Let \(I : \R \to [0, \infty]\) be a good rate function. If \(I\) is convex, then \(I\) is a proper, closed convex function. Recall that a function \(f : \R \to [0, \infty]\) is closed if the sublevel sets \(\{x \in \R : f(x) \leq \alpha\}\) are closed for each \(\alpha \in \R\).
    \end{lemma}
    \begin{proof}
        Since \(I\) is a good rate function, the sublevel sets \(\{x \in \R : I(x) \leq \alpha\}\) are compact. Thus, \(I\) is closed. Since \(I\) is a rate function, it follows that \(\inf_{x \in \R} I(x) = 0\). Since \(I\) is a good rate function, it is further the case that there exists a point \(x \in \R\) such that \(I(x) = 0\). Therefore, \(I\) is finite for some point, and is trivially always greater than \(-\infty\). Hence, \(I\) is proper.
    \end{proof}

    \begin{lemma}[pg. 264 of \cite{rockafellarConvexAnalysis1970}]
        Let \(f : \R \to (-\infty, \infty]\) be a convex function. A point \(m \in \R\) satisfies \(f(m) = \inf_{x \in \R} f(x)\) if and only if \(0 \in \partial f(m)\). 
    \end{lemma}

    With these definitions and results from convex analysis stated, we are ready to begin the proof of Theorem \ref{thm:HC_tight}. First, we state and prove two propositions regarding good, convex rate functions \(I\).

    \begin{proposition}\label{prop:inf_I_t0}
        Let \(I : \R \to [0, \infty]\) be a good rate function. Suppose \(I\) is convex. Let \(I_{-}'\) be the left derivative of \(I\) and extend the domain of definition as in the statement of Theorem \ref{thm:rockafellar24.1}. Define 
        \begin{equation}
            t_0 := \sup\{t \geq 0 : I_{-}'(t) \leq 0\}
        \end{equation}
        and set \(t_0 = 0\) if \(\{t \geq 0 : I_{-}'(t) \leq 0\} = \emptyset\). If \(t_0 < \infty\), then for \(c \geq 0\),
        \begin{equation}\label{eqn:inf_I_t0}
            \inf_{t \geq c} I(t) =
            \begin{cases}
                I(t_0) & \text{if } c < t_0,  \\
                I(c) & \text{if } c \geq t_0. 
            \end{cases}
        \end{equation}    
    \end{proposition}
    \begin{proof}
        As \(I\) is a good convex rate function, it follows by Lemma \ref{lemma:convex_rate_is_proper} that \(I\) is closed and proper. By Theorem \ref{thm:rockafellar23.1}, both the left and right derivatives \(I_{-}'(t)\) and \(I_{+}'(t)\) exist for all \(t \in \R\) such that \(I(t) < \infty\). For convenience, extend the domain of definition of \(I_{-}'\) and \(I_{+}'\) as done in the statement of Theorem \ref{thm:rockafellar24.1}. 

        Since \(I_{-}'\) is an increasing function by Theorem \ref{thm:rockafellar24.1}, it follows that \(\{t \geq 0 : I_{-}'(t) \leq 0\}\) is an interval containing \([0, t_0)\). Moreover, Theorem \ref{thm:rockafellar24.1} implies that 
        \begin{equation*}
            I_{-}'(t_0) = \lim_{s \uparrow t_0} I_{-}(s) \leq 0
        \end{equation*}
        and so 
        \begin{equation*}
            \{t \geq 0 : I_{-}'(t) \leq 0\} = [0, t_0].
        \end{equation*}
        
        Furthermore, consider that \(I_{-}'(s) \geq 0\) for all \(s > t_0\). Therefore, Theorem \ref{thm:rockafellar24.1} implies that
        \begin{equation*}
            I_{+}'(t_0) = \lim_{s \downarrow t_0} I_{-}'(s) \geq 0.
        \end{equation*}
        Therefore, \(0 \in \partial I(t_0) = [I_{-}'(t_0), I_{+}'(t_0)]\), and so \(t_0\) is a minimizer of \(I\). Therefore, if \(c < t_0\), then \(I(t_0) = \inf_{t \geq c} I(t)\). 
        
        If \(c > t_0\), consider that \(I_{-}'(c) \geq 0\). Therefore, for any \(d \geq c\), we have 
        \begin{equation*}
            I(d) \geq I(c) + c^* \cdot (d - c)
        \end{equation*}
        for any subgradient \(c^* \in [I_{-}'(c), I_{+}'(c)]\). This immediately yields 
        \begin{equation*}
            \frac{I(d) - I(c)}{d-c} \geq c^* \geq 0
        \end{equation*}
        since \(d - c \geq 0\) and \(c^* \geq I_{-}'(c) \geq 0\). Therefore, it must be the case that \(I(d) - I(c) \geq 0\), and so \(I(d) \geq I(c)\). Since this holds for all \(d \geq c > t_0\), it follows that \(I\) is monotonically increasing on the interval \((t_0, \infty)\). Since \(t_0\) is a minimizer of \(I\), it immediately follows that \(I\) is monotonically increasing on \([t_0, \infty)\), and so \(\inf_{t \geq c} I(t) = I(c)\) when \(c \geq t_0\). Hence, we have proved (\ref{eqn:inf_I_t0}) as claimed.
    \end{proof}

    \begin{proposition}\label{prop:sup_t_It_t1}
        Let \(I : \R \to [0, \infty]\) be a good rate function. Suppose \(I\) is convex. Let \(D := \{t \in \R : I(t) < \infty\}\). Let \(I_{-}'\) be the left derivative of \(I\) and extend the domain of definition as in the statement of Theorem \ref{thm:rockafellar24.1}. Define 
        \begin{equation}
            t_1 := \sup\{t \geq 0 : I_{-}'(t) \leq 1\}
        \end{equation}
        and set \(t_1 = 0\) if \(\{t \geq 0 : I_{-}'(t) \leq 1\} = \emptyset\). If \(t_1 < \infty\), then for \(c \geq 0\),
        \begin{equation}\label{eqn:sup_t_It_t1}
            \sup_{t > c} \left\{t - I(t)\right\} = 
            \begin{cases}
                t_1 - I(t_1) & \text{if } c < t_1, \\
                c - I(c) & \text{if } c \geq t_1 \text{ and } c \in D^\circ, \\
                -\infty & \text{otherwise}.
            \end{cases}
        \end{equation}
    \end{proposition}
    \begin{proof}
        Note that \(D\) is an interval as \(I\) is convex. Moreover, since \(I\) is convex, it follows that \(I\) is continuous on \(D^\circ\). As \(I\) is a good convex rate function, it follows by Lemma \ref{lemma:convex_rate_is_proper} that \(I\) is closed and proper. By Theorem \ref{thm:rockafellar23.1}, both the left and right derivatives \(I_{-}'(t)\) and \(I_{+}'(t)\) exist for all \(t \in \R\) such that \(I(t) < \infty\). For convenience, extend the domain of definition of \(I_{-}'\) and \(I_{+}'\) as done in the statement of Theorem \ref{thm:rockafellar24.1}. 

        Note that \(t \mapsto I(t) - t\) is a convex function. By Theorem 23.8 in \cite{rockafellarConvexAnalysis1970}, it follows that for all \(t \in \R\) we have \(\partial (I(t) - t) = \partial I(t) + \partial (-t) = \partial I(t) + \{-1\} = [I_{-}'(t)-1, I_{+}'(t)-1]\) where the symbol \(+\) denotes the Minkowski sum (i.e. for two sets \(A\) and \(B\), define \(A + B := \{a + b : a \in A, b \in B\}\)). Now observe that \(\{t \geq 0 : I_{-}'(t) \leq 1\} = \{t \geq 0 : I_{-}'(t) - 1 \leq 0\} = \{t \geq 0 : (I(t) - t)_{-}'(t) \leq 0\}\). Thus, the argument for proving (\ref{eqn:inf_I_t0}) in Proposition \ref{prop:inf_I_t0} can be repeated with some slight modifications to yield
        \begin{equation*}
            \inf_{t \geq c} \left\{I(t) - t \right\} =
            \begin{cases}
                I(t_1) - t_1 & \text{if } c < t_1, \\
                I(c) - c &\text{if } c \geq t_1.
            \end{cases}
        \end{equation*}
        From this, it is clear that if \(c < t_1\), then \(\inf_{t > c} \{I(c) - c \} = I(t_1) - t_1\). If \(c \geq t_1\) and \(c \in D^\circ\), then since \(I\) is continuous on \(D^\circ\) and \(I(t) - t\) monotone increasing on \([t_1, \infty)\), it follows that \(\{\inf_{t > c} I(t) - t\} = I(c)-c\). If \(c \geq t_1\) and \(c \not \in D\), then we trivially have \(\inf_{t > c} \{I(t) - t\} = \infty\) since \(I(t) = \infty\) for all \(t > c\) as \(D\) is an interval. If \(c \geq t_1\) and \(c\) is on the boundary of \(D\), then it also follows that \(\inf_{t > c} \{I(t) - t\} = \infty\). Thus, we've proved 
        \begin{equation*}
            \inf_{t > c} \{I(t) - t\} = 
            \begin{cases}
                I(t_1) - t_1 & \text{if } c < t_1, \\
                I(c) - c &\text{if } c \geq t_1 \text{ and } c \in D^\circ, \\
                \infty &\text{ otherwise},
            \end{cases}
        \end{equation*}
        which immediately yields (\ref{eqn:sup_t_It_t1}).
    \end{proof}

    \begin{proof}[Proof of Theorem \ref{thm:HC_tight}]
        As \(I\) is a good convex rate function, it follows by Lemma \ref{lemma:convex_rate_is_proper} that \(I\) is closed and proper. By Theorem \ref{thm:rockafellar23.1}, both the left and right derivatives \(I_{-}'(t)\) and \(I_{+}'(t)\) exist for all \(t \in \R\) such that \(I(t) < \infty\). For convenience, extend the domain of definition of \(I_{-}'\) and \(I_{+}'\) as done in the statement of Theorem \ref{thm:rockafellar24.1}.
        Note that the conditions of Corollary \ref{corollary:tight_limit} hold by assumption, so we have
        \begin{equation*}
            \underline{\beta}^* = \overline{\beta}^* = \beta^* = \frac{1}{2} + 0 \vee \sup_{t\geq 0} \left\{t - I(t) + \frac{1 \wedge I(t)}{2}\right\}.
        \end{equation*}
        To prove the desired result, it suffices to show that the lower bound (\ref{eqn:beta_HC_lower_bound}) for \(\underline{\beta}^{\HC}\) matches the upper bound \(\underline{\beta}^*\). Recall from Theorem \ref{prop:HC_bound_formula} that 
        \begin{equation*}
            \underline{\beta}^* \geq \underline{\beta}^{\HC} \geq \frac{1}{2} + \sup_{c\geq 0} \left\{ \sup_{t > c} \left\{t - I(t)\right\} + \frac{1 \wedge \inf_{t \geq c} I(t)}{2} \right\}.
        \end{equation*}
        Consider that 
        \begin{align*}
            \sup_{c\geq 0} \left\{ \sup_{t > c} \left\{t - I(t)\right\} + \frac{1 \wedge \inf_{t \geq c} I(t)}{2}\right\} &= \max\left( \sup_{c \in [0, t_0]} E(c), \sup_{c \in [t_0, t_1]} E(c), \sup_{c \in (t_1, \infty)} E(c) \right) 
        \end{align*}
        where 
        \begin{equation*}
            E(c) = \sup_{t > c} \{t - I(t)\} + \frac{1 \wedge \inf_{t \geq c} I(t)}{2}
        \end{equation*}
        for \(c \geq 0\). 
        
        Note that \(I_{-}'\) is an increasing function by Theorem \ref{thm:rockafellar24.1}, and so \(t_0 \leq t_1\). Combining this fact with (\ref{eqn:inf_I_t0}) of Proposition \ref{prop:inf_I_t0} and (\ref{eqn:sup_t_It_t1}) of Proposition \ref{prop:sup_t_It_t1}, it follows that 
        \begin{align*}
            \sup_{c \in [0, t_0]} E(c) &= t_1 - I(t_1) + \frac{1}{2}\min(1, I(t_0)).
        \end{align*}
        Additionally, it follows from (\ref{eqn:inf_I_t0}) of Proposition \ref{prop:inf_I_t0} and (\ref{eqn:sup_t_It_t1}) of Proposition \ref{prop:sup_t_It_t1} that
        \begin{align*}
            \sup_{c \in [t_0, t_1]} E(c) &= \sup_{c \in [t_0, t_1]} \left\{ t_1 - I(t_1) + \frac{1}{2}\min\left(1, I(c)\right) \right\} \\
            &= t_1 - I(t_1) + \frac{1}{2}\min\left(1, I(t_1)\right)
        \end{align*}
        where the final equality follows from the fact that \(I\) is monotonically increasing on the interval \([t_0, \infty)\). Note further that since \(I\) is monotonically increasing on \([t_0, \infty)\), it follows that
        \begin{equation*}\label{eqn:I_smaller_II}
            \sup_{c \in [0, t_0]} E(c) \leq \sup_{c \in [t_0, t_1]} E(c).
        \end{equation*}
        Finally, observe that 
        \begin{align*}
            \sup_{c \in (t_1, \infty)} E(c) &= \max\left(-\infty, \sup_{c \in (t_1, \infty) \cap D^\circ} E(c)\right) \\
            &= \sup_{c \in (t_1, \infty) \cap D^\circ} E(c) \\
            &= \sup_{c \in (t_1, \infty) \cap D^\circ} \left\{c - I(c) + \frac{1}{2}\min\left(1, I(c)\right) \right\}.
        \end{align*}
        Therefore, 
        \begin{equation}\label{eqn:beta_HC_simplified_lbound}
          \underline{\beta}^{\HC} \geq \frac{1}{2} + \left[t_1 - I(t_1) + \frac{1}{2}\min\left(1, I(t_1)\right)\right] \vee \left[\sup_{c \in (t_1, \infty) \cap D^\circ} \left\{c - I(c) + \frac{1}{2}\min\left(1, I(c)\right) \right\}\right].
        \end{equation}
        To show that the lower bound (\ref{eqn:beta_HC_simplified_lbound}) matches \(\underline{\beta}^*\), we now examine \(\underline{\beta}^*\). Observe that since the conditions of Corollary \ref{corollary:tight_limit} hold, it follows that
        \begin{align*}
            \underline{\beta}^* &= \frac{1}{2} + \sup_{t > 0} \left\{t - I(t) + \frac{1 \wedge I(t)}{2}\right\}\\
            &= \frac{1}{2} + \sup_{t \geq 0} \left\{t - I(t) + \frac{1 \wedge I(t)}{2}\right\} \\
            &= \frac{1}{2} + \max\left( \sup_{t \in [0, t_1]} \widetilde{E}(t), \sup_{t \in (t_1, \infty)} \widetilde{E}(t) \right)
        \end{align*}
        where 
        \begin{equation*}
            \widetilde{E}(t) = t - I(t) + \frac{1 \wedge I(t)}{2}.
        \end{equation*}
        From earlier discussion it is clear that 
        \begin{equation*}
            \sup_{t \in [0, t_1]} \widetilde{E}(t) = t_1 - I(t_1) + \frac{1 \wedge I(t_1)}{2}
        \end{equation*}
        and 
        \begin{align*}
            \sup_{t \in (t_1, \infty)} \widetilde{E}(t) &= \max\left(-\infty, \sup_{t \in (t_1, \infty) \cap D} \widetilde{E}(t) \right) \\
            &= \sup_{t \in (t_1, \infty) \cap D} \widetilde{E}(t) \\
            &= \sup_{t \in (t_1, \infty) \cap D^\circ} \widetilde{E}(t)
        \end{align*}
        as it is assumed that \(I\) is left/right-continuous at any contained endpoints of \(D\). Hence, we have established that 
        \begin{align*}
            \underline{\beta}^* = \frac{1}{2} + \left[t_1 - I(t_1) + \frac{1 \wedge I(t_1)}{2} \right] \vee \left[\sup_{t \in (t_1, \infty) \cap D^\circ} \left\{t - I(t) + \frac{1 \wedge I(t)}{2} \right\} \right]
        \end{align*}
        and so \(\underline{\beta}^*\) is exactly the right hand side of (\ref{eqn:beta_HC_simplified_lbound}). Thus, \(\underline{\beta}^* = \underline{\beta}^{\HC}\) as claimed.
    \end{proof}

    \section{Acknowledgements}
    I thank Chao Gao for suggesting this problem, providing detailed and helpful comments on multiple drafts, and offering substantial encouragement. 

    \bibliographystyle{plain}
    \bibliography{sparse_mixture_detection}
\end{document}